\documentclass[11pt,letterpaper,reqno]{amsart}

\usepackage{amsmath,amsthm,amssymb,amsfonts,epsfig,color,graphicx,enumerate,enumitem,accents}
\usepackage{rotating,multirow,color,enumerate,url,mathrsfs,mathrsfs}
\usepackage{arydshln}
\usepackage{bbm}
\usepackage{cases}
\usepackage{esint}
\usepackage{gensymb}
\usepackage{epstopdf}
\usepackage{comment}
\usepackage{leftidx}
\usepackage{graphicx}
\usepackage{todonotes}
\usepackage[normalem]{ulem}

\usepackage{mathtools}
\usepackage{enumerate}
\usepackage{subfig}
\usepackage{amsaddr}
\usepackage[rightcaption]{sidecap}

\usepackage[utf8]{inputenc}
\usepackage[T1]{fontenc}

\usepackage{scalerel}
\usepackage{bbm}
\usepackage{tikz}
\usepackage{IEEEtrantools}
\usepackage{thmtools}
\usepackage{thm-restate}

\newtheorem{theorem}{Theorem}[section]
\newtheorem{corollary}[theorem]{Corollary}
\newtheorem{lemma}[theorem]{Lemma}
\newtheorem{proposition}[theorem]{Proposition}

\newtheorem{prob}[theorem]{Problem}

\theoremstyle{definition}
\newtheorem{definition}[theorem]{Definition}
\newtheorem{remark}[theorem]{Remark}
\newtheorem{example}[theorem]{Example}

\numberwithin{equation}{section}

\newcommand{\med}{\medskip\noindent}
\newcommand{\e}{\varepsilon}

\def\f{\varphi}

\newcommand{\R}{\mathbb{R}}

\newcommand{\Rd}{\R^d}

\newcommand{\ident}{{\mathrm{id}}}
\newcommand{\Ident}{{\mathrm{Id}}}

\newcommand{\ov}{\overline}

\newcommand{\A}{\mathcal{A}}
\newcommand{\B}{\mathcal{B}}
\newcommand{\K}{\mathcal{K}}
\newcommand{\LL}{\mathcal{L}}
\newcommand{\JJ}{\mathcal{J}}
\newcommand{\FF}{\mathcal{F}}
\newcommand{\var}{{\rm{var}}}

\newcommand{\weakstar}{{\overset{\ast}{\rightharpoonup}}}

\newcommand \res{\mathop{\hbox{\vrule height 7pt width .5pt depth 0pt
			\vrule height .5pt width 6pt depth 0pt}}\nolimits}

\newcommand{\opnorm}[1]{{\left\vert\kern-0.25ex\left\vert\kern-0.25ex\left\vert #1 
		\right\vert\kern-0.25ex\right\vert\kern-0.25ex\right\vert}}

\def\O{\Omega}
\newcommand{\Ob}{{\overline{\Omega}}}
\newcommand{\bO}{{\partial\Omega}}

\newcommand{\argu}{{\,\cdot\,}}

\newcommand{\Mes}{{\mathcal{M}}}

\newcommand{\Sddp}{{\mathcal{S}^{d\times d}_+}}

\newcommand{\dive}{{\mathrm{div}}}

\newcommand{\tr}{{\mathrm{Tr}}}

\newcommand{\PP}{\mathcal{P}}

\newcommand{\pairing}[1]{{\left \langle #1 \right \rangle}}

\newcommand{\abs}[1]{{\left \lvert #1 \right \rvert}}
\newcommand{\Lip}{{\mathrm{lip}}}

\newcommand{\eps}{\varepsilon}

\newcommand{\I}{\mathcal{I}}
\newcommand{\G}{\mathcal{G}}
\newcommand{\D}{\mathcal{D}}
\newcommand{\C}{{C}}

\newcommand{\V}{{\mathcal{V}}}

\newcommand{\Sf}{{{S}_f}}

\newcommand{\Ha}{\mathcal{H}}

\newcommand{\Sdd}{{\mathcal{S}^{d \times d}}}

\newcommand{\mres}{\mathbin{\vrule height 1.6ex depth 0pt width
		0.13ex\vrule height 0.13ex depth 0pt width 1.3ex}}

\newcommand{\one}{{{\bf 1}
		\kern-0,28em \rm l}}

\DeclareMathOperator{\spt}{sp}

\begin{document}

	\title{Kantorovich-Rubinstein duality theory for the Hessian}

	

	\author{Karol Bo{\l}botowski \and Guy Bouchitt\'{e}}


	\date{\today}
	
	\begin{abstract}  
		
    The classical Kantorovich-Rubinstein duality theorem establishes a significant connection between Monge optimal transport and maximization of a linear form on the set of 1-Lipschitz functions. This result has been widely used in various research areas. In particular, it unlocks the optimal transport methods in some of the optimal design problems. This paper puts forth a similar theory when the linear form is maximized over $C^{1,1}$ functions whose Hessian lies between minus and plus identity matrix. The problem will be identified as the dual of a specific optimal transport formulation that involves three-point plans. The first two marginals are fixed, while the third must dominate the other two in the sense of convex order. The existence of optimal plans allows to express solutions of the underlying Beckmann problem as a combination of rank-one tensor measures supported on a graph. In the context of two-dimensional mechanics, this graph encodes the optimal configuration of a grillage that transfers a given load system.
\end{abstract}

\maketitle

\noindent \textbf{Keywords:}  {Hessian-constrained problem, Monge optimal transport,  tensor-valued measures,
 second-order Beckmann problem, convex order, stochastic dominance, optimal grillage.}

\noindent \textbf{2020 Mathematics Subject Classification:}  49J45, 49K20, 28A50, 60E15, 74P05
	
	\dedicatory{}
	
	
	\maketitle

	\vskip1cm
	
	
\section{introduction}

The classical Kantorovich-Rubinstein duality theorem plays a fundamental role in the Monge optimal transport theory \cite{santambrogio2015,villani2003}.
In the Euclidean framework, it states that,  
for given probability measures $\mu, \nu$ on $\R^d$ with finite first-order moments, the $L^1$ Monge-Kantorovich distance,
$$ W_1(\mu,\nu) = \inf \left\{\iint |x-y|\, \gamma(dxdy) \ :\  \gamma\in \Gamma(\mu,\nu) \right\}, $$
coincides with the maximum in the following  linear programming problem,
\begin{equation}\label{firstorder} {\I_1(f)}= \sup\left\{  {\int u \, df }\ :\  u\in C^{0,1}(\R^d),\ \  \Lip(u)\le 1\right\} ,
\end{equation}
{for the signed measure $f=\nu -\mu$.} Above  $\Gamma(\mu,\nu)$ stands for the set of probability measures on
$\R^d\times \R^d$ having the first and second marginal $\mu$ and $\nu$, respectively. 

The equality $W_1(\mu,\nu)= {\I_1(f)}$ is a key ingredient for deriving a
  PDE approach to the optimal transport problem,  see e.g. \cite{BBS1997,evans1999}.  It also allows to interpret the Monge distance as the total variation of a vector measure $\sigma \in \Mes(\R^d;\R^d)$ which solves the so-called
Beckmann problem,
\begin{equation}
	\label{first-order_Beckmann}
	\min \left\{ \int |\sigma| \ :\  {-\dive\, \sigma= f}  \ \ \text{in } \D'(\Rd) \right\}.
\end{equation} 
The geometric insights into the optimal transport interpretation are
 very meaningful. Using the notion of transport rays or, more generally, geodesics, it is possible to  recast the solutions to the Beckmann problem from the optimal transport plans $\gamma$ via the decomposition formula,
\begin{equation} \label{truss-measure} \sigma = \iint \lambda^{x,y} \, \gamma(dxdy), \qquad  \lambda^{x,y}:=\frac{y-x}{\abs{y-x}}  \, \Ha^1 \mres [x,y].\end{equation}
 In particular, it follows that every optimal measure $\sigma$ is supported on the convex (geodesic) hull of the support of {$f$}.
 
In the late 1990s, a bridge between Monge optimal transport and optimal design was proposed in \cite{BBS1997}. It was found that an optimal measure $\sigma$ for \eqref{first-order_Beckmann} represents the heat flow in a conductor to be optimally designed for a given source $f$. The optimal transport approach can be applied accordingly to derive the best distribution of the conductive material. In particular, this approach covers the case of concentrated source terms $f$, which cannot be treated  by the classical PDE methods.

Later on, the paper \cite{bouchitte2001} has shown that a larger class of optimal design problems can be tackled by solving
a Beckmann-type problem, see also the recent work \cite{bolbotowski2022}. This includes compliance minimization of elastic bodies where $\sigma$ is tensor-valued, while the potential $u$ is vectorial, and they represent stress and displacement, respectively.
 However, the optimal transport strategy does not automatically extend to these cases. Note that  similar issues concern the auction design problem in mathematical finance \cite{kolesnikov2022}. The single-bidder formulation, also known as the monopolist problem, admits both Beckmann and optimal transport reformulations. For the extension to multiple bidders, only the Beckmann interpretation remains.

In view of the important underlying applications, there is an urge to study possible extensions of the Kantorovich-Rubinstein duality principle. The question could be formulated as follows.
 Let $A$ be a linear differential operator on smooth vector-valued functions {$u:\O \subset\R^d \to \R^n$ such that $Au:\Omega \to \R^n \times \R^d$}, and for $\varrho$ take a semi-norm on real $n \times d$ matrices. Is there an optimal transport formulation that we can employ to address the following maximization problem,
\begin{equation} \label{LCP}
\sup\Big\{ \pairing{u,f} :\ u\in C^\infty(\R^d; \R^n),\ \ \varrho(Au)\le 1 \ \text{in $\O$}\Big\} ,
\end{equation}
 where $\O$ is a domain in $\R^d$, and $f$ is a suitable source term supported  on $\Ob$\nolinebreak\ ?  By standard  convex duality, the above supremum  can be written as a minimum in the following Beckmann-type problem,
 \begin{equation}\label{Beckman}
\min \left\{ \int \varrho^0(\sigma) \ :\ A^* \sigma= f \ \text{ in $\big(\D'(\Rd)\big)^n$}  ,\ \spt{\sigma}\subset\Ob\right\},
\end{equation}
 where $\varrho^0$ {is the polar} of $\varrho$ given by,
 \begin{equation}\label{dualnorm}
\varrho^0(S) := \sup \big\{ \pairing{Q,S}  \, :\ \varrho(Q)\le 1 \big\}.
\end{equation}
In the classical Rubinstein-Kantorovich framework, $f$ is a scalar measure ($n=1$), $\varrho$ is the Euclidean norm, and $A$ is the gradient operator. In order that the supremum  \eqref{LCP} is finite, $f$ must be  \textit{balanced}, that is $\mu=f_+$ and $\nu=f_-$ must have the same mass. Note that if $\O$ is not convex, then the Euclidean distance appearing in the definition of $W_1(\mu,\nu)$ should be replaced by the geodesic distance induced by $\O$. For a detailed study see \cite{bouchitte2001} where a number of explicit examples is given. 

When $A$ is no longer the gradient operator, there are very few results regarding a possible optimal transport approach.
 In the recent work \cite{bolbou2022}, the present authors put forward a formulation where
  the Monge-Kantorovich distance emerges and is maximized with respect to a suitable class of metrics $d(x,y)$ on $\Omega$.
Therein, the potential $u$ becomes a pair $(v,w):\Omega \to \Rd\times \R$, and it must meet the constraint $e(v)+\frac{1}{2} \nabla w \otimes \nabla w \leq \mathrm{Id}$, where $e(v)$ is the symmetrized gradient.  The non-linear operator on the left hand side defines the strain tensor in the F\"{o}ppl's membrane model \cite{conti2006}, rendering the formulation in \cite{bolbou2022} an optimal membrane problem. Despite the non-linearity, the problem admits the form \eqref{Beckman} upon the right choice of $A$ and $\varrho$, see \cite{bolbou2022} for more details.
 Prior to the latter work, an attempt to treat the case where $A$ is simply the symmetrized gradient (i.e. $Au= e(u)$ for $u:\O\to \R^d$) was made in \cite{bouchitte2008} where \eqref{Beckman} is nothing else but the famous Michell problem. 
However, to our knowledge, the bridge with optimal transport in this case  has not yet been established, and  challenging open problems remain.

A different direction was taken in the work \cite{hanin1994}, later also developed in, e.g., \cite{hanin1997,olubummo2004}. With $A$ being the $k$-th derivative operator and $\varrho$ being the operator norm, the authors proposed a \textit{transshipment} formulation as the dual problem to \eqref{LCP}.  In contrast to the transport formulations, it fixes the difference of the marginals as $\nu-\mu$ rather than imposing them separately. As a result, transshipment formulations lack the existence results in general. Transshipment problems have also been analysed in the setting of vector valued plans, see \cite{ciosmak2021}.

\medskip

 The aim of the present paper is to provide an optimal transport approach in the case of  the Hessian operator  $Au = \nabla^2 u$ and with  $\varrho$ being the operator norm. We will limit ourselves to the case when $\Omega = \Rd$, and we will assume that $f$ is a measure, more accurately
 $f= \nu-\mu$
 for two probabilities $\mu,\nu$. The special choice of $\varrho$ makes it possible to rewrite \eqref{LCP} as the second-order counterpart of \eqref{firstorder},
\begin{equation}\label{secondorder} \I(f) := \sup\left\{  \int u \, d\nu -   \int u \, d\mu \ :\  u\in C^{1,1}(\R^d),\ \  \Lip(\nabla u)\le 1\right\}.
\end{equation}
The supremum  $\I(f)$ is finite and is attained if and only if $\mu$ and $ \nu$  share the barycentre $[\mu]=[\nu]$ while exhibiting finite second-order moments: $\mu, \nu \in \PP_2(\Rd)$.  This way $\I(f)$ defines a distance between such  $\mu$ and $\nu$.   We should point out that it is a particular case of the family of  \textit{ideal metrics}, introduced years ago by V.M. Zolotarev \cite{zolotarev1977} with the aim of studying continuity and stability of stochastic models in probability theory.

Under the foregoing assumptions on the data $\mu,\nu$, we will see that the classical duality theory leads to a well-posed second-order Beckmann formulation,
   \begin{equation}\label{stress} \I'(f):= \min
\biggl\{ \int \varrho^0(\sigma) \ : \ \sigma \in \Mes(\R^d;\Sdd), \ \ \dive^2 \sigma = \nu-\mu \ \ \text{in } \D'(\Rd) \biggr\}
\end{equation} 
with the zero-gap equality $\I(f)=\I'(f)$. 
 Here, $\varrho^0$ is the Schatten norm on symmetric matrices $S\in \Sdd$ given by $\varrho^0(S) = \sum_{i=1}^d |\lambda_i(S)|$, and
   $\int \varrho^0(\sigma)$ is intended in the sense of convex one-homogeneous functionals on measures 
 \cite{Goffman}.
Through the classical methods, the optimality  conditions involving pairs $(u,\sigma)$ can be derived,  even in the case of singular measures $\sigma$ and for  general semi-norms $\varrho$. For a detailed study we refer to e.g.  \cite{bouchitte2007} where  optimal design problems for plates
 are considered.  
   
At the core of our work lies a connection between the pair \eqref{secondorder}, \eqref{stress} and a newly proposed\textit{ three-marginal optimal transport problem.} Let us introduce the cost function defined for each triple $(x,y,z) \in (\Rd)^3$ by,
\begin{equation}\label{cost}
	c(x,y,z):= \frac{1}{2} \Bigl( \abs{z-x}^2 + \abs{z-y}^2 \Bigr).
\end{equation}
We are looking for probabilities $\pi \in \PP_2((\Rd)^3)$ (three-marginal transport plans) that solve,
\begin{equation}\label{OT3} \JJ(\mu,\nu) :=  \inf
\biggl\{ \iiint c(x,y,z) \, \pi(dxdydz)  \ :   \ \pi\in \Sigma(\mu,\nu)\biggr\}.
\end{equation} 
The set $\Sigma(\mu,\nu) \subset \PP_2((\Rd)^3)$ consists of 3-plans $\pi$ whose first and second marginal is
$\mu$ and $\nu$, respectively, and which satisfy the following  equations,
\begin{equation}\label{equilibrium}
	\iiint \! \pairing{z-x, \Phi(x)} \, \pi(dxdydz)  = \iiint \! \pairing{z-y, \Psi(y)} \, \pi(dxdydz) =0,
\end{equation}
for any smooth test functions $\Phi,\Psi: \Rd \to \Rd$.
It will unravel that these relations have a natural  interpretation in probability theory (via martingale plans and convex order) as well as
in structural mechanics (moment equilibrium at junctions of a grillage). 
The main result of the paper reads as follows:

\begin{theorem}\label{main-intro} Take $\mu, \nu \in\PP_2(\Rd)$ sharing the barycentre $[\mu]=[\nu]$, and  let $\JJ(\mu,\nu)$ be defined by \eqref{OT3}.
For  $f=\nu-\mu$ let the value $\I(f)$ be given by \eqref{secondorder}. Then,

\med (i) it holds that,
\begin{equation}
    \label{central_eq}
    \I(f) = \JJ(\mu,\nu),
\end{equation}
while  there exist optimal pairs $(u,\pi)$ solving  \eqref{secondorder} and  \eqref{OT3}, respectively;

\med (ii) an admissible pair $(u,\pi)$ is optimal if and only if the following three-point equality is satisfied  $\pi$-a.e.,
with $c$ defined by \eqref{cost},
\begin{align}\label{3-eq}
[u(y) + \pairing{\nabla u(y), z-y}] -[u(x) + \pairing{\nabla u(x), z-x} ] 
= c(&x,y,z) \\
\nonumber
& \text{for $\pi$-a.e. $(x,y,z)$.}
\end{align}

\end{theorem}

It is worth mentioning  that the equality \eqref{3-eq} is in close relation with the admissibility of $u$ in \eqref{secondorder}. Indeed, according to  \cite{legruyer2009}, the condition $\Lip(\nabla u)\le 1$ is equivalent to the  existence of a continuous vector function $U:\R^d\to \R^d$
such that,
\begin{align}\label{3-jet}
 [u(y) + \pairing{{U(y)}, z-y}]  - [u(x) + \pairing{{ U(x)}, z-x} ]
\le c(&x,y,z) \\
\nonumber
 &\forall\, (x,y, z) \in (\Rd)^3 .
\end{align}
In addition,  the inequality \eqref{3-jet} can be satisfied only for {$U= \nabla u$}. 

With the equality \eqref{central_eq} at hand,  we are in a position to propose the tensor counterpart of the decomposition \eqref{truss-measure} that was useful in the first-order gradient case. We show that optimal  measures for  the second-order Beckmann problem \eqref{stress}
can be decomposed to rank-one tensor measures supported on polygonal lines. More precisely, for any triple $(x,y,z)$  let us define  the following measure valued in the space of symmetric matrices, i.e. an element of  $\Mes(\R^d;\Sdd)$,
\begin{equation}\label{si-xyz}
	\sigma^{x,y,z}(d\xi) := \abs{\xi-z}\,\Big(\sigma^{z,x}(d\xi) - \sigma^{z,y}(d\xi)\Big) ,
\end{equation}
where we denote,
\begin{equation}
 \sigma^{a,b} = \frac{b-a}{|b-a|} \otimes  \frac{b-a}{|b-a|} \, \Ha^1 \mres [a,b] .
\end{equation}
The measure $\sigma^{x,y,z}$ (see Fig. \ref{fig:sigmaxyz}) is supported on the set  $[x,z] \cup [z,y]$. Its second distributional divergence is $\dive^2 \sigma^{x,y,z} =f^{x,y,z}$, where 
\begin{equation}\label{f-xyz}
f^{x,y,z}:= \delta_y - \delta_x - \dive\big((z-y)\,\delta_y - (z-x)\,\delta_x \big)
\end{equation}
 includes a first-order term. The key observation is that  $\sigma^{x,y,z}$ solves \eqref{stress} 
 for $f=f^{x,y,z}$ with the equality $\I(f^{x,y,z}) = c(x,y,z)$  provided that  $z$ belongs to the ball $B\big(\frac{x+y}{2}, \frac{|x-y|}{2}\big)$ (see Proposition \ref{optibasic}). 
 Accordingly,  our strategy for solving \eqref{stress} for $f=\nu-\mu$\, is to search for optimal tensor measures
  $\sigma$ in the form,
\begin{equation*}\label{grid-measure}
\sigma = \iiint \sigma^{x,y,z} \, \pi(dxdydz),
\end{equation*}
where $\pi$ is a suitable  three-marginal plan, see the convention \eqref{eq:convention} below. Satisfying the admissibility condition $\pi \in \Sigma(\mu,\nu)$ in the three-marginal optimal transport formulation \eqref{OT3} guarantees that $\dive^2 \sigma = \iiint f^{x,y,z} d\pi = \nu-\mu$.

\begin{corollary}\label{coro1-intro}\  Let $\ov \pi$ be an optimal 3-plan for \eqref{OT3}. Then,

\med (i) the tensor measure $\ov\sigma = \iiint \sigma^{x,y,z} \, \ov\pi(dxdydz)$ solves the second-order Beckmann problem \eqref{stress} for $f=\nu-\mu$;

\med (ii) let {$\ov\gamma:= \Pi_{1,2}^\#(\ov{\pi})$} be the marginal of $\ov\pi$ with respect to the first two variables, and let $\ov u$
be any solution of \eqref{secondorder}. Then, for any {$\ov\pi$-integrable} test function $\varphi: (\Rd)^3\to \R$,  we have the disintegration formula,
$$  \iiint \varphi \,  d\ov\pi =  \iint \varphi\big(x,y, z_{\ov u}(x,y)\big) \, \ov\gamma(dxdy),$$
where
\begin{equation}\label{zbar}
  z_{\ov u}(x,y) := \frac{x+y}{2} + \frac{\nabla \ov u(y)-\nabla \ov u(x)}{2}.
  \end{equation}
  As a result, the optimal measure $\ov\sigma$ is supported on the closed subset,
  \begin{equation}\label{spt-sigma}
\B(\spt \mu, \spt \nu) := \bigcup \left\{ B\Big(\frac{x+y}{2}, \frac{|x-y|}{2}\Big)\,:\, (x,y)\in \spt\mu \times \spt \nu\right\},
  \end{equation}
  where $B(x_0,r)$ is the closed ball of radius $r>0$ and centred at $x_0\in \Rd$.
  
\end{corollary}

Let us recall that, in the first-order case \eqref{first-order_Beckmann}, a geometric bound on the support of any optimal measure can be expressed as in  \eqref{spt-sigma} if we replace the ball on the right-hand side by the line segment $[x,y]$. Contrarily, in the Hessian case, a larger set is required  to cover the support of the possible optimal measures. This will be confirmed on a number of examples.
 Thus, we refute the conjecture in  \cite{bouchitte2007} where, assuming mild conditions on the norm $\varrho^0$ entering \eqref{stress}, it was suggested that optimal measures are supported on the convex hull of the source $f$.
 
 The proof  of the central equality \eqref{central_eq}, that we will present here,  passes through an unexpected link between our three-marginal optimal transport formulation \eqref{OT3} and optimization under the convex order dominance constraints. This connection is based on the observation that  the existence of a 3-plan $\pi\in \Sigma(\mu,\nu)$ which admits  $\rho$ as the third marginal
is equivalent to the convex order conditions $\rho \succeq_c \mu,\ \rho \succeq_c \nu$, that is,
\begin{equation}
	\int \varphi \,d\rho\ \geq \ \max \left\{ \int \varphi \,d\mu , \int\varphi \, d\nu\right\} \quad \text{for all convex \ $\varphi:\R^d \to \R$}.
\end{equation}
 
\begin{theorem}\label{thm2-intro}\ 
Let $\mu, \nu \in \PP_2(\R^d)$ be  probability measures satisfying $[\mu]=[\nu]$, and set
\begin{equation} \label{min-var} 
\V(\mu,\nu):= \inf \Big\{ \var (\rho) \ :\  \rho\in \PP_2(\R^d), \ \ 
\rho \succeq_c \mu  ,\ \ \rho \succeq_c \nu \Big\}  , \end{equation}
where $\var (\rho)$ is the variance of  the measure $\rho$. 

\med (i) With $f = \nu-\mu$, the following equalities hold true,
\begin{align*}
	 \I(f) = \V(\mu,\nu)  - \frac{ \var (\mu) + \var (\nu)}{2} =  \JJ(\mu,\nu).
\end{align*}
Moreover, an admissible $3$-plan $\pi\in \Sigma(\mu,\nu)$ is optimal  for  \eqref{OT3} if and only if its third marginal $\rho$ 
is a minimizer in \eqref{min-var}. 

\med (ii) The infimum in \eqref{min-var} is achieved. Moreover,  to any minimal $\rho$  we can associate at least one  3-plan  $\pi$ that solves \eqref{OT3} and  whose three  marginals are  $\mu,\nu$, and $\rho$, subsequently.
\end{theorem}

Here, several comments are in order. When proving the relation between $\I(f)$ and $\V(\mu,\nu)$, we shall
 pass through a duality result (Proposition \ref{2conv}) which allows to identify $\V(\mu,\nu)$ as  the supremum  in the following auxiliary problem,
\begin{align*}
	 \V'(\mu,\nu) :=  \sup \bigg\{ \int \f\, d\mu + \int \psi\, d\nu \, :\,  \f ,&\, \psi \ \text{are convex  $C^{1,1}$ functions} , \\
	  &\f(z)+\psi(z)\le |z-z_0|^2 \ \ \ \forall \, z \in\Rd\bigg\},
\end{align*}
where $z_0=[\mu]=[\nu]$.
Then, the crucial equality $ \I(f)  = \V'(\mu,\nu) - \frac{1}{2} \big(\var(\mu) + \var(\nu) \big)$ will be derived directly in Section \ref{IversusW}  by exploiting a smoothing effect due to convexification of semi-concave functions \cite{cependello1998,dephilippis2015, azagra2018}, see
 Lemma \ref{fundamental} where we restate this property.

On another note, the construction of an optimal 3-plan $\ov\pi$ announced in the assertion (ii) will use two martingale  transports \cite{beiglbock2013}: one between $\mu$ and $\rho$, and the other between $\nu$ and $\rho$. Their existence follows from the classical theorem of Strassen.  An optimal $\ov\pi$ can be then constructed via a gluing  argument, see the statement and the proof of Lemma  \ref{equival}. In general, such $\ov\pi$ is not unique. 
 
 \medskip

Let us provide more bibliographical context. In recent decades, there has been an increase in new variants of optimal transport problem. To an extent, our formulation \eqref{OT3} bears  resemblance to a number of them, yet it is not a special case of any. For instance, the search for multi-marginal transport plans is a classical topic,  see \cite{gangbo1998optimal} where a quadratic cost similar to $c(x,y,z)$ was considered. The difference here lies in the freedom of the third marginal and, more importantly, in the extra constraints \eqref{equilibrium}. These prevent the existence of optimal plans that are induced by transport maps, which was one of the main results of \cite{gangbo1998optimal}. From a different perspective, it will appear that \eqref{equilibrium} encodes two martingale-type conditions, naturally placing \eqref{OT3} next to the martingale optimal transport introduced in \cite{beiglbock2013} with financial application in mind, see also \cite{Ghoussoub,henry-labordere2017}. In the latter context, the data $\mu,\nu$ must be in convex order, which makes the formulation dissymmetric with respect to the marginals. By contrast, our problem  \eqref{OT3} is  symmetric thanks to the presence of the free marginal.
In particular,  its well posedness requires merely that the barycentres coincide: $[\mu]=[\nu]$. This symmetry also positions  \eqref{OT3} outside the broad scope of weak optimal transport theory, which was initiated in \cite{gozlan,ABC}, see also \cite{gozlan2020,backhoff-veraguas2022}.

On the other hand, the problem \eqref{min-var} does fall within a larger class of stochastic optimization problems under dominance constraints, for which there exist many applications in mathematical finance, statistical decision theory, or economics \cite{Dentcheva2003,MulSca2006,Wiesel2023}.
In the context of optimal transport,  \eqref{min-var} and its natural extension to multiple data $\mu_1,\ldots, \mu_n$, could be compared to the Wasserstein barycentre problem \cite{agueh2011barycenters}. Both problems give rise to a probability on $\Rd$ that minimizes a specific quadratic cost. Nonetheless, there are some fundamental differences. We will demonstrate that the convex order constraints often cause the support of the optimal dominant $\rho$ to exceed the convex hull of the supports of given measures, which does not occur for the Wasserstein barycentre. Furthermore, Wasserstein barycentres are known to inherit  the $L^\infty$ regularity of the data, whilst it will appear that the solutions to \eqref{min-var} tend to charge lower dimensional sets.

 \medskip

\begin{figure}[h]
	\centering
	\includegraphics*[trim={0cm 0cm -0cm -0cm},clip,width=0.4\textwidth]{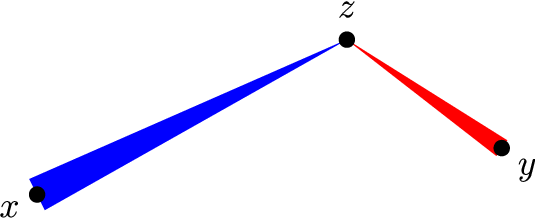}
	\caption{The tensor measure $\sigma^{x,y,z}$; the density with respect to $\Ha^1 \mres([x,z] \cup [z,y])$ is illustrated.  Blue and red indicate the positive and the negative part, respectively.}
	\label{fig:sigmaxyz}   
\end{figure}

We close the introduction with a comment on the close relation between the results presented herein and an optimal design problem in mechanics when $d=2$. Any measure $\sigma \in \Mes(\Rd;\mathcal{S}^{2\times 2})$ that satisfies the equation $\dive^2\sigma=f$ represents a \textit{bending moment tensor} in a plate that is subject to a load $f$. If the measure $\sigma$ is of the form $\iiint \sigma^{x,y,z} d\pi$, then we speak of a \textit{grillage} -- a particular plate that decomposes to straight bars. The bars exhibit linearly varying rank-one bending moments, see Fig. \ref{fig:sigmaxyz} depicting the basic two-bar measure $\sigma^{x,y,z}$. A natural issue studied in the literature  \cite{rozvany1972a,prager1977,bolbotowski2018a} consists in finding an optimal configuration of the grillage, i.e. a coupling $\ov{\pi}$ that minimizes a certain total energy functional. To date, however, the existence result was not available, and neither were the criteria for the finite support of potential solutions $\ov{\pi}$, which corresponds to practical designs in the form of finite systems of bars. 

In this work we show that, when the load $f$ is a measure, an optimal grillage can be recast by solving the new three-marginal optimal transport formulation \eqref{OT3}. As a byproduct, we get the bound $\B(\spt f_+, \spt f_-)$ on the support of the associated tensor measure $\ov{\sigma}$.  Moreover, a finitely supported optimal 3-plan $\ov\pi$ can be selected provided that $f$ is also finitely supported. In this case, the induced measure $\ov{\sigma}$ concentrates on a graph.  It should be noted that, despite a similarity to the optimal grillage problem, there is no such OT reformulation for the more popular \textit{optimal truss problem}. In fact, it has been known for 120 years that optimal trusses do not exist even for the simplest load data \cite{michell1904}. In this case, a relaxation in the form of the famous  Michell formulation \cite{bouchitte2008,lewinski2019} is essential. On top of that, a geometric bound on the support of its solutions is still pending. 

Finally, we stress that our results concerning optimal grillages do not immediately extend to the case of a source $f$ containing a first-order distribution term, or to the case when the support of the induced stress $\ov\sigma$ is confined within a given domain $\Omega\subset\Rd$. Such extensions are beyond the scope of this paper and are worthy of future study.

\bigskip
	
	The paper is organized as follows.
	In the preliminary Section \ref{Prelim}  we adapt the classical duality  theory to show the no-gap equality between \eqref{secondorder} and \eqref{stress}, as well as the existence of an optimal pair $(u,\sigma)$. Besides, in view of the forthcoming connection with the stochastic optimization,  we give a short background on convex order and its relation with martingale transport. 
	The Section \ref{proofs} presents the proofs of Theorem \ref{main-intro}, Corollary \ref{coro1-intro},
	and Theorem \ref{thm2-intro}.
	 In Section \ref{examples}, we give a series of examples where optimal configurations are determined  explicitly. 	 
	 The Section \ref{meca} is devoted to the underlying 2D optimal design formulation. Numerical examples of optimal grillages are given and discussed along with the related open questions. Readers who are less interested in the applications should feel free to skip this final  section.

\subsection*{Acknowledgements}  
The first author would like to thank the Laboratoire IMATH, Université de Toulon  for hosting his visit in June 2022. 
The second author was partially supported by the French project ANR-23-CE40-0017.   
Both authors are grateful to the Lagrange Mathematics and Computing Research Center in Paris for facilitating fruitful discussions with Guillaume Carlier, Quentin Mérigot, and Filippo Santambrogio. The authors are thankful for their remarks which have influenced this paper. 
    
\subsection*{Notations} Throughout the paper we will use the following notations.

\begin{itemize}[leftmargin=1.2\parindent]
\item  The Euclidean norm of $z\in\R^d$ is denoted by $|z|$.

\item By $\Sdd$ we shall denote the space of $d\times d$ symmetric matrices, while $\Sddp$ will be its subset whose elements are positive semi-definite.
Given $A,B \in \Sdd$, we will write $A\le B$ if $B-A\in \Sddp$. Moreover, $\tr A$ stands for the trace of $A$, while $\Ident$ is the identity matrix.

\item  For natural $k\in \mathbbm{N} \cup\{+\infty\}$, $C^{k}(\Rd)$ is the  spaces of functions on $\Rd$ that are continuously differentiable up to order $k$, while $\nabla u$ and $\nabla^2 u$ are the gradient and the Hessian of a function $u\in C^1(\Rd)$ and $u\in C^2(\Rd)$, respectively. Moreover, $C_0(\Rd) \subset C^0(\Rd)$ denotes the subset of continuous functions that vanish at infinity.

\item  $\D(\Rd)$ denotes the  space of $C^\infty$ functions that are compactly supported, and $\D'(\Rd)$ is the space of distributions on $\Rd$ (the dual of $\D(\Rd)$).

\item  For a function $v:\Rd \to \R^n$, $\Lip(v)$ stands for the Lipschitz constant equal to $\sup_{x \neq y} \frac{\abs{v(x)-v(y)}}{\abs{x-y}}$.

\item  By $C^{0,1}(\Rd)$ (resp. $C^{1,1}(\Rd)$) we understand the Banach space of these functions $u \in C^0(\Rd)$  (resp.  $u\in C^1(\Rd)$) for which $\Lip(u)<+\infty$ (resp. $\Lip(\nabla u) <+\infty$).

\item For a natural $k$, $W^{k,\infty}_{\mathrm{loc}}(\Rd)$ is the space of functions $u$ that belong to the Sobolev space $W^{k,\infty}(\Omega)$ for any pre-compact domain $\Omega \subset \Rd$. For $u\in W^{2,\infty}_{\mathrm{loc}}(\Rd)$, the weak Hessian is denoted by $\nabla^2 u$.

\item $\Mes_+(\Rd)$ is the space of Borel measures on $\R^d$ with values in $[0,+\infty]$.  
  The Banach space of Borel measures valued in a finite dimensional normed vector space $E$ is denoted by $\Mes(\Rd;E)$. In addition, we agree that $\Mes(\Rd):=\Mes(\Rd;\R)$.

\item  $\mathcal{L}^d,\Ha^k, \delta_{x_0}$ are, respectively, the Lebesgue, the $k$-dimensional Hausdorff, and the Dirac delta measures on $\Rd$.
  
 \item The topological support of $\mu \in \Mes(\Rd;E)$ is denoted  by $\spt\mu$, 
 while $\mu\res A$ is the  restriction to a Borel subset $A\subset \Rd$. By the symbol $\mu \ll \nu$ we understand the absolute continuity of a measure $\mu$ with respect to $\nu \in \Mes_+(\Rd)$.

\item For a measure $\mu$ and a $\mu$-measurable map $T$, by  $T^\#(\mu)$ we understand the push forward, i.e.  $T^\#(\mu)(B) := \mu\big( T^{-1}(B) \big)$ for every Borel set $B$.

 \item  $\PP(\Rd):=\{ \mu \in \Mes_+(\Rd)\, :\, \mu(\Rd)=1\}$ is the set of probabilities on $\Rd$.

\item For $\gamma \in \PP(\Rd \times \ldots \times \Rd)$ on the product of $n$ ambient spaces, by $\gamma_{k_1,\ldots,k_m}$ we understand the marginal $\Pi_{k_1,\ldots,k_m}^\#(\gamma)$ where, for $m\leq n$,  $\Pi_{k_1,\ldots,k_m}$ is the projection onto the coordinates ${k_1,\ldots,k_m}$.

\item Assume $\mu\in \Mes_+(\R^n)$ and a map $x \mapsto \lambda^x \in \Mes(\R^m;E)$ that is $\mu$-measurable in the sense that $x \mapsto \lambda^x(A)$ is $\mu$-measurable for any Borel set $A \subset \R^m$. Provided that $\int \abs{\lambda^x}(\Rd) \,\mu(dx) < +\infty $, we will use the notation,
\begin{equation}
	\label{eq:convention}
	 \nu = \int  \lambda^x \mu(dx), \qquad \gamma= \mu \otimes \lambda^x
\end{equation}
to define measures $\nu \in \Mes(\R^m;E)$ and $\gamma \in \Mes(\R^n \times \R^m;E)$  that satisfy,
\begin{align*}
	\nu(A) := \int \lambda^x(A)\, \mu(dx), \quad \gamma(B):= \int  \left(  \int  \chi_{B}(x,y)\,\lambda^x(dy) \right) \mu(dx)
\end{align*}
for every Borel sets $A \subset \R^m$ and $B \subset \R^n \times \R^m$, where $\chi_B$ is the characteristic function of the latter.

 \item $[\mu]$ stands for the barycentre of a probability $\mu\in \PP(\Rd)$ whose first-order moment is finite, whilst $\var(\mu):= \int \abs{x-[\mu]}^2 \mu(dx)$ is its variance.

\item $\nu \succeq_c \mu$ denotes the convex order between  two probability measures $\mu,\nu\in \PP(\Rd)$ of finite first-order moments, and $MT(\mu,\nu) \subset \PP(\Rd \times \Rd)$ is the set of martingale transport plans.

\item $\mu \star \nu$ stands for the convolution of two probabilities  $\mu,\nu\in \PP(\Rd)$.

\item  $\pairing{\argu,\argu}$ will be used to denote a canonical scalar product in a finite dimensional space of vectors or matrices, whilst in the case of infinite dimensional spaces it will stand for the  duality bracket.

\item  The double distributional divergence $\dive^2$ of a matrix measure  $\sigma \in \Mes(\Rd;\Sdd)$  is an element of $\D'(\Rd)$ that is defined as follows,
\begin{equation*}
    \dive^2 \sigma = f \ \ \text{in \ } \D'(\Rd) \quad \Leftrightarrow \quad {\pairing{\nabla^2 \varphi, \sigma} }= \pairing{\varphi,f} \quad \ \ \forall\,\varphi \in \D(\Rd).
\end{equation*}

\item  Given a tensor-valued measure $\sigma\in \Mes(\R^d;\Sdd)$,  $\int \rho^0(\sigma)$ will denote 
 the  integral in the sense of the Goffman-Serrin convention \cite{Goffman}, namely,
 \begin{equation*}
    \int \rho^0(\sigma) := \int  \rho^0\! \Big(\frac{d \sigma}{d \theta}\Big) d\theta,
\end{equation*}
where $\theta$ is any non-negative Radon measure $\theta\in \Mes_+(\R^d)$ such that $\sigma \ll \theta$.
Due to the one-homogeneity of $\rho^0$, the above expression does not depend on $\theta$.
\end{itemize}


\bigskip
\section{Preliminaries}\label{Prelim}

\subsection{The classical duality framework}\label{dualclassic}  The duality theory involving  the linear constrained problem $\I(f)$ in \eqref{LCP}
and the general Beckmann formulation \eqref{Beckman}
is well understood in the case of the Hessian operator as far as we are confined in a bounded domain  $\O \subset \R^d$ and when $\varrho$ is any norm on $\Sdd$, see for instance \cite{bouchitte2007}. Since we are  concerned with the case $\O=\R^d$, some specific  functional spaces will prove useful in showing the existence  of solutions as well as  in the duality arguments.
Therefore, in addition to the general notations  given in the introduction, for $p\ge 1$ we introduce:
\begin{enumerate}[leftmargin=1.2\parindent]

\medskip
\item [-]  $\Mes_p(\Rd)$, the space of Borel signed  measures $\mu$ on $\R^d$ such that
 $\int (1+|x|^p) \, |\mu|(dx)<+\infty$. Then, $\PP_p(\R^d)$ stands for the subset of $\Mes_p(\R^d)$ consisting of probability measures with finite $p$-moment. The definition extends naturally to $\Mes_p(\Rd;E)$, where $E$ is a finite dimensional normed vector space.

\medskip
\item [-] $X_p(\R^d)$, the set of continuous functions $u\in \C^0(\R^d)$ such that 
$\Vert u\Vert_{X_p}:= \sup \frac{|u(x)|}{1+ |x|^p}<+\infty$.
The closed subspace $X_{p,0}(\Rd)$, consisting of those $u$ such that $\lim_{|x|\to+\infty} \frac{|u(x)|}{1+ |x|^p}=0$, is a separable
Banach space.

\end{enumerate}
A pairing  between $X_p(\R^d)$ and $\Mes_p(\Rd)$ is defined by $\pairing{u, \mu}:= \int_{\Rd} u\, d\mu$.
Noticing that  {$X_{p,0}(\Rd) = {(1+ |\argu|^p)}\, \C_0$},  it is easy to see  that the topological dual of $X_{p,0}(\Rd)$ can be identified with  $\Mes_p(\Rd)$ through this duality bracket. 
As a consequence of the dominated convergence theorem, we have a useful  convergence criterium for a sequence $(v_n)$ in $X_p(\R^d)$, namely,
\begin{align}\label{domiCV}
\sup_n \| v_n\|_{X_p} <+\infty \quad \text{and}\quad \lim_{n\to \infty} & v_n(x) =0 \quad \forall \,x\in\R^d \\
\nonumber
 &\Rightarrow \quad \ \ \pairing{v_n,\mu}  \to 0 \quad \ \ 
 \forall\, \mu\in\Mes_p(\R^d) .
\end{align}

The next result applies to general first-order distributional source terms  of the kind $f = f_0 - \dive F$, where
 $(f_0, F)$ is any pair in $\Mes_2(\R^d)\times\Mes_1(\R^d;\R^d)$ such that the following balance condition is met,
\begin{equation}\label{genload}
  \int f_0=0, \qquad  {\int x f_0 + \int F=0
} .
\end{equation}
The two conditions  mean that $f$ is orthogonal to affine functions, which  is clearly necessary for the finiteness of
\begin{equation}
	\label{If}
	 \I(f) =  \sup\Big\{ \pairing{u,f} :\ u\in C^{1,1}(\R^d),\ \ \Lip(\nabla u)\le 1  \Big\}.
\end{equation}
We recall the dual problem that we have named the second-order Beckmann formulation,
\begin{equation}
	\label{Ifprime}
	\I'(f) = \inf \biggl\{ \int \varrho^0(\sigma) \, : \, \sigma \in \Mes(\R^d;\Sdd), \  \dive^2 \sigma = f \ \text{in } \D'(\Rd) \biggr\} ,
\end{equation}
where  $\int \varrho^0(\sigma)$ is intended in the sense of convex one-homogeneous functionals on measures 
 \cite{Goffman}.
 
\begin{proposition}\label{nogap-classic} Assume that $f$ given  above satisfies \eqref{genload}. Let $\varrho$ be any norm on $\Sdd$ 
and $\varrho^0$ its polar defined by \eqref{dualnorm}. Then,  the supremum in \eqref{If} and the infimum in \eqref{Ifprime} are reached. 
Furthermore, we have the equality,
\begin{equation} \label{gap=0}
\I(f) =  \I'(f).
\end{equation}

\end{proposition}
\begin{proof}  We begin by proving the existence of a maximizer for \eqref{If}.  By the orthogonality conditions \eqref{genload}, we may restrict the supremum to functions $u$ belonging to the subset,
$$  K_0:= \Big\{u\in C^{1,1}(\R^d) \ : \ \Lip(\nabla u)\le 1,\ u(0)=0,\ \nabla u(0)=0 \Big\}.$$
Let $(u_n)$ be  a maximizing sequence in $K_0$. Then, $|u_n(x)| \le \frac1{2} |x|^2$ and  $|\nabla u_n(x)|\le |x|$.
By applying Arzela-Ascoli  theorem, we can assume that
$(u_n,\nabla u_n) \to (u,\nabla u)$ uniformly on compact subsets, where $u$ is a suitable element of $K_0$. 
To prove that $u$ is optimal, we only need to check that  $\pairing{u_n,f} \to \pairing{u,f}$, which, due to the particular form of $f$,   reduces to showing that,
$$  \pairing{u_n-u,f_0} \to 0, \qquad \pairing{\nabla (u_n-u),F} \to 0.$$
Put $v_n = u_n-u$. Then,  $(v_n,\nabla v_n)\to (0,0)$ pointwisely, while $|v_n(x)| \le |x|^2$ and $|\nabla v_n(x)|\le 2  |x|$.
 Therefore, $(v_n,\nabla v_n)$ is bounded in  $X_2(\R^d)\times X_1(\R^d;\Rd)$, and the convergence criterium \eqref{domiCV} applies.

  The existence of a minimal $\sigma$ in \eqref{Ifprime} follows from the direct method. Indeed,  the convex functional 
 $\sigma \in \Mes(\R^d;\Sdd) \mapsto  \int \varrho^0(\sigma)$  is coercive (hence inf-compact for the weak-* topology of $\Mes(\R^d;\Sdd)$),  while  the distributional constraint $\dive^2\sigma=f$ is weakly-* closed.
 
 We now prove the equality \eqref{gap=0} within two steps.
 
 \noindent {\bf Step 1:   $\I(f)\le \I'(f).$}  It is enough to prove  the following inequality,
 \begin{equation}\label{inf>sup}
 \pairing{u,f}   \le  \int \varrho^0(\sigma) \qquad\text{for every $(u,\sigma)\in K_0\times \Sf$},
\end{equation}
 where\   $\Sf:=\big\{\sigma \in \Mes(\R^d;\Sdd)\ :\ \dive^2 \sigma = f \ \text{in $\D'(\R^d)$}\big\}.$
 
 First we observe that we need only to show \eqref{inf>sup} for $u\in K_0\cap C^\infty$. Indeed, we may approximate any $u\in K_0$
 by the convolution $u_n= u\star \rho_n$, where  $\rho_n= n^d \rho(nx)$ is a sequence of mollifiers ($\rho\in \D(\R^d;\R_+)$, and $\int \rho=1$).
 Then, $\Lip(\nabla u_n) \le  \Lip(\nabla u) \le 1$, and, therefore, the sequence $(u_n,\nabla u_n)$ is bounded in $X_2(\R^d)\times X_1(\R^d;\Rd)$. The convergence $\pairing{u_n,f}\to \pairing{u,f}$ can be obtained
 by applying once more the criterium \eqref{domiCV} to $v_n=u_n-u$ and to $\nabla v_n$. 
 
 Let us now consider an element $u\in K_0 \cap C^\infty$ and a generic $\sigma\in \Sf$. Then, recalling \eqref{dualnorm}, we have
 $$ { \pairing {\nabla^2 u, \sigma} } \le  \int \varrho(\nabla^2u)  \, \varrho^0(\sigma) \le   \int \varrho^0(\sigma) <+\infty.$$
 Then, our claim \eqref{inf>sup} follows from Lemma \ref{lemma:byparts} (see Appendix \ref{app:byparts}) which states that,
\begin{equation}\label{byparts}
 {\pairing {\nabla^2 u, \sigma}}  =   \pairing{u,f}  =\pairing{u,f_0} + \pairing{\nabla u, F} \qquad  \forall \, \sigma\in \Sf .
 \end{equation}
This concludes Step 1.

 \smallskip
\noindent {\bf Step 2:   $\I(f)\ge \I'(f)$}. 
 We are going to show the equality,
$$ \I_{\rm reg}(f):= \sup \Big\{ \pairing {u,f}\ :\  u\in \D(\R^d), \ \ \Lip(\nabla u)\le 1\Big\}  = \I'(f) .$$
Clearly,  $\I_{\rm reg}(f)$  is not larger than $\I(f)$ and, thanks to Step 1,  the  equality above will imply that the three quantities coincide\footnote{Unfortunately, we were unable to find a direct approximation of an admissible $u$ by a sequence 
$(u_n)$ of compactly supported functions such that $\Lip(\nabla u_n)\le 1$.}. We introduce the value function $h: \C_0(\R^d;\Sdd) \to \R \cup \{+\infty\}$ defined by,
$$ h(\zeta) := \inf \Big\{ - \pairing {u,f} \, :  \,  u\in \D(\R^d), \ \varrho( \nabla^2 u + \zeta)\le 1  \Big\},
 $$
 with $h(\zeta)=+\infty$ if no admissible $u$ exists.
Then, $h$ is a proper convex functional such that $h(\zeta)\le 0$ whenever $\sup_x \varrho\big(\zeta(x)\big)\le 1$ ($u=0$ is then an admissible
competitor). It follows that $h$ is continuous at $0$ (with respect to the norm topology of $C_0(\R^d;\Sdd)$),
where it takes the value $h(0) = - \I_{\rm reg}(f)$. By the classical result of convex analysis (see Appendix \ref{app:convex}), it holds that,
 $$h(0) = h^{**}(0) =  -\min h^* .$$
 Then, the wished equality $\I_{\rm reg}(f)=\I'(f)$ will follow if we can identify the convex conjugate of $h$ as,
\begin{equation}\label{h*}
 h^*(\sigma)= \int \varrho^0(\sigma)\quad \text{if\quad$\dive^2 \sigma=f $ in $\D'(\R^d)$}
 , \qquad h^*(\sigma)= +\infty \ \  \text{otherwise}.
\end{equation}
For $\sigma\in \Mes(\R^d;\Sdd)$ we compute,
 \begin{align*} h^*(\sigma) &= \sup_{\zeta\in C_0(\R^d;\Sdd)} \left\{ \pairing{\zeta, \sigma} - h(\zeta)  \right\}   \\
 &=  \sup_{u\in \D(\Rd)} \ \sup_{\zeta\in C_0(\R^d;\Sdd)}  \left\{ \pairing{\zeta, \sigma} + \pairing{u,f} \ : \ 
 \varrho(\nabla^2 u+\zeta)\le 1 \right\}  \\
 &=  \sup_{u\in \D(\Rd)} \ \sup_{\chi\in C_0(\R^d;\Sdd)} \left\{ \pairing{\chi, \sigma} -\pairing{\nabla^2u, \sigma}+  \pairing{u,f} \ : \ 
 \varrho(\chi)\le 1 \right\} \\
 &= \int \varrho^0(\sigma) +\sup_{u\in \D(\Rd)} \left\{  -\pairing{\nabla^2u, \sigma}+  \pairing{u,f}  \right\} , 
\end{align*}
where: 
\begin{enumerate}
\item [-]  in the third line, we changed variable to $\chi= \nabla^2 u+ \zeta$ which runs over the whole space $C_0(\R^d;\Sdd)$;
\item [-]  in the last line, for fixed $u$, we have taken the supremum with respect to $\chi$ recovering $\int \varrho^0(\sigma)$ which agrees with the support function of  the subset $\{\chi \in C_0(\R^d;\Sdd) : \varrho(\chi)\le 1\}$, see \cite{bouchitte1988}.
\end{enumerate}
The finiteness of $h^*(\sigma)$ implies that $\pairing{u,f}= \pairing{\nabla^2u, \sigma}$ for every $u\in \D(\R^d)$, meaning that
$\dive^2 \sigma= f$ in $\D'(\Rd)$. This proves  \eqref{h*} and concludes Step 2.
The proof of Proposition \ref{nogap-classic} is complete.
\end{proof}




\subsection{Convex order and martingale transport}\label{MTsurv} Stochastic ordering plays an important role in probability theory as a tool for comparing random variables through their probability laws. 
 Here, we are concerned specifically with the convex order between measures in $\Mes_p(\R^d;\R_+)$ ($p\in \{1,2\}$).
 To any measure $\mu\in \Mes_1(\Rd;\R_+)$ we associate its total mass $\|\mu\|$ and its barycentre $[\mu]$ given by,
$$   \|\mu\| = \int  \mu, \qquad [\mu]=  \frac1{\|\mu\|} \int x \, \mu(dx).$$
 
 \begin{definition}\label{def:convexorder} Given  two non-negative measures $\mu,\nu$ in  $\Mes_1(\R^d;\R_+)$,  we say that $\nu$ dominates $\mu$ in the sense of convex order, in short  $\nu  \succeq_c \mu$,  if for every \emph{convex} function $\varphi:\R^d \to \R$ there holds the inequality,
\begin{equation}\label{order}
\int \varphi\, d\nu \ge \int  \varphi\, d\mu .
 \end{equation}
 \end{definition}
 By the  Moreau-Yosida infimal convolution technique, we know that any convex lower semi-continuous function 
 $\varphi: \R^d \to \R \cup {+\infty}$ is the non-decreasing limit 
 of a sequence of convex Lipschitz continuous functions. 
 Therefore, in order to show that $\nu  \succeq_c \mu$,  the inequality \eqref{order} needs to be checked  for  convex Lipschitz continuous functions  only. If it is the case, then  \eqref{order} automatically extends to any convex lower semi-continuous $\varphi: \R^d \to \R\cup \{+\infty\}$.

The following  properties are straightforward.
\begin{itemize}

\item By testing \eqref{order}  with affine functions (these are integrable), we see  that,
  $$  \nu  \succeq_c \mu \qquad \Rightarrow  \qquad  \|\mu\| =\|\nu\| \quad \text{and} \quad [\mu]=[\nu] .$$
\item \  $\mu \succeq_c \delta_{[\mu]}$ \   (Jensen's inequality).
\end{itemize}

The next characterization of the convex order is crucial. Beforehand, let us recall that for any  coupling measure  $\gamma \in \Gamma(\mu,\nu)$ there exists a unique disintegration of the form $\gamma = \mu \otimes \gamma^x$ (see the notation \eqref{eq:convention}), where $x \mapsto \gamma^x \in \PP(\Rd)$ is a $\mu$-measurable  mapping, cf. for instance Section 2.5 in \cite{ambrosio2000}. By the set of martingale transport plans $MT(\mu,\nu) $ we understand the family of  those couplings $\gamma \in \Gamma(\mu,\nu)$
whose disintegration
satisfies the condition $[\gamma^x] = x$ \   $\mu$-a.e. It turns out that the set $MT(\mu,\nu)$ is non-empty if and only if $\nu  \succeq_c \mu$. This fact is a direct corollary of the Strassen theorem \cite{strassen1965}:

\begin{theorem}[\textit{Strassen}] The convex order
    $\nu  \succeq_c \mu$ holds true if and only if there exists a $\mu$-measurable map
 $x\mapsto p^x \in \PP_1(\R^d)$ such that:
 \begin{enumerate}[label={(\roman*)}] 
\item   $[p^x] = x$ \   $\mu$-a.e.,
\item  $\nu (B) = \int p^x(B)\, \mu(dx)$  for any Borel set $B\subset \R^d$. 
\end{enumerate}
\end{theorem}

 \noindent Some straightforward consequences of Strassen theorem for measures $\mu,\nu \in \PP_2(\R^d)$ are listed below:
 
 \begin{itemize}
 
 \item [(p1)] Assume that $\nu  \succeq_c \mu$. Then, $  \var(\nu)\ge \var(\mu)$, while  
strict inequality holds unless $\mu=\nu.$
Indeed, assuming that $[\mu]=0$, for $p^x$ such that $\nu= \int p^x \, \mu(dx)$ with $[p^x]=x$ \ $\mu$-a.e.,
 we have
 $$  \var(\nu) -\var(\mu)=  \int \big( \pairing{|\argu|^2,p^x} - |x|^2\big)\, \mu(dx),$$
 which is positive unless the Jensen's inequality $ \pairing{|\argu|^2,p^x} \ge |x|^2$ is an equality for $\mu$-a.e. $x$.  By the strict convexity of $|\cdot|^2$, this  is possible only if $p^x =\delta_x$, hence, if $\nu=\mu$.

 \item [(p2)] Assume that $[\nu]=0$, and take the convolution $\rho=\mu\star \nu$. Then, it holds that   $ \rho  \succeq_c \mu .$\
Indeed,  $\rho= \int p^x \, \mu(dx)$ where $p^x := (\ident+x)^\# \nu$ satisfies the condition $[p^x]=x$ ($\mathrm{id}$ is the identity map).
Thanks to this  property, one can check (see \cite{muller2001stochastic}) that, for centred Gaussian distributions $\rho,\mu$ on $\Rd$ with the respective covariance matrices  $R, M \in \Sddp $, the condition $\rho\succeq_c \mu$ reduces to the order relation $ R \geq M$ in the sense of quadratic forms.   

\end{itemize}

Finally, we {recall} that optimal transport  problems under the martingale constraint of the kind
$ \inf\left\{ \iint c(x,y) \, \gamma(dxdy)\, :\, \gamma\in MT(\mu,\nu)\right\} $
 are  considered  in the literature, most often for the cost $c(x,y) =|x-y|^p$ with $p\ge1$, see for instance \cite{ABC,beiglbock2013,Wiesel2023,gozlan,Ghoussoub}.
 A peculiarity of the quadratic cost $p=2$  is 
that, for $\nu  \succeq_c \mu$,   the infimum above is reached by any $\gamma\in MT(\mu,\nu)$
since the total cost remains constant and equal to $\var(\nu)-\var(\mu)$ on this subset. This fact will manifest itself in the proof of Theorem 1.3.

\bigskip
\section{Proofs} \label{proofs}
A quite technical direct proof of Theorem \ref{main-intro} could be derived  by leveraging the Le Gruyer's three-point characterization \eqref{3-jet}  of the  feasible set $\{u\in C^{1,1} : \Lip(\nabla u)\le 1\}$ (see   \cite{legruyer2009}  {and Lemma \ref{lipConv} below} for more details on this characterization).
However, as we aim to emphasize the important link between our initial problem and stochastic optimization under convex order dominance,  we choose here to deal first with the proof of Theorem \ref{thm2-intro}. After that, our main result in Theorem \ref{main-intro} and its Corollary \ref{coro1-intro} will follow smoothly. 

 In the whole section we assume that $\mu,\nu$ are centred probability measures in $\PP_2(\R^2)$. Fixing the zero barycentre $[\mu] = [\nu] = 0$ is not restrictive since  all the problems considered herein are translation invariant.

\subsection{Dualization of the minimal variance problem and optimality conditions}  Let us rewrite $\V(\mu,\nu)$  defined in  \eqref{min-var} in the form $\V(\mu,\nu) = \inf \big\{ \var(\rho) \, :\,  \rho \in\A(\mu,\nu)\big\}$ where,
 \begin{equation}\label{def:Amunu}
\A(\mu,\nu):= \Big\{ \rho\in \PP_2(\R^d) \ : \ \rho \succeq_c \mu, \ \ \rho \succeq_c \nu \Big\}.
\end{equation}
  By the properties (p1), (p2) that conclude Section \ref{MTsurv}, we know that
   $\rho=\mu\star\nu$ belongs to $\A(\mu,\nu)$, whence,
\begin{equation}\label{Vbounds}
 \max \big\{ \var(\mu) , \var(\nu) \big\} \ \le\  \V(\mu,\nu) \ \le\ \var(\mu) + \var(\nu) .
\end{equation}
Next, we consider the dual variational problem that involves pairs  $(\f,\psi)$ of convex functions.
    Let $\K$ be the set of convex functions that are in {$C^{1,1}(\Rd)$. Note that we have $\K\subset X_2(\R^d)$.  Then, we set 
\begin{equation}\label{2conv}
\V'(\mu,\nu) := \sup \left\{ \int \f\, d\mu + \int \psi\, d\nu \ :\ (\f,\psi)\in \FF\ \right\},
\end{equation}
where $\FF:=\big\{(\f,\psi)\in \K^2 \, :\, {\f(z)+\psi(z) \le \abs{z}^2 \ \ \forall\, z \in \Rd } \big\}$.
 \begin{proposition}\label{duality-2convex} \ There exists an optimal $\rho$ for \eqref{min-var}, and we have the no-gap equality,
 $$  \V(\mu,\nu)  = \V'(\mu,\nu) .$$
 Furthermore, $\rho\in\A(\mu,\nu) $ and $(\f,\psi)\in \FF$ are optimal for \eqref{min-var} and \eqref{2conv}, respectively,
if and only if the following optimality conditions are fulfilled:
	\begin{align}
		\label{eq:opt_cond_rho1}
		\begin{cases}
			(i) &  \varphi + \psi = \abs{\argu}^2 \quad \text{$\rho$-a.e.},\\
			(ii) & \int \varphi \, d\rho=  \int \varphi \, d\mu, \quad \int \psi \, d\rho=  \int \psi \, d\nu.
		\end{cases}
	\end{align}	
\end{proposition}

\begin{remark}\label{existence?}\ The existence issue  for $\V'(\mu,\nu)$ is not trivial. In fact,  optimal pairs $(\f,\psi)\in \FF$ will be deduced from  the solutions to \eqref{secondorder} by means  of   Proposition \ref{LCP=2convex}
 in Section \ref{IversusW}.
 \end{remark}

\begin{proof}  We start by proving that \eqref{min-var} admits solutions. 
 By \eqref{Vbounds}, there exists a minimizing sequence $(\rho_n)$ in $\A(\mu,\nu)$ such that
$\var(\rho_n)\to \V(\mu,\nu) <+\infty$. Then, $(\rho_n)$ is bounded in $\Mes_2(\R^d)$ and, up to extracting a subsequence, 
we have $\rho_n\weakstar \rho$ in the duality between $X_{2,0}(\R^d)$ and  $\Mes_2(\R^d)$. Therefore, as $\rho_n\in\A(\mu,\nu)$, by passing to the limit $n\to\infty$, the convex order relations  $\int \varphi\, d\rho \ge \max\{\int \varphi\, d\mu, \int \varphi\, d\nu\}$ are deduced for convex Lipschitz continuous functions $\varphi$ since they belong to $X_{2,0}(\R^d)$.  As pointed out below Definition \ref{def:convexorder}, this is enough to ensure that $\rho\in\A(\mu,\nu)$. The optimality of $\rho$ follows 
since $\V(\mu,\nu)=\liminf_n \var(\rho_n)\ge \var(\rho)$. 

Next, we prove the equality $ \V(\mu,\nu)  = \V'(\mu,\nu)$. Notice that the inequality  $ \V(\mu,\nu) \ge\V'(\mu,\nu)$ is straightforward since for every admissible $\big(\rho,(\f,\psi)\big)$ we have,
\begin{equation}\label{sup<inf}
		\int \varphi \, d\mu + \int \psi \,d\nu \leq  \int \varphi \, d\rho + \int \psi \,d\rho \leq \int \abs{\argu}^2 d\rho.
	\end{equation}
To show the opposite inequality, we introduce the perturbation function $h:X_{2,0}(\R^d) \to \R \cup \{+\infty\}$ defined by,
$$  h(\chi)  := \inf\left\{  -\left(\int \f\, d\mu + \int \psi\, d\nu\right) \, :\, (\f,\psi)\in \K^2 , \ \ \f +\psi + \chi \le |\cdot|^2 \right\}.$$	
We see that $h(0) = - \V'(\mu,\nu)$ is finite, while the function $h$ is convex. Moreover, by taking $\f=\psi= -\frac1{2}$ as a competitor, we have $  h(\chi) \le 1$  whenever \  $\chi\le 1 +|\cdot|^2  .$
Thus, $h$ has a finite upper bound on the unit ball in the Banach space $X_{2,0}(\R^d)$. Therefore, $h$ is continuous at $0$ and, by Appendix \ref{app:convex}, it holds that
$h(0)= h^{**}(0) = - \min h^*$, where $h^*$ denotes the Fenchel conjugate of $h$ on the dual space $\Mes_2(\R^d)$.
The asserted equality will follow if we can prove that,
\begin{equation}\label{claim:h^*}
 h^*(\rho)= \var(\rho)  \quad \text{if\quad$\rho\in\A(\mu,\nu)$}, \qquad  h^*(\rho) = +\infty  \quad \text{otherwise}.
 \end{equation}
Let us compute,
\begin{align*}  h^*(\rho) =  \sup_{\substack{(\f,\psi)\in \K^2 \\ \chi\in X_{2,0}(\R^d)}}  \left\{ \int \chi\, d\rho + \int \f\, d\mu + \int \psi\, d\nu \ :\   \f +\psi + \chi \le |\cdot|^2 \right\}.
\end{align*}
Clearly, one has $h^*(\rho) \le \var(\rho) = \int \abs{z}^2 \rho(dz)$ if $\rho\in\A(\mu,\nu)$, and $h^*(\rho) = +\infty$ when $\rho$ is not positive. Assuming now that $\rho \geq 0$, we look for the lower bound for $h^*$ by restricting the supremum above to pairs $(\f,\psi)\in \K^2$ which are Lipschitz continuous. Fixing such a pair, we see that the function $\ov \chi(z):= |z|^2 -\f(z)-\psi(z)$ belongs to $X_2(\R^d)$ and it is positive for large $|z|$.  By truncation, it  can be approximated by a sequence $\chi_n \in \C_0(\R^d)$ such that $\chi_n\to \ov \chi$ increasingly, and $\sup_n \|\chi_n\|_{X_2(\R^d)} <+\infty$. Since $\chi_n+\f+\psi\le |\cdot|^2$,
 after certain manipulations we are led to\footnote{All integrals involved below are finite since Lipschitz continuous functions belong to $X_2(\R^d)$.},
$$h^*(\rho) \ge \int (\chi_n + \f+\psi)\, d\rho + \left(\int \f\, d\mu - \int \f \, d\rho\right) + \left(\int \psi \, d\nu - \int \psi \, d\rho\right).$$
Then, passing to the limit as $n\to\infty$ (see \eqref{domiCV}), we get the inequality,
 $$h^*(\rho) \ge \int |z|^2\, d\rho +  \left(\int \f\, d\mu - \int \f \, d\rho\right) + \left(\int \psi \, d\nu - \int \f \, d\rho\right) ,$$
which holds true for every pair of convex Lipschitz continuous functions $(\f,\psi)$. Therefore, the finiteness of $h^*(\rho)$ implies that $\rho$ dominates $\mu$ and $\nu$ in the convex order. In this case, we infer that $\rho\in\A(\mu,\nu)$ (in particular $[\rho] = 0$), while   $h^*(\rho)\ge  \int |z|^2\, d\rho =\var(\rho)$.
 This proves our claim \eqref{claim:h^*}, hence the equality $\V(\mu,\nu)=\V'(\mu,\nu)$. 
 
 We see now that  a pair $\big(\rho,(\f,\psi)\big) \in A(\mu,\nu)\times \FF$ is optimal if and only if  the inequalities in \eqref{sup<inf} are  equalities. In turn, these equalities are equivalent to the conditions (i),\,(ii) stated in Proposition \ref{duality-2convex}.
 \end{proof}

\subsection{ Convex-order characterization of the admissible set $\Sigma(\mu,\nu)$} The proof of Theorem \ref{thm2-intro} will rely on the relation between the admissible subset $\Sigma(\mu,\nu)$
for the  optimal transport problem \eqref{OT3} and the admissible subset $\A(\mu,\nu)$ for \eqref{min-var}.
This relation is illuminated by the following result:
\begin{lemma}\label{Sigmamunu=Amunu} Let $\pi\in \PP_2((\Rd)^3)$ be a $3$-plan with marginals $(\mu,\nu,\rho)$. Define the marginals $\pi_{1,3} : =\Pi_{1,3}^\#(\pi)$ and $\pi_{2,3} :=\Pi_{2,3}^\#(\pi)$, which are the push forwards of $\pi(dxdydz)$ through the projection maps $(x,y,z)\to (x,z)$ and $(x,y,z)\to (y,z)$, respectively.  
Then,
\begin{equation}\label{equival}
\pi\in \Sigma(\mu,\nu) \quad \Leftrightarrow \quad  \begin{cases} \pi_{1,3} \in MT(\mu,\rho), \\
  \pi_{2,3} \in MT(\nu,\rho). \end{cases}
\end{equation}
Accordingly, we obtain the equality,
\begin{equation}\label{Sigma=A}
\A(\mu,\nu) = \left\{ \rho \in \PP_2(\R^d)\ :\ \exists\, \pi\in  \Sigma(\mu,\nu), \ \ \Pi_3^\#(\pi) =\rho \right\} .
 \end{equation}
  
\end{lemma}

\begin{proof} Upon recalling the equilibrium conditions \eqref{equilibrium} which characterize the convex subset $\Sigma(\mu,\nu)$,
checking the equivalence  \eqref{equival} amounts to verifying  the two equivalences:
\begin{align*}
	 &\iiint  \pairing{z-x, \Phi(x)} \, \pi(dxdydz) = 0\quad \forall\, \Phi\in C_0(\R^d;\R^d) \ \   \Leftrightarrow \ \    \pi_{1,3}\in MT(\mu,\rho), \\
	 & \iiint  \pairing{z-y, \Psi(y)} \, \pi(dxdydz) = 0 \quad  \forall\, \Psi\in C_0(\R^d;\R^d) \ \   \Leftrightarrow \ \    \pi_{2,3}\in MT(\nu,\rho) .
\end{align*}
Let us prove the first equivalence; the proof of the second one is similar and will be skipped. We consider the disintegration $\pi_{1,3}(dxdz) = \mu(dx) \otimes p^x(dz)$, which gives,
\begin{align*} \iiint  \pairing{z-x, \Phi(x)} \, \pi(dxdydz) &= \iint \pairing{z-x, \Phi(x)}\, \pi_{1,3}(dxdz) \\
&=  \int \left(\int\pairing{z-x, \Phi(x)}\, p^x(dz)\right)  \mu(dx) \\ &= \int \pairing{[p^x]-x, \Phi(x)} \, \mu(dx) .
\end{align*}  
Clearly, these integrals vanish for every $\Phi\in C_0(\R^d;\R^d)$ if and only if $[p^x]=x$ holds $\mu$-a.e. This is exactly the martingale condition that characterizes $\pi_{1,3}\in MT(\mu,\rho)$.

Let us now prove the equality  \eqref{Sigma=A}. 
By \eqref{equival}, the condition $\pi\in \Sigma(\mu,\nu)$ implies that $\rho\in\A(\mu,\nu)$. Conversely, if
$\rho \succeq_c \mu$ and $\rho \succeq_c \nu$, Strassen theorem ensures  the existence of martingale transports 
$\gamma_{1,3}\in MT(\mu,\rho)$ and $ \gamma_{2,3}\in MT(\nu,\rho)$. Then, we can recover an element $\pi\in \Sigma(\mu,\nu)$
with $\rho$ for its third marginal by using a gluing construction between $\gamma_{1,3}$ and $\gamma_{2,3}$. A simple one (it is not unique)  is as follows. Let us  consider the disintegrations of the measures $\gamma_{i,3} $ ($i\in\{1,2\}$) with respect to their second  marginal $\rho$. This gives $\rho$-measurable families $\{p_i^z\}$ in $\PP(\R^d)$ such that,
\begin{align*}
	\gamma_{1,3}(dxdz) &= \int \big( p_1^z(dx) \otimes \delta_\xi(dz)\big) \, \rho(d\xi), \\
	\gamma_{2,3}(dydz) &= \int  \big(p_2^z(dy) \otimes \delta_\xi(dz)\big) \,  \rho(d\xi).
\end{align*}
 Then, it is easy to check that the measure  $\pi(dxdydz)= \int  \big(p_1^z(dx)\otimes p_2^z(dy)\big)\otimes \delta_\xi(dz)\, \rho(d\xi)$ has $(\mu,\nu,\rho)$ for its marginals, and it satisfies {$\pi_{i,3}:=\Pi_{i,3}^\#(\pi) = \gamma_{i,3}$}.  
\end{proof}

\subsection{Proof of  the equality $\V'(\mu,\nu)= \I(\nu-\mu) + \frac1{2} \big(\var (\mu)+\var(\nu)\big)$} \label{IversusW}
{First, we expound the links  between the bounds on the Hessian, the convexity properties, and the three-point inequality \eqref{3-jet}.}
\begin{lemma}
	\label{lipConv}
	For any continuous function $u: \Rd \to \R$ the following conditions are equivalent:
	\begin{enumerate}[label={(\roman*)}] 
		\item $u \in C^{1,1}(\Rd)$ and $\Lip(\nabla u) \leq 1$;
		\label{phi_C11}
		\item \label{hessian} $u \in W^{2,\infty}_{\mathrm{loc}}(\Rd)$ and $-\mathrm{Id} \leq \nabla^2 u \leq \mathrm{Id}$ a.e. in $\Rd$; 
		\item \label{both_convex} both functions $ \frac{1}{2}\abs{\argu}^2 +u$ and $\frac{1}{2}\abs{\argu}^2 - u$ are convex;
		\item there exists a continuous function $U:\Rd \to\Rd$ such that, 
		\begin{align}\label{ineq-3points}
			[u(y) + \pairing{U(y), z-y}]  - [u(x) + \pairing{U(x), z-x} ]
			\le c(&x,y,z) \\
			\nonumber
			&\forall\, (x,y, z) \in (\Rd)^3 .
		\end{align}
	\end{enumerate}
	Moreover, whenever (iv) holds true, $U = \nabla u$.
\end{lemma}
\begin{proof}
	The equivalences (i) $\Leftrightarrow$ (ii) $\Leftrightarrow$ (iii) are classical, whilst the equivalence with (iv) was established in \cite{legruyer2009}. For the reader's convenience, we shall provide an alternative proof of the latter. For a given $u$, we introduce the pair of functions,
	\begin{equation}
		\label{phipsi}
		\varphi:= \frac1{2} \abs{\argu}^2 - u, \qquad   \psi:= \frac1{2} \abs{\argu}^2 + u.
	\end{equation}
	
	Assume first that $u$ satisfies (i),\,(ii),\,(iii), which makes $\varphi,\psi$ convex and smooth. We can define two bivariate functions
	$  {\tilde \f}(x,z) := \f(x) + \pairing{ \nabla\f(x), z-x}$ and 
	${\tilde \psi}(y,z) := \psi(y) + \pairing{ \nabla\psi(y), z-y}.$
	By the convexity, we have  $  {\tilde \f}(x,z) \leq \f(z)$ and $ {\tilde \psi}(y,z) \leq \psi(z)$. Since $\varphi + \psi = \abs{\argu}^2$, for every $(x,y,z)\in (\R^d)^3 $ the inequality follows,
		\begin{align*}
		\f(x) + \pairing{ \nabla\f(x), z-x} +  \psi(y) + \pairing{ \nabla\psi(y), z-y} &\, \le \, |z|^2.
	\end{align*}
	Upon changing the variables back to $u$ and recalling that $c(x,y,z)= \frac{1}{2} ( \abs{z-x}^2 + \abs{z-y}^2 )$, simple manipulations lead to \eqref{ineq-3points} for $U =\nabla u$.
	
	Conversely, assume that $u$ satisfies (iv) for some $U$. Then, performing the same manipulations backwards from \eqref{ineq-3points},  we get
		\begin{align*}
		\f(x) + \pairing{ \Phi(x), z-x} +  \psi(y) + \pairing{ \Psi(y), z-y} &\, \le \, |z|^2,
	\end{align*}
	where  $(\Phi,\Psi):=( \ident - U, \ident + U)$.
	We define the convex functions,
	\begin{equation*}
		\check{\f}(z): = \sup_{x \in \Rd} \Big\{	\f(x) + \pairing{ \Phi(x), z-x}  \Big\}, \quad 	\check{\psi}(z) := \sup_{y \in \Rd} \Big\{	\psi(y) + \pairing{ \Psi(y), z-y}  \Big\}.
	\end{equation*}
	It is clear that $\check{\f}+\check{\psi} \leq \abs{\argu}^2$ and ${\f} \leq \check\f$, ${\psi} \leq \check\psi$. Since $\f +\psi = \abs{\argu}^2$, we deduce that ${\f} = \check\f$, ${\psi} = \check\psi$, which renders $\varphi,\psi$ convex with $\Phi,\Psi$ as the respective subgradients. Recalling \eqref{phipsi}, we deduce (iii) and thus also (i),\,(ii),  whence the smoothness of $u, \varphi,\psi$.  Finally, we see that $U = \Psi - \ident = \nabla \psi - \ident = \nabla u$.
\end{proof}
	\color{black}

Next, we consider the subclass $\G\subset \C^0(\R^d)$ consisting of  continuous functions $\f$ such that $|\argu|^2 -\f$ admits an affine minorant. Note that $\f\in\G$ if and only if $\f(x) \le \abs{x-x_0}^2 + b $ for a suitable pair $(x_0,b)\in \R^d\times \R_+$.
Then, we introduce the transform $\LL : \f\in \G\to \hat\f\in\G$ defined by,
\begin{equation} \label{hatfi}
	\LL \f = \hat \f \quad \text{where}\quad \hat\f (x) := \abs{x}^2  -(\abs{\argu}^2-\varphi)^{**}(x) .
\end{equation}
Below we show that $\LL$ preserves convexity while adding smoothness. This is a rephrase of the fact that taking convex envelope preserves semi-concavity. The following result essentially reproduces Theorem 2.3 in \cite{azagra2018}.
\begin{lemma} \label{fundamental}
	\label{lem:LL} The transform $\LL$ enjoys the following properties:
	\begin{enumerate}[label={(\roman*)}] 
		\item $\LL \f \ge \f$, while $\LL\f \equiv  \f$ if and only if \	$\abs{\argu}^2-\varphi$ is convex;
		\item $\LL \circ \LL =\LL$ (idempotence);
		\item If $\f$ is convex, then   $\hat\f:= \LL \f$ is convex and $C^{1,1}$, whilst $u:= \frac{1}{2} \abs{\argu}^2 -\hat\f$ satisfies $\Lip(\nabla u)\le 1$.
	\end{enumerate}
\end{lemma}

\begin{proof} The first two properties are straightforward.  In order to show that $\hat \f$ is convex, we need only to check the Jensen inequality $ \int \hat \f(z+\xi) \,p_0(d\xi) \ge  \hat \f(z)$ for every centred  finitely supported
	probability  $p_0$ and for any $z\in\R^d$. In view of the particular form of $\hat \f$ given in \eqref{hatfi}, this amounts to showing that,
	\begin{equation}\label{claim:Jensen}
		\int (\abs{\argu}^2-\varphi)^{**}(z+\xi)\, p_0(d\xi)\ \le\ (\abs{\argu}^2-\varphi)^{**}(z)  + \var(p_0).
	\end{equation}
	To prove \eqref{claim:Jensen}, we fix $\e>0$ and  choose a finitely supported probability $p^z$ such that $[p^z] = z$ and, $$ (\abs{\argu}^2-\varphi)^{**}(z) \geq \int \big( \abs{\zeta}^2 - \varphi(\zeta) \big) \, p^z(d\zeta) - \eps . $$
	Then, by applying  Jensen's inequality to  $(\abs{\argu}^2-\varphi)^{**}$ (which is majorized by $\abs{\argu}^2-\varphi$), we infer that for every $\xi\in \Rd$ we have,
	\begin{align*}\label{epsbound}
		(\abs{\argu}^2-&\varphi)^{**}(z+\xi)- (\abs{\argu}^2-\varphi)^{**}(z) \\
		&\leq 
		\int   \left( \big( \abs{\zeta+\xi}^2 - \varphi(\zeta+\xi) \big)-  \big( \abs{\zeta}^2 - \varphi(\zeta) \big)\right)  p^z(d\zeta)+\eps \\
		&=  \abs{\xi}^2 - 2 \pairing{z,\xi} - \int \big(\varphi(\zeta+\xi) - \varphi(\zeta)\big)\, p^z(d\zeta) + \,\eps\,  .
	\end{align*}
	By integrating  with respect to the centred measure $p_0(d\xi)$ and by  Fubini theorem, we deduce that,
	\begin{align*}
		\int (\abs{\argu}^2-&\varphi)^{**}(z+\xi)\, p_0(d\xi) -  (\abs{\argu}^2-\varphi)^{**}(z)\\
		& \leq \var(p_0) + \eps  - \iint \big(\varphi(\zeta+\xi) - \varphi(\zeta)\big)\, p^z(d\zeta) \otimes p_0(d\xi)  \\
		& =  \var(p_0) + \eps - \int \left(\int \big(\varphi(\zeta+\xi) - \varphi(\zeta)\big)\,p_0(d\xi)\right)   p^z(d\zeta) \\
		& \le  \var(p_0)  + \eps .
	\end{align*}
	Let us point out that, in order to reach  the last line above, we used  the convexity of  $\varphi$, which, by Jensen's inequality,  renders the integral with respect to $p_0(d\xi)$ non-negative. Since $\eps$ can be chosen arbitrarily small, we get our claim \eqref{claim:Jensen}, hence the convexity of $\hat \f$.  
	
	To complete the proof of the assertion (iii), we observe 
	that the function $u=\frac{1}{2} \abs{\argu}^2 -\hat\f$ is such that $\frac{1}{2} \abs{\argu}^2- u= \hat\f$ and
	$\frac{1}{2} \abs{\argu}^2+ u=(\abs{\argu}^2-\varphi)^{**}$, which are convex functions. By virtue of Lemma \ref{lipConv}, it follows that $u$ (hence also  $\hat\f$) is $C^{1,1}$, and there also holds  $\Lip(\nabla u)\le 1.$ 
\end{proof}

\begin{proposition}\label{LCP=2convex} \ Let $\ov u$ be a solution to \eqref{secondorder}. Then, the pair of convex functions
	$(\ov\f,\ov\psi)$ given by,
	\begin{equation}\label{ufpsi}
		\ov\f=\frac{1}{2} \abs{\argu}^2 -\ov u, \qquad \ov\psi= \frac{1}{2} \abs{\argu}^2 +\ov u 
	\end{equation}
	solves the maximization problem  \eqref{2conv}.
	Accordingly, we have the equality,
	$$  \I(\nu-\mu) + \frac1{2} \big(\var(\mu)+\var(\nu)\big) = \V'(\mu,\nu).$$
\end{proposition}

\begin{proof}  Since $\ov u$ is $C^{1,1}$ with $\Lip(\nabla \ov u)\le 1$,  the pair of  functions
	$(\ov\f,\ov\psi)$  given by \eqref{ufpsi} belongs to the class $\FF$ of admissible competitors for \eqref{2conv}, thanks to the equivalence stated in  Lemma \ref{lipConv}. 
	Therefore,
	\begin{align*}
		\I(\nu-\mu) &=  \int \ov u\, d\nu -  \int \ov u\, d\mu\\
		& =   \int \ov\f\, d\mu + \int \ov\psi\, d\nu -\tfrac{\var(\mu)+\var(\nu)}{2}\\
		&\le  \V'(\mu,\nu) -\tfrac{\var(\mu)+\var(\nu)}{2} .
	\end{align*}
	Thus, we are done if we can prove the converse inequality, namely  
	\begin{equation}\label{claim:W<I}
		\V'(\mu,\nu) \le\,  \I(\nu-\mu) + \frac1{2} \big(\var(\mu)+\var(\nu)\big) .
	\end{equation}
	Let $(\f,\psi)\in \FF$ be  any admissible pair for \eqref{2conv}. Since the  convex continuous function  $\psi$  admits an affine minorant,  the inequality $\psi\le |\cdot|^2 -\f$ implies that $\psi=\psi^{**} \le ( |\cdot|^2 -\f)^{**}$, while $\f$ belongs  
	to the subclass $\G$ on which the $\LL$-transform is well defined. By virtue of Lemma \ref{fundamental},  
	$\hat \f:=\LL \f$  is convex and satisfies $\hat\f \ge \f$. Therefore, it holds that,
	$$\int \f \, d\mu + \int \psi\, d\nu \, \le \int \hat \f \, d\mu + \int ( |\cdot|^2 -\f)^{**}\, d\nu\, =\int \hat \f \, d\mu 
	+\int   (|\cdot|^2 -\hat\f)\, d\nu . $$
	In terms of $u:= \frac1{2} |\cdot|^2 - \hat\f$, the latter inequality can be rewritten as follows,
	$$  \int \f \, d\mu + \int \psi\, d\nu\,  \le  \int u\, d\nu -  \int u\, d\mu \, +\,  \frac1{2} \big(\var(\mu)+\var(\nu)\big) .$$
	By the assertion (iii) of Lemma \ref{fundamental},  $u$ is an admissible competitor for \eqref{secondorder},  hence $ \int u\, d\nu - \int u\, d\mu \le \I(\nu-\mu)$. This gives the following upper bound,
	$$  \int \f \, d\mu + \int \psi\, d\nu\,  \le\,  \I(\nu-\mu)+  \frac1{2} \big(\var(\mu)+\var(\nu)\big).$$
	The desired inequality  \eqref{claim:W<I} is obtained by taking the supremum with respect to all pairs $(\f,\psi)\in \FF$.    \end{proof}

\subsection{Proof of Theorem \ref{thm2-intro}}
  By Proposition \ref{duality-2convex} and 
Proposition \ref{LCP=2convex}, we already know that,
$$ \I(\nu-\mu) + \frac1{2} \big(\var(\mu)+\var(\nu)\big) =\V'(\mu,\nu)  = \V(\mu,\nu).$$    
Therefore,  it remains  to check that the infimum $\JJ(\mu,\nu)$ in the three-marginal optimal transport  problem \eqref{OT3} satisfies
the equality,
\begin{equation}\label{JJV}
\JJ(\mu,\nu) = \V(\mu,\nu) - \frac1{2} \big(\var(\mu)+\var(\nu)\big) .
\end{equation}
Let $\pi\in \Sigma(\mu,\nu)$ be  a competitor for  \eqref{OT3}, and let $\rho$ be its third marginal. Then, by Lemma \ref{equival},
we know that $\rho\in\A(\mu,\nu)$,  while $ \iiint  \pairing{z-x, x} \, \pi(dxdydz)=\iiint \pairing{z-y, y} \,  \pi(dxdydz) =0$
by particularizing the equilibrium condition  \eqref{equilibrium} for $\Phi= \Psi= \ident$. Thus, recalling the formula for the cost $c(x,y,z)=\frac1{2} (|z-x|^2 + |z-y|^2)$, we have,
\begin{align*}  \iiint& c(x,y,z) \, \pi(dxdydz)\\
	&= \frac1{2} \big(\var(\mu)+\var(\nu)\big) + \var(\rho) -\iiint  \pairing {z,x+y}\, \pi(dxdydz) \\
&=      \var(\rho)- \frac1{2} \big(\var(\mu)+\var(\nu)\big) .
\end{align*} 
The equality \eqref{JJV} then follows from \eqref{Sigma=A} by noticing that  taking the infimum with respect to $\pi\in \Sigma(\mu,\nu)$ on the left hand side above amounts to taking the infimum with respect to $\rho\in\A(\mu,\nu)$ in the last line. 
As a consequence, we see that   $\iiint c(x,y,z) \, \pi(dxdydz) = \JJ(\mu,\nu)$ if only if  $\var(\rho) =\V(\mu,\nu)$.
That proves the assertion  (i) of Theorem \ref{thm2-intro}. The assertion (ii) is a direct consequence of Proposition \ref{duality-2convex} (existence of optimal $\rho$) and of \eqref{Sigma=A} (existence of $\pi\in \Sigma(\mu,\nu)$
with the third marginal $\rho$).  
\qed

\subsection{Proof of Theorem \ref{main-intro}}

The equality $\I(f) = \JJ(\mu,\nu)$ and the existence of an optimal  $\pi\in \Sigma(\mu,\nu)$ have been already established (see Theorem \ref{thm2-intro}). The existence of an optimal $u$ solving \eqref{secondorder}
follows from Proposition \ref{gap=0} applied to $f = f_0 = \nu-\mu$.
Before proving the assertion (ii), we  recall  that if $u$ is admissible for \eqref{secondorder}, then  by integrating \eqref{ineq-3points}  with respect to any $\pi\in \Sigma(\mu,\nu)$ and by taking the relations \eqref{equilibrium} into account, we get
\begin{align*}
	\iiint & c(x,y,z)\, \pi(dxdydz)  \\
	&\geq \iiint \Big(\left[u(y) + \pairing{\nabla u(y), z-y}\right] -\left[u(x) + \pairing{\nabla u(x), z-x} \right] \Big) \pi(dxdydz) \\
	&= \int u \,d\nu- \int u\, d\mu.
\end{align*}
Therefore, since $\I(\nu-\mu)= \mathcal{J}(\mu,\nu)$, the optimality of $(u,\pi)$ is equivalent to the fact that the above inequality is an equality.
In view of {Lemma \ref{lipConv}}, this happens  if and only if the three-point condition \eqref{ineq-3points} holds true $\pi$-a.e. 
\qed

\subsection{Proof of Corollary \ref{coro1-intro}}  Let $\pi \in \Sigma(\mu,\nu)$ be an admissible 3-plan for \eqref{OT3},
and let us consider the  associated tensor valued measure, namely $\sigma= \iiint \sigma^{x,y,z} \, \pi(dxdydz)$. 
We claim that,
\begin{equation}\label{si-admi}
\int \varrho^0(\sigma)  \le  \iiint c(x,y,z)\, \pi(dxdydz),\qquad  \dive^2 \sigma = \nu-\mu \ \ \text{in $\D'(\R^d)$}\, .
\end{equation}
Then, if $\pi$ is optimal for \eqref{OT3}, we  will deduce that,
$$ \I'(\nu-\mu)\  (=\min \eqref{stress}) \, \le \int \varrho^0(\sigma) \, \le \iiint c(x,y,z)\, \pi(dxdydz) = \JJ(\mu,\nu) ,$$
hence the optimality of $\sigma$ since we have {$\I'(\nu-\mu) =\I(\nu-\mu)=\JJ(\mu,\nu)$} by virtue of Proposition \ref{gap=0}
and Theorem \ref{main-intro}.

Let us now prove \eqref{si-admi}. By  the subadditivity property of the convex one-homogenous functional
$\Mes(\R^d;\Sdd) \ni \sigma\mapsto \int \varrho^0(\sigma)$, we have
$$  \int \varrho^0(\sigma) \, \le \iiint \left(\int \varrho^0(\sigma^{x,y,z})\right) \pi(dxdydz) \,
\le \iiint c(x,y,z)\, \pi(dxdydz) .$$
Indeed, recalling the definition of the rank-one measure $\sigma^{x,y,z}$ given in \eqref{si-xyz}, we have
$$   \int \varrho^0(\sigma^{x,y,z}) \le  \int\limits_{[z,x]} \!\abs{\xi-z}\, \Ha^1(d\xi) + \int\limits_{[z,y]} \!\abs{\xi-z}\, \Ha^1(d\xi) 
= \frac1{2} (|x-z|^2 + |y-z|^2),  $$
with the inequality being an equality if the  segments $[x,z]$ and $ [y,z]$ do not overlap.
Finally, let us show  that $\sigma$ satisfies the distributional constraint $\dive^2 \sigma= \nu-\mu$. 
Recalling that $ \dive^2 \sigma^{x,y,z} =f^{x,y,z}$, where $f^{x,y,z}:= \delta_y - \delta_x - \dive\big((z-y)\,\delta_y - (z-x)\,\delta_x \big)$ (see \eqref{f-xyz}), for each test function $\f\in \D(\R^d)$ we have,
\begin{align*} &\pairing{\f,\dive^2 \sigma} =  \iiint\pairing{\f,\dive^2 \sigma^{x,y,z}} \, \pi(dxdydz) = 
 \iiint\pairing{\f, f^{x,y,z}} \, \pi(dxdydz)\\
&\quad= \ \iiint \Big(\f(y)-\f(x) + \pairing{\nabla\f(y),z-y} - \pairing{\nabla\f(x),z-x}   \Big) \, \pi(dxdydz) \\
&\quad= \ \int \f \, d\nu - \int \f\, d\mu  ,   
\end{align*} 
where the last equality relies on the relations \eqref{equilibrium}.
This proves our claim \eqref{si-admi}, hence the first assertion of Corollary \ref{coro1-intro}. 
Let us now consider the marginal $\gamma= \pi_{1,2}:=\Pi_{1,2}^\#(\pi) $ of an admissible $\pi\in \Sigma(\mu,\nu)$ with respect to the first two coordinates.
There is no loss of generality in assuming that  $\iiint c\, d\pi <+\infty$, which allows to deduce that $\pi\in  \PP_2((\R^d)^3)$.
Then, there exists a $\gamma$-measurable family $\{{\pi^{x,y}}\}$ in $\PP_2(\R^d)$ satisfying the disintegration formula $\pi(dxdydz) = \gamma(dxdy) \otimes \pi^{x,y}(dz)$, see the convention \eqref{eq:convention}. It yields,
$$\iiint \alpha(x,y,z) \, \pi(dxdydz) = \iint \pairing {\alpha(x,y,\argu),{\pi^{x,y}}} \, \gamma(dxdy) \quad \ \ \forall \alpha\in X_2((\R^d)^3) .$$
Let us apply this formula to the following element of $X_2((\R^d)^3)$,
$$\alpha_u (x,y,z) := \left[u(y) + \pairing{\nabla u(y), z-y}\right]-  \left[u(x) + \pairing{\nabla u(x), z-x}\right]  - c(x,y,z)
 ,$$
where $u$ is admissible for \eqref{secondorder}. By \eqref{ineq-3points} we have $\alpha_u \le 0$, while, by virtue of 
the second assertion of Theorem \ref{main-intro},    $\alpha_u=0$
holds $\pi$-a.e. whenever the pair $(u,\pi)$ is optimal.  In this case, we get
$$  0 = \iiint \alpha_u(x,y,z) \, \pi(dxdydz)  = \iint \pairing {\alpha_u(x,y,\argu),{\pi^{x,y}}} \, \gamma(dxdy) ,$$
yielding that $\spt( {\pi^{x,y}}) \subset \{z : \alpha_u(x,y,z)=0\}$ for $\gamma$-almost all $(x,y)\in (\R^d)^2$. 
Next, we show that the subset  $\{\alpha_u(x,y,\cdot)=0\}$ reduces to the singleton $\{z_u(x,y)\}$ where,
\begin{equation}\label{zu}
z_u(x,y) =  \frac{x+y}{2} + \frac{\nabla u(y)-\nabla  u(x)}{2} .
\end{equation}
 For fixed  $(x,y)$ the function $z\to \alpha_u(x,y,z)$ is strictly concave; hence, it reaches its maximum on $\R^d$ at the unique point $z_u(x,y)$ where  $\partial_z \alpha_u(x,y,z) = \nabla u (y) -\nabla u(x) - (2 z -(x+y))$ vanishes.
This furnishes \eqref{zu}. Since $\nabla u$ is $1$-Lipschitz, $z_u(x,y)$ belongs to the ball $B(\frac{x+y}{2},\frac{|x-y|}{2})$.
Accordingly,  any optimal transport plan $\ov\pi$ is supported on {$(\B(\spt \mu, \spt \nu))^3$},
so the associated optimal tensor measure $\ov \sigma = \iiint \sigma^{x,y,z} d\ov\pi$ satisfies   \eqref{spt-sigma}.
The proof of the assertion (ii) is now complete.  
\qed

%
%
%
%
\bigskip
\section{Examples} \label{examples}
In this section we give exact solutions for some classes of data $\mu, \nu$. In each case we propose a pair $(u,\pi)$ and prove its optimality  by checking the optimality condition (ii) in Theorem \ref{main-intro}. It turns out that, after checking the three-point equality \eqref{3-eq}, the main challenge is to check  the admissibility conditions $-\mathrm{Id} \leq \nabla^2 u \leq \mathrm{Id}$ and $\pi \in \Sigma(\mu,\nu)$.
Once the optimality of $(u,\pi)$ is proved,  an optimal convex dominant $\rho$ is computed as the third marginal of $\pi$, see Theorem \ref{thm2-intro}. Meanwhile,  according to the Corollary \ref{coro1-intro},  a solution of the second-order Beckmann  problem \eqref{stress} of the form $\sigma= \iiint \sigma^{x,y,z}\,\pi(dxdydz)$ is derived.

\subsection{Ordered measures}
The simplest class of data is the one of $\mu,\nu \in \PP_2(\Rd)$ that are in  convex order. Let us assume that
\begin{equation*}
	\mu \preceq_c \nu.
\end{equation*}
Then, for any martingale transport plan $\gamma \in  MT(\mu,\nu)$,  an optimal pair $(u,\pi)$ is given by,
\begin{equation} \label{optichoice}
	u(x) = \frac{1}{2} \abs{x}^2, \qquad \pi(dxdydz) = \gamma(dxdy) \otimes \delta_y(dz),
\end{equation}
see the convention \eqref{eq:convention}. Recall that $MT(\mu,\nu)$ is non-empty  by virtue of Strassen theorem.

Admissibility of $u$ is clear, and $\pi \in \Sigma(\mu,\nu)$ follows easily from Lemma \ref{Sigmamunu=Amunu}. Due to the form of $\pi$, the three-point optimality condition \eqref{3-eq} has to be checked merely for the triples $(x,y,z)$ with $z=y$.
  This is automatic since  $u$ satisfies the identity\footnote{
The left  hand side is nothing else but the Bregman divergence of $u$ at $y$ around $x$.},
 $$  u(y) - [ u(x) + \pairing{\nabla u(x), y-x}]  = \frac1{2} |x-y|^2 \, .\quad $$ 
 With the validated optimality of the pair $(u,\pi)$,  we can deduce the minimal energy,
 $$\mathcal{I}(\nu-\mu) = \int u\, d(\nu-\mu) = \frac{1}{2}\big(\var(\nu) - \var(\mu)\big).$$
Moreover, the solution $\sigma$ provided by Corollary \ref{coro1-intro} takes the form \linebreak $\iint \sigma^{x,y,y} \,\gamma(dxdy)$
where, by \eqref{si-xyz}, $\sigma^{x,y,y}$  is positive semi-definite, thus   $\sigma \in \Mes(\Rd;\Sddp)$. 
 Eventually, in view of the property (p2) (in Section \ref{MTsurv}), we see that $\rho=\nu$ is the unique  minimizer of the optimal convex dominance problem $\V(\mu,\nu)$. 
 In contrast, the solution $\sigma$ to \eqref{stress} is not unique, as is shown in the remark below. 
 Our argument will be based on  the simple criterium as follows:
 \begin{proposition}\label{positive}
	Assume that $\mu \preceq_c \nu$. Then, a measure $\sigma \in \Mes(\Rd;\Sdd)$ satisfying the constraint $	 \dive^2 \sigma = \nu - \mu$ solves the second-order Beckmann problem \eqref{stress} if and only if it is positive semi-definite.
\end{proposition}
\begin{proof}
	Using the integration by parts formula \eqref{Rd-byparts}, for any $\sigma$ satisfying $\dive^2 \sigma = f  = \nu - \mu$, we have
	\begin{equation*}
		\int \varrho^0(\sigma) \geq  \int \tr\,\sigma  = \pairing{\nabla^2 u,\sigma} = \pairing{u,f} = \mathcal{I}(f),
	\end{equation*}
	where $u = \frac{1}{2} \abs{\argu}^2$ is an optimal potential according to \eqref{optichoice}.
	By the zero-gap result \eqref{gap=0},  the tensor measure $\sigma$ is optimal for \eqref{stress} if and only if 
	$\int \varrho^0(\sigma) =\I(f)$. This means that  the above inequality is an equality. Noticing that $\varrho^0(A) = \tr\,A$ for $A\in\Sdd$ implies that all the eigenvalues  of $A$ are non-negative,
	we infer that an admissible $\sigma$ is optimal if and only if it is an element of $ \Mes(\R^d;\Sddp)$. 
	\end{proof}


\begin{remark}[\emph{the non-uniqueness issue}] In general,  even if $\rho$ is unique, one can expect
that $\pi$ given in \eqref{optichoice} is not unique since there may exist multiple martingale transports  $\gamma\in MT(\mu,\nu)$.  In turn, this translates to possibly multiple  optimal tensor measures $\sigma$.
In fact, we can exploit Proposition \ref{positive} to see that non-uniqueness of optimal $\sigma$ goes beyond the one induced by the non-uniqueness of $\pi$.

 Let us consider the simple example when $\mu = \delta_0$ and $\nu = \sum_{i=1}^4 \frac{1}{4} \delta_{y_i}$ where $y_i$ are corners of the square centred at the origin. Clearly $\mu \preceq_c \nu$, and $\gamma =  \sum_{i=1}^4 \frac{1}{4} \delta_{(0,y_i)}$ is the unique element of $MT(\mu,\nu)$. It follows that $\Sigma(\mu,\nu)$ is a singleton, which gives uniqueness of optimal $\pi$. The induced minimizer $\sigma$  is the rank-one tensor measure defined as follows,
\begin{equation*}
	\sigma(d\xi) = \iiint \sigma^{x,y,z}(d\xi)\,\pi(dxdydz) = \sum_{i=1}^4   \frac{\abs{\xi-y_i}}{4} \, \frac{y_i}{\abs{y_i}} \otimes  \frac{y_i}{\abs{y_i}}  \, \Ha^1(d\xi) \mres[0,y_i].
\end{equation*}
Such $\sigma$ is demonstrated in Fig. \ref{fig:1vs4}(a). More accurately, the figure shows  the density of $\varrho^0(\sigma)$ with respect to $\Ha^1$ measure restricted to the four segments.

Meanwhile, the set of $\sigma \geq 0$ for which $\dive^2 \sigma = \nu-\mu$ is very rich. Figs \ref{fig:1vs4}(b,c) give examples of such measures. After Proposition \ref{positive}, they are also optimal for the second-order Beckmann problem \eqref{stress}.  It is even  possible to find optimal $\sigma$ that has an absolutely continuous part.
This example not only shows that we may experience great flexibility in the choice of optimal $\sigma$ but also that not every such optimal measure can be decomposed with respect to a three-point measure $\pi \in \Sigma(\mu,\nu)$ as  in Corollary \ref{coro1-intro}. This is a significant difference with respect to the classical first-order Beckmann problem  where 
all minimizers can be decomposed along transport rays by virtue of Smirnov theorem (see \cite{Smirnov}  and  Proposition 2.3 in \cite{Dweik}).
\begin{figure}[h]
	\centering
	\subfloat[]{\includegraphics*[trim={0cm 0cm -0cm -0cm},clip,width=0.2\textwidth]{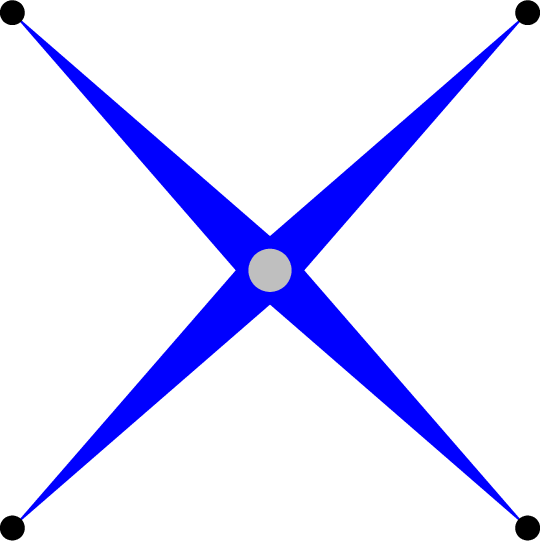}}\hspace{1.5cm}
	\subfloat[]{\includegraphics*[trim={0cm 0cm -0cm -0cm},clip,width=0.2\textwidth]{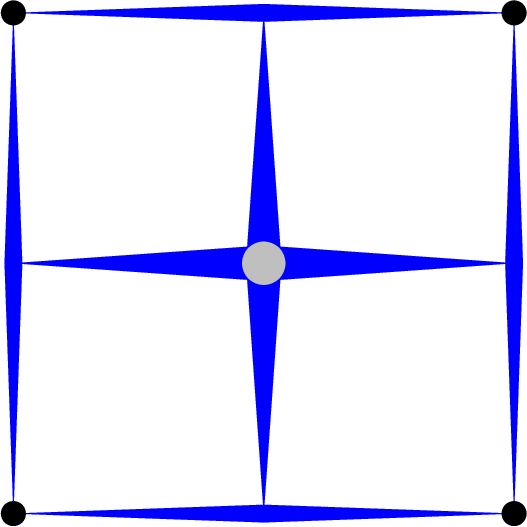}}\hspace{1.5cm}
	\subfloat[]{\includegraphics*[trim={0cm 0cm -0cm -0cm},clip,width=0.2\textwidth]{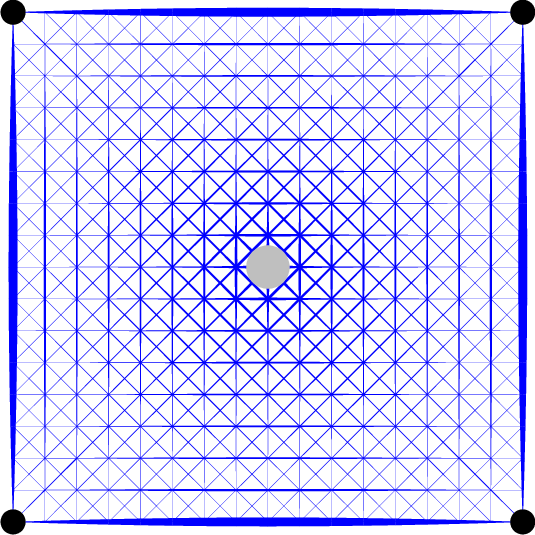}}
	\caption{Various optimal $\sigma$ (blue) for the data $\mu = \delta_0$ (gray) and $\nu = \sum_{i=1}^4 \frac{1}{4} \delta_{y_i}$ (black). Only the density of the 1D measure $\sigma$ is displayed.}
	\label{fig:1vs4}       
\end{figure}

\end{remark}

\subsection{Gaussian measures}
\label{ex:Gaussian}
In this example we assume the data to be two centred Gaussian distributions on $\Rd$,
\begin{equation*}
	\mu = \mathcal{N}(0,M), \quad \nu = \mathcal{N}(0,N),
\end{equation*}
where $M,N \in \Sddp$ are two positive semi-definite covariance matrices.
Note that if these matrices are ordered,  we find ourselves in the framework of the former example (see the comment after 
(p2) in Section \ref{MTsurv}). 
In the general case, at the core of the solution lies the spectral decomposition of the difference of the covariance matrices,
\begin{equation*}
	N-M =  \sum_{i=1}^d \lambda_i\, a_i \otimes a_i,
\end{equation*}
where $a_i$ are mutually orthogonal vectors on the unit sphere $S^{d-1}$. Let us define the projection matrices,
\begin{equation*}
	P_- := \sum_{\{i \, : \, \lambda_i < 0\} }  a_i \otimes a_i, \qquad 	P_+ := \sum_{\{i \, : \, \lambda_i \geq 0\} } a_i \otimes a_i = \mathrm{Id} - P_-.
\end{equation*}
The following symmetric positive semi-definite matrices will prove to be essential:
\begin{align*}
	M \vee N & := M + (N-M)_+ = N + (M-N)_+, \\ 
	M \wedge N &:= M - (M-N)_+ = N - (N-M)_+,
\end{align*}
where,
$$(N-M)_+ =  \sum_{i=1}^d (\lambda_i)_+  a_i \otimes a_i, \qquad (M-N)_+ =   \sum_{i=1}^d (\lambda_i)_-\, a_i \otimes a_i .$$
According to Remark \ref{wedge-vee}, $M \vee N$ can be seen as the least majorant of the matrices $M,N$, and $M\wedge N$ as their greatest minorant.


We are going to now show  that an optimal pair $(u,\pi)$ is given by,
\begin{equation*}
	u(x) = \frac{1}{2} \sum_{i=1}^d \mathrm{sgn}(\lambda_i) \pairing{a_i,x}^2,  \qquad  	\pi = \gamma(dxdy) \otimes \delta_{z_u(x,y)}(dz),
\end{equation*}
where we agree to the convention that $\mathrm{sgn}(0) = 1$, while:
\begin{enumerate}
\item[-] the transport plan $\gamma \in \PP(\Rd \times\Rd)$ is  a normal distribution,
\begin{equation*}
	\gamma = \mathcal{N}\left(0,G \right), \qquad G = \begin{bmatrix}
		M & M \wedge N \\
		M \wedge N  & N
	\end{bmatrix};
\end{equation*}

\item[-] the function $z_u$ is computed according to \eqref{zbar}, which here leads to,
\begin{equation*}
	z_u(x,y) =  P_- x + P_+ y.
\end{equation*}

\end{enumerate}
The positive semi-definiteness of $G$ is clear since $M \wedge N$ is a minorant for both $M$ and $N$.
Since $\nabla^2 u =  \sum_{i=1}^d \mathrm{sgn}(\lambda_i)\, a_i \otimes a_i$, feasibility of $u$ is also straightforward. In view of the disintegrated form of $\pi$, it is sufficient to show that the equality \eqref{3-eq} holds for every triple $\big(x,y,z_u(x,y)\big)$ where $(x,y)$ ranges in the whole $(\Rd)^2$. This reduces to a tedious but elementary computation.

The more involved part is showing the admissibility $\pi \in \Sigma(\mu,\nu)$. As the first and second marginals of $\pi$ coincide with those of $\gamma$, they are equal to $\mu$ and $\nu$, respectively. Thus, by virtue of Lemma \ref{Sigmamunu=Amunu}, it is enough to show that the marginals $\pi_{1,3} := \Pi^\#_{1,3} (\pi)$ and $\pi_{2,3} := \Pi^\#_{2,3} (\pi)$ are martingale plans.
Integrating against a test function $\phi \in \mathrm{C}_0(\Rd \times \Rd)$, we obtain
\begin{align*}
	\iint \phi(x,z) \,\pi_{1,3}(dxdz) = \iiint \phi(x,z) \, \pi(dxdydz) &= \iint \phi\big(x,z_u(x,y)\big) \, \gamma(dxdy) \\
	 = \iint \phi\big(x,x + P_+(y-x)\big) \, \gamma(dxdy) &= \iint \phi(x,x+z)  \,\hat\gamma(dxdz).
\end{align*}
Above $\hat{\gamma}$ is the push forward of $\gamma$ through the map
$A(x,y) = (x,z) = (x, P_+(y-x))$. As $A$ is linear, it might be identified with a $2d \times 2d$ matrix. Accordingly, $\hat{\gamma}$ is another Gaussian given by,
\begin{equation*}
	\hat{\gamma} = \mathcal{N}(0,\hat{G}), \qquad \hat{G} = AGA^\top =  \begin{bmatrix}
		M & 0\\
		0   &(N-M)_+
	\end{bmatrix}.
\end{equation*}
Note that the matrix multiplication above is straightforward once we observe that,
$$(M\wedge N) P_+ = M P_+, \qquad (M\wedge N) P_- = N P_-.$$
 The structure of the matrix $\hat{G}$ shows that $\hat{\gamma}$ is a product of two Gaussians:  $\hat{\gamma} = \mathcal{N}(0,M) \otimes \mathcal{N}(0,(N-M)_+)$. We continue the chain of equalities,
\begin{align*}
	\iint \phi(&x,z) \,\pi_{1,3}(dxdz)\\
	&=  \int \left( \int \phi(x,x+z) \,\mathcal{N}\big(0,(N-M)_+\big)(dz)\right) \mathcal{N}(0,M) (dx) \\
	& =  \int \left( \int \phi(x,z) \,\mathcal{N}\big(x,(N-M)_+\big)(dz)\right) \mathcal{N}(0,M) (dx),
\end{align*}
in order to arrive at,
\begin{equation}\label{pi13}
	\pi_{1,3}(dxdz) = \mu(dx) \otimes \mathcal{N}\big(x,(N-M)_+\big)(dz).
\end{equation}
It is clear that $\pi_{1,3}$ is a martingale. In a similar way one shows that $	\pi_{2,3} = \nu \otimes \mathcal{N}\big(y,(M-N)_+\big)$,
which is also a martingale. We have thus proved that $\pi \in \Sigma(\mu,\nu)$ and, ultimately, that $(u,\pi)$ are optimal. The minimal energy equals,
\begin{equation*}
	\mathcal{I}(\nu - \mu) = \int u\, d(\nu - \mu) = \frac{1}{2} \sum_{i=1}^d \mathrm{sgn}(\lambda_i)\, \pairing{N-M,a_i \otimes a_i}  = \frac{1}{2} \varrho^0(N-M).
\end{equation*}

To identify the optimal measure $\rho$ we compute the third marginal of $\pi$. 
Utilizing the disintegration formula \eqref{pi13} for $\pi_{1,3}$ we find that it is a convolution of two Gaussians,
\begin{equation*}
	\rho = \pi_3 = \mu \star \mathcal{N}\big(0,(N-M)_+\big) =  \mathcal{N} \big(0,M+(N-M)_+\big) = \mathcal{N} (0,M \vee N)
§\end{equation*}
(note that we obtain the same result  when computing the second marginal of $\pi_{2,3}$).

\begin{remark}\label{wedge-vee}
	It is possible to show directly that  $\ov \rho:=\mathcal{N} (0,M \vee N)$  is a solution to the minimal variance problem \eqref{min-var}. Indeed, since $\ov \rho$ satisfies the dominance constraints (cf. (p2) in Section \ref{MTsurv}), we have $\V(\mu,\nu) \le \tr(M\vee N)$.
In the opposite direction, any admissible $\rho\in \A(\mu,\nu)$ admits a covariance matrix $R\in \Sddp$ such that  $R \geq M, \ R \geq N$. Therefore, since $\var(\rho) = \tr\, R$, we have
	 \begin{equation*}
	\V(\mu,\nu) \ \ge \ 	\min_{R \in \Sddp} \Big\{ \tr \, R \, : \, R \geq M, \ R \geq N \,\Big\}.
	\end{equation*}
It is not difficult to check that the right hand side above is  a semi-definite program	which admits a unique solution given by $R= M \vee N$. The optimality of $\ov \rho$ follows.
 Notice that, similarly, the matrix $M \wedge N$ uniquely solves the analogous maximization problem where the convex order constraints are reversed. In this sense, $M \vee N$ is the least majorant of the matrices $M,N$, whilst $M\wedge N$ is their greatest minorant.
\end{remark}

\subsection{Two-point measures }
\label{ex:2vs2}

The simplest non-trivial data possible is when both measures are supported by two points,
\begin{equation*}
	\mu = \sum_{i=1}^2 \mu_i \,\delta_{x_i}, \qquad \nu = \sum_{j=1}^2 \nu_j \,\delta_{y_j}.
\end{equation*}
As the barycentres must coincide, the problem is virtually planar. We can thus \textit{a priori} assume that $d = 2$. In addition, we enforce that the four points are not aligned so that 1D scenario is avoided.

As before, we assume that the measures are centred, i.e. $[\mu] = [\nu] =0$. In this case $x_1 = - \frac{\mu_2}{\mu_1} x_2$,  $y_1 = - \frac{\nu_2}{\nu_1} y_2$. Note that the weights follow automatically from the positions,
\begin{equation}
	\label{eq:weights}
	\mu_i =\frac{\abs{x_{i'}}}{\abs{x_1}+\abs{x_2}}, \qquad  \nu_j = \frac{|y_{j'}|}{ \abs{y_1}+\abs{y_2}},
\end{equation}
where $i' =3 -i$, $j'=3-j$.

The main challenge lies in the fact that the type of the solution switches depending on the geometrical property of the convex quadrilateral formed by the points $x_1,y_2,x_2,y_1$. Indeed, the two cases below must be considered:
\begin{enumerate}
	\item [(A)] the pairs of opposite edges of the quadrilateral are inclined at an angle non-greater than $\pi/2$;
	\item [(B)] the angle between one of the pairs of opposite edges exceeds $\pi/2$.
\end{enumerate}
The pairs of lines extending the edges in questions are drawn in Fig. \ref{fig:2vs2}(a). In fact, being in the scenario (A) is equivalent to the system of two inequalities: 
\begin{subequations}
	\label{eq:angles}
	\begin{align}
		\label{eq:angles1}
		&\pairing{x_2-y_2,y_1-x_1} \geq 0,\\
		\label{eq:angles2}
		&\pairing{x_1-y_2,y_1- x_2} \geq 0.
	\end{align}
\end{subequations}
It is worth emphasizing that at least one of those inequalities is always met.

 \noindent\underline{Case (A)}

To an extent, this case is similar to the Gaussian example as again the spectral decomposition of the difference of the covariance matrices will play the central role. Defining $M= \int x \otimes x \, \mu(dx)$ and $N= \int y \otimes y \, \nu(dy)$ we can make use of \eqref{eq:weights} to show that,
\begin{equation}
	\label{eq:MN_formulas}
	M = - x_1 \otimes x_2 = - x_2 \otimes x_1 , \qquad N= -y_1 \otimes y_2=-y_2 \otimes y_1.
\end{equation}
Since we assumed that the four points are not collinear, the difference always has two eigenvalues of opposite signs,
\begin{equation*}
	N-M = \lambda_a \, a \otimes a + \lambda_b \, b \otimes b, \qquad \lambda_a <0, \ \  \lambda_b>0,
\end{equation*}
where $a \perp b$, and $a,b \in S^1$.
In what follows we prove that in the case (A) the problems $\mathcal{I}(\nu-\mu)$ and $\mathcal{J}(\mu,\nu)$ are solved by, respectively,
\begin{equation}
	\label{upi_caseA}
	u(x) = \frac{1}{2} \big( \pairing{b,x}^2 -\pairing{a,x}^2   \big), \qquad \pi = \sum_{i,j=1}^{2} \gamma_{ij}\, \delta_{(x_i,y_j,z_{ij})},
\end{equation}
where,
\begin{equation}
	\label{eq:zij_rij}
		\gamma_{ij}= \mu_i \frac{\langle b,y_{j'} -x_i\rangle }{\langle b, y_{j'}-y_j \rangle}, \qquad z_{ij} =  \pairing{a,x_i} \, a + \pairing{b,y_j}\,b.
\end{equation}
We observe that $z_{ij} = z_u(x_i,y_j) = P_- x_i +P_+ y_j$ for $P_- = a\otimes a$, $P_+ = b \otimes b$. Accordingly, both admissibility of $u$ and  the three-point optimality condition \eqref{3-eq} can be shown identically as in Example \ref{ex:Gaussian}. The biggest challenge consists in showing that $\pi \in \Sigma(\mu,\nu)$. In fact, it is the positivity of $\gamma_{ij}$ that is the most delicate. The following result shows that it characterizes the case (A):
\begin{lemma}
	\label{lem:positive_rij}
	The  inequalities \eqref{eq:angles} hold true if and only if $\gamma_{ij} \geq 0 $ for all $i,j \in \{1,2\}$.
\end{lemma}
\noindent As the proof is rather long and technical, it is moved to Appendix \ref{app:examples}. We can readily check that $\pi \in \Sigma(\mu,\nu)$ relying on Lemma \ref{Sigmamunu=Amunu}. The fact that the first marginal of $\pi$ is $\mu$ amounts to observing that $\sum_{j=1}^2 \gamma_{ij} = \mu_i$. Next we compute,
\begin{align*}
	\pi_{1,3}  = \sum_{i,j=1}^2 \gamma_{ij} \, \delta_{(x_i,z_{ij})} = \sum_{i=1}^2 \mu_i \, \delta_{x_i}\!  \otimes p^i, \qquad p^i := \sum_{j=1}^2 \frac{\langle b,y_{j'} -x_i\rangle }{\langle b, y_{j'}-y_j \rangle}  \,\delta_{z_{ij}}.
\end{align*}
Noting that  $z_{ij} = x_i+ \pairing{b,y_j-x_i} \, b$, it is easy to show that $[p^i] = x_i$, rendering $\pi_{1,3}$ a martingale.

To show that $\pi_2 = \nu$ and that $\pi_{2,3}$ is martingale as well, we derive an alternative formula for $\gamma_{ij}$ that is symmetric to \eqref{eq:zij_rij}. First, observe that $\mu_i = \frac{\pairing{a,x_{i'}}}{\langle a, x_{i'} -x_i \rangle}$ thanks to \eqref{eq:weights}. This starts the chain of equalities below in which we exploit the equality
$\pairing{a \otimes b,M} = \pairing{a\otimes b,N}$ and formulas \eqref{eq:MN_formulas},
\begin{align*}
	\gamma_{ij} &= \frac{\pairing{a,x_{i'}}}{\langle a, x_{i'} -x_i \rangle} \frac{\langle b,y_{j'} -x_i\rangle }{\langle b, y_{j'}-y_j \rangle} =\frac{ \pairing{ a \otimes b, -x_{i'} \otimes x_i} +\pairing{a,x_{i'}} \langle b,y_{j'} \rangle }{\langle a, x_{i'} -x_i \rangle \langle b, y_{j'}-y_j \rangle}  \\
	&=\frac{ \pairing{a \otimes b,-y_{j} \otimes y_{j'},} +\pairing{a,x_{i'}} \langle b,y_{j'} \rangle }{\langle a, x_{i'} -x_i \rangle \langle b, y_{j'}-y_j \rangle} = \frac{ \langle b,y_{j'}\rangle }{ \langle b, y_{j'}-y_j \rangle} \frac{\pairing{a,x_{i'}-y_j}}{\langle a, x_{i'} -x_i \rangle } \\
	&= \nu_j\, \frac{\pairing{a,x_{i'}-y_j}}{\langle a, x_{i'} -x_i \rangle }.
\end{align*}
Readily, arguments put forward above for the marginals $\pi_1, \pi_{1,3}$ can be now reproduced for $\pi_2, \pi_{2,3}$. Admissibility $\pi \in \Sigma(\mu,\nu)$ is thus established and, hence, also the optimality of the pair $(u,\pi)$.

It remains to give the solutions of $\mathcal{V}(\mu,\nu)$ and of the second-order Beckmann problem \eqref{stress},
\begin{equation*}
	\rho = \pi_3  = \sum_{i,j=1}^2 \gamma_{ij} \, \delta_{z_{ij}}, \quad \ \ 	\sigma = \iiint \sigma^{x,y,z}\,\pi(dxdydz) = \sum_{i,j=1}^2 \gamma_{ij} \, \sigma^{x_i,y_j,z_{ij}}.
\end{equation*}

\noindent \underline{Case (B):}

It would be impractical to give a unified solution for all possible positions of the points that fall within the scope of the case (B). Instead, we shall assume that
$\pairing{x_1,y_1} \geq 0$ and $ \abs{x_1}\abs{y_2} \leq \abs{x_2} \abs{y_1}$. It is not restrictive as one can always relabel the points to guarantee it. Under those assumptions, one can easily observe that the inequality \eqref{eq:angles2} is automatically satisfied. Accordingly, the case (B) is characterized by the strict inequality,
\begin{equation}
	\label{eq:ineq_viol}
	\pairing{x_2-y_2,y_1-x_1} < 0.
\end{equation}

We start by defining the point $z_0 \in \R^2$ as the intersection of the two straight lines that extend the segments $[x_1,y_1]$ and $[x_2,y_2]$, see Figs \ref{fig:2vs2}(e,f). Let us endow the plane $\R^2$ with a polar coordinate system $x \mapsto \big(\varrho(x), \vartheta(x)\big) \in [0,\infty) \times [0,2\pi)$
where the pole  and the orientation of the system are fixed by,
\begin{equation*}
	\varrho(z_0) = 0, \quad  \vartheta(x_1) = 0, \quad \vartheta(x_2) \in (0,\pi).
\end{equation*}
Next, we define two coefficients:
\begin{equation*}
	\alpha = \frac{\pi}{ 2 \, \angle(x_2-y_2,y_1-x_1)}, \qquad \beta = \frac{\alpha}{4\alpha-1},
\end{equation*}
where $\angle$ is the angle between two vectors that \textit{a priori} ranges in $[0,\pi]$.
Under the assumption \eqref{eq:ineq_viol} we have $\alpha \in (\frac{1}{2},1)$ and $\beta \in (\frac{1}{3},\frac{1}{2})$. In particular, $\alpha \neq \beta$. In the polar coordinates the maximizer of $\mathcal{I}(\nu-\mu)$ is,
\begin{equation*}
	{\upsilon}(r,\theta) = \frac{1}{2}  h(\theta) \, r^2,
\end{equation*}
where,
\begin{equation*}
	h(\theta) = \begin{cases}
		h_1(\theta)=\cos(2\alpha \theta) & \text{if } \ \theta \in \big[2k \pi, 2k \pi + \pi/(2\alpha) \big),\\
		h_2(\theta) = \cos\big(2\beta (2\pi - \theta) \big) & \text{if } \ \theta \in \big[2k \pi + \pi/(2\alpha) , 2(k+1) \pi \big),
	\end{cases} 
\end{equation*}
where $k$ is any integer.
Finally, the following pair will be proved to solve the problems  $\mathcal{I}(\nu- \nolinebreak \mu)$ and $\mathcal{J}(\mu,\nu)$:
\begin{subequations}
	\label{eq:u_def}
	\begin{align}
		&u(x) := \upsilon\big( \varrho(x), \vartheta(x) \big), \\
		&\pi:= \nu_1 \, \delta_{(x_1,y_1,y_1)} + \mu_2 \, \delta_{(x_2,y_2,x_2)} +(\mu_1 - \nu_1) \, \delta_{(x_1,y_2,z_0)}.
	\end{align}
\end{subequations}
This time, the main difficulty  is to prove the admissibility of $u$. With the following lemma we see that it holds exactly in the case (B). The proof can be found in the Appendix \ref{app:examples}.
\begin{lemma}
	\label{lem:u} Assume that $\angle(x_2-y_2,y_1-x_1) \neq 0$. Then, the function $u$ in \eqref{eq:u_def} is an element of
	$ W^{2,\infty}_{\mathrm{loc}}(\R^2)$, whilst $u \notin C^2(\R^2)$ unless $\pairing{x_2-y_2,y_1-x_1}  = 0$.
	Moreover, the condition 
	\begin{equation*}
		-\mathrm{Id} \leq \nabla^2 u(x) \leq \mathrm{Id} \qquad \text{for a.e. $x \in \R^2$}
	\end{equation*}
	holds true if and only if $\pairing{x_2-y_2,y_1-x_1}  \leq 0$.
\end{lemma}

We move on to check the admissibility $\pi \in \Sigma(\mu,\nu)$. First, observe that $\mu_1 - \nu_1 = \nu_2 - \mu_2$ is non-negative thanks to the assumption $ \abs{x_1}\abs{y_2} \leq \abs{x_2} \abs{y_1}$. Then, checking that the first and second marginals of  $\pi$ are equal to $\mu$ and $\nu$, respectively, is straightforward. Prior to showing that $\pi_{1,3}, \pi_{2,3}$ are martingales we make an observation. By equality of the barycentres there holds 
$\int (x-z_0) \,\mu(dx) = \int (y-z_0) \,\nu(dy)$. In this particular case it leads to $ \mu_1\, (x_1 - z_0) - \nu_1 \, (y_1 - z_0) = 	\nu_2\,(y_2-z_0) - \mu_2\, (x_2-z_0)$. Both triples $(z_0,x_1,y_1)$ and $(z_0,y_2,x_2)$ are collinear, and the respective lines are never parallel (cf. Fig. \ref{fig:2vs2}(f)), so the vectors on each side of the equality must be zero. In turn, it generates the two equalities,
\begin{equation*}
	x_1 = \frac{\nu_1}{\mu_1} \, y_1  +  \frac{\mu_1 - \nu_1}{\mu_1} \,z_0, \qquad \quad y_2 = \frac{\mu_2}{\nu_2} \, x_2 + \frac{\nu_2 - \mu_2}{\nu_2} \, z_0.
\end{equation*}
By exploiting the first one, we check that $\pi_{1,3}$ is indeed a martingale,
\begin{align*}
	\pi_{1,3} &= \nu_1 \delta_{(x_1,y_1)} + \mu_2 \delta_{(x_2,x_2)} + (\mu_1-\nu_1) \delta_{(x_1,z_0)} \\
	&= \mu_1 \delta_{x_1} \otimes \left( \tfrac{\nu_1}{\mu_1} \, \delta_{y_1}  +  \tfrac{\mu_1 - \nu_1}{\mu_1} \, \delta_{z_0}  \right) + \mu_2 \, \delta_{x_2} \otimes \delta_{x_2}.
\end{align*}
Handling $\pi_{2,3}$ is similar. Ultimately,  $\pi \in \Sigma(\mu,\nu)$ is established.

It remains to check the three-point equality \eqref{3-eq}, and, in view of the form of $\pi$, it must be tested for the three triples $(x,y,z)$. The construction of $u$ ensures that,
\begin{equation}
	\label{eq:u_on_segments}
	u(\xi) = \frac{1}{2} \,\abs{\xi-z_0}^2 \quad \forall \, \xi \in L_1,  \qquad 	u(\xi) =- \frac{1}{2} \,\abs{\xi-z_0}^2 \quad \forall \, \xi \in L_2, 
\end{equation}
where $L_1$ and $L_2$ are the lines on which the triples $(z_0,x_1,y_1)$ and $(z_0,y_2,x_2)$ lie, respectively. As a result, one arrives at the following identities:
\begin{align}
	\label{eq:L1}
	u(\xi) - \big[u(x) + \pairing{\nabla u(x),\xi - x} \big]&= \tfrac{1}{2} \abs{\xi-x}^2 \qquad \forall \, \xi,x \in L_1, \\
	\label{eq:L2}
	\big[u(y) + \pairing{\nabla u(y),\xi - y} \big] - u(\xi)&= \tfrac{1}{2} \abs{\xi-y}^2 \qquad \forall \, \xi,y \in L_2.
\end{align}
We are ready to verify the condition \eqref{3-eq}. For the triple $(x_1,y_1,y_1)$ it reduces to \eqref{eq:L1} with $x= x_1$, $\xi = y_1$, while for $(x_2,y_2,x_2)$ to \eqref{eq:L2} with $y= y_2$, $\xi = x_2$. Finally, condition \eqref{3-eq} for the triple $(x_1,y_2,z_0)$ can be validated by adding equalities  \eqref{eq:L1} and  \eqref{eq:L2}, written for $x=x_1$, $\xi = z_0$ and, respectively, $y=y_2$, $\xi = z_0$.

Optimality of the pair $(u,\pi)$ is now established. Solutions to $\V(\mu,\nu)$ and to the second-order Beckmann problem \eqref{stress} read:
\begin{align*}
	\rho &= \pi_3 = \nu_1 \, \delta_{y_1} + \mu_2 \, \delta_{x_2} +(\mu_1 - \nu_1) \, \delta_{z_0}, \\
	 \sigma &= \nu_1 \, \sigma^{x_1,y_1,y_1} + \mu_2\, \sigma^{x_2,y_2,x_2} +  (\mu_1 - \nu_1)\, \sigma^{x_1,y_2,z_0}.
\end{align*}

\begin{figure}[h]
	\centering
	\subfloat[]{\includegraphics*[trim={0cm 0cm -0cm -0cm},clip,width=0.28\textwidth]{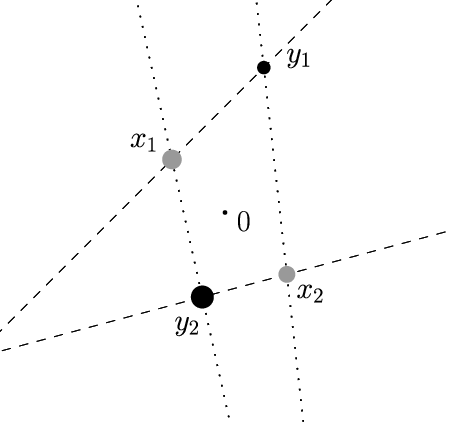}}\hspace{0.7cm}
	\subfloat[]{\includegraphics*[trim={0cm 0cm -0cm -0cm},clip,width=0.28\textwidth]{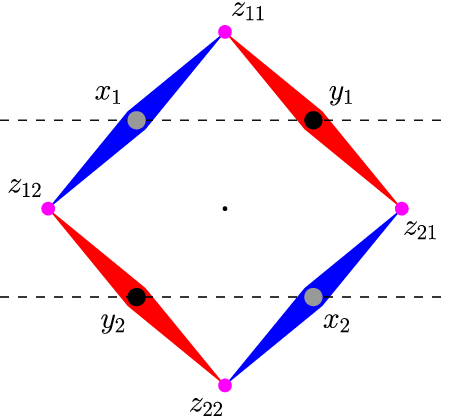}}\hspace{0.7cm}
	\subfloat[]{\includegraphics*[trim={0cm 0cm -0cm -0cm},clip,width=0.28\textwidth]{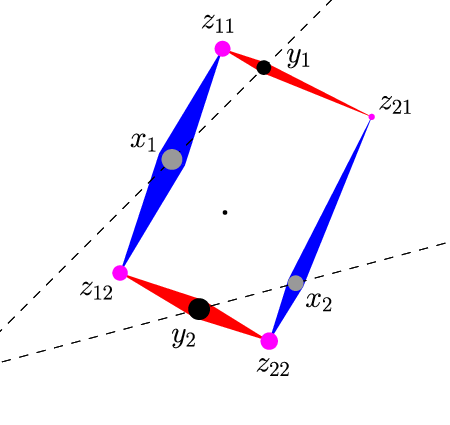}} \\
	\subfloat[]{\includegraphics*[trim={0cm 0cm -0cm -0cm},clip,width=0.28\textwidth]{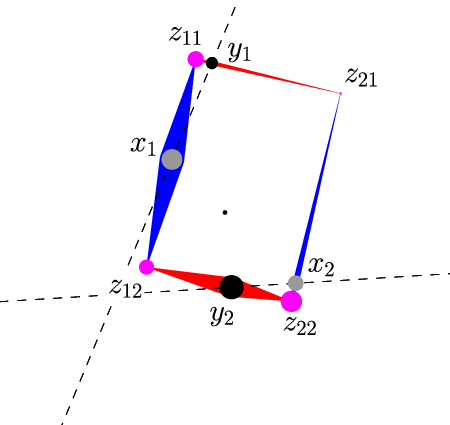}}\hspace{0.7cm}
	\subfloat[]{\includegraphics*[trim={0cm 0cm -0cm -0cm},clip,width=0.28\textwidth]{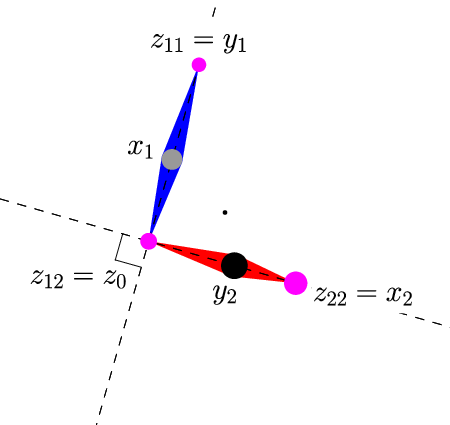}}\hspace{0.7cm}
	\subfloat[]{\includegraphics*[trim={0cm 0cm -0cm -0cm},clip,width=0.28\textwidth]{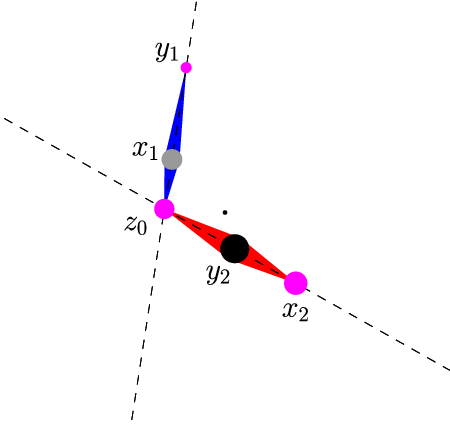}}
	\caption{Data $\mu$ (gray) and $\nu$ (black), optimal $\rho$ (magenta), and optimal $\sigma$ (blue and red for the positive and negative part). (a) generic data in the case (A); (b,c,d) solutions for various data in the case (A); (e) solution for the limit case; (f) solution for data in the case (B).}
	\label{fig:2vs2}       
\end{figure}

\medskip
For both cases (A) and (B), the solutions $\rho$ and $\sigma$ are displayed in Fig. \nolinebreak \ref{fig:2vs2}. The blue colour matches the segments where $\sigma$ is a positive semi-definite rank-one matrix, whilst the red colour matches the negative part.
Figs \ref{fig:2vs2}(b,c,d) correspond to the case (A) where the two lines form an acute angle. Fig. \ref{fig:2vs2}(f) demonstrates case (B) when this angle is obtuse. Finally, Fig. \ref{fig:2vs2}(e) shows the limit case for the right angle. In this case, the mass  at the point $z_{21}$ vanishes, and the solution adheres to the formulas given either for the case (A) or the case \nolinebreak (B).

\subsection{The basic first-order distribution data}

Unlike in the previous examples, here we shall consider a source which is not a measure but the 
 first-order distribution $f^{x,y,z}$ defined in the introduction. It is  supported on the two points $x,y \in \R^2$ and parametrized by the third point $z$,
\begin{equation*}
	f^{x,y,z}= \delta_y - \delta_x - \dive\big((z-y)\,\delta_y - (z-x)\,\delta_x \big).
\end{equation*}
To focus attention we shall assume that the vectors $x -z$ and $y-z$ form an angle ranging in $(0,\pi]$. That is to say that $x\neq y$, while $z$ cannot lie on the line crossing $x,y$ except on the open segment $]x,y[$.

As announced in the introduction, $f^{x,y,z}=\dive^2 \sigma^{x,y,z}$. Namely, it is the source term induced by the measure $\sigma^{x,y,z}$ that serves as an elementary block for building solutions $\sigma$ of the second-order Beckmann problem \eqref{stress} for sources that are measures. 
Since $\sigma^{x,y,z}$ is a competitor in the problem \eqref{stress} for the source $f = f^{x,y,z}$, it is natural to ask if it is optimal for such a basic  first-order distribution data. This short subsection is to settle this issue.

To that aim we exploit the construction of $u$ put forth in Example \ref{ex:2vs2}, case (B). With the polar coordinate system satisfying $\varrho(z) = 0$, $\vartheta(x) = 0$, $\vartheta(y) \in (0,\pi]$, we repeat the construction of $u$ with the parameter $\alpha = \pi/\big(2 \,\angle(x-z,y-z)\big)$. By the property that is analogous to 	\eqref{eq:u_on_segments}, one obtains,
\begin{align*}
	\pairing{u,f^{x,y,z}} &= 	\big[u(y) + \pairing{\nabla u(y),z - y} \big]  - 	\big[u(x) + \pairing{\nabla u(x),z - x} \big]  \\
	&= \tfrac{1}{2} \abs{x-z}^2 + \tfrac{1}{2} \abs{y-z}^2 = c(x,y,z).
\end{align*}
On the other hand, from the proof of Corollary \ref{coro1-intro} we also know that $\int \varrho^0(\sigma^{x,y,z}) = c(x,y,z)$.  Owing to the duality result in Proposition \ref{nogap-classic}, optimality of the pair $(u,\sigma^{x,y,z})$ will follow provided that $u$ is admissible. In view of Lemma \ref{lem:u}, it is the case only if $x-z$ and $y-z$ form an obtuse angle, which is to say that $z$ lies in the disk of the diameter $[x,y]$. We have arrived at the following result:
\begin{proposition}\label{optibasic}
	Assume that $z \in B\big(\frac{x+y}{2}, \frac{|x-y|}{2}\big)$. Then, $\sigma=\sigma^{x,y,z}$ solves the second-order Beckmann problem \eqref{stress} for the first-order distribution data $f = f^{x,y,z}$. Accordingly, we have the equality,
 $$	\mathcal{I}(f^{x,y,z}) = c(x,y,z) .$$
\end{proposition}

\begin{remark}
	The result above  is valid for $z =x$ (or for $z = y$).  In this case, 
	  $\sigma^{x,y,z}$ is negative (or positive) semi-definite, while the optimal potential is given by $u = - \frac{1}{2} \abs{\argu}^2$ (or $u = \frac{1}{2} \abs{\argu}^2$).
	  
	On the other hand, we stress the fact that $\sigma^{x,y,z}$ is no longer optimal if $z$ is outside of the disc $B\big(\frac{x+y}{2}, \frac{|x-y|}{2}\big)$. Indeed, in this case we can show that $\mathcal{I}(f^{x,y,z}) =  |z-\frac{x+y}{2} | \abs{x-y}$ which is strictly less than $c(x,y,z)$ for such $z$. An exception occurs when $z$ lies on the extension of the segment $[x,y]$. This is due to the cancelling effect between the positive and negative parts of $\sigma^{x,y,z}$.

\end{remark}

\section{The optimal grillage}\label{meca}

We conclude with a section devoted to an application of the results developed in this paper to optimal design in mechanics. Classically, by a \textit{grillage} one understands a planar multi-junction structure whose components are 1D straight bars. Although geometrically identical to \textit{trusses} \cite{bouchitte2008}, a grillage -- typically constituting a bearing structure of a ceiling -- lies in a horizontally oriented plane  and  it is loaded vertically at its junctions. The load causes the bars to bend rather than stretch, ultimately resulting in different equilibrium configurations for the two types of structures.

The optimal design of trusses is famously known to be ill-posed, calling for relaxation in the form of the \textit{Michell problem}  \cite{bouchitte2008,lewinski2019}. We will utilize Corollary \ref{coro1-intro} to prove that, in contrast, optimal grillages do exist provided that the load is a measure.
Despite the vast literature on grillage optimization initiated in  \cite{rozvany1972a}, it seems to be the first result of its kind. Before stating the theorem, we will briefly recall the topic of truss optimization. We will finish with two open problems, including the extension of the existence result to data that are  first-order distributions.

\subsection{Review on truss optimization and Michell problem}
A truss is a particular case of a 2D or 3D elastic solid that decomposes to one-dimensional straight bars. In general, the stress tensor in a solid can be described as a matrix valued measure $\sigma \in \Mes(\Rd;\Sdd)$. It must satisfy the equilibrium equation $-\dive \, \sigma = F$ in  $(\D'(\Rd))^d$ for a system of forces $F \in \Mes(\Rd;\Rd)$. For $\sigma$ to exist, the load $F$ has to be balanced in the following sense: ${\pairing{v_0,F}} =0$ whenever $v_0(x) = Ax + b$ for $b \in \Rd$ and a skew-symmetric $d \times d$ matrix $A$. By a truss we can understand the stress tensors that are of the form:
\begin{equation}
	\label{eq:truss}
	\sigma_\lambda = \iint \sigma^{x,y}   \, \lambda(dxdy), \qquad \lambda \in \Mes\big((\Rd)^2;\R\big),
\end{equation}
where $\sigma^{x,y} = \frac{y-x}{|y-x|} \otimes  \frac{y-x}{|y-x|} \, \Ha^1 \mres [x,y]$ for $x \neq y$, and $\sigma^{x,x} =0$. The positive and negative part  $\lambda_+(dxdy)$, $\lambda_-(dxdy)$  represent, respectively, the tensile and compressive forces in the bars $[x,y]$. 

Optimizing trusses amounts to looking for a measure $\lambda$ that, under the condition of equilibrating $F$, minimizes the total energy, cf. \cite{bouchitte2008}. Energy of a single bar $[x,y]$  that is subject to a unit tensile/compressive force is the total variation $\int \abs{\sigma^{x,y}}=\abs{y-x}$. Accordingly, the optimal truss problem reads,
\begin{equation}
	\label{eq:optimal_truss}
	 \inf\left\{ \iint \abs{y-x} \, \abs{\lambda}(dxdy) \, : \, \lambda \in \Mes\big((\Rd)^2;\R\big)  , \ \ -\dive\, \sigma_\lambda = F  \right\}.
\end{equation}
Note that the support of $\lambda$ can exceed the set $(\spt F)^2$, which is to say that we can add junctions that are not loaded.

In \eqref{eq:optimal_truss} the total mass of $\lambda$ is not controlled, raising the issue of existence. Moreover, in practice
engineers expect that for a finitely supported load $F$ there is a solution $\ov{\lambda}$ that  is finitely supported. This means that the structure can be manufactured as a junction of a finite number of bars. 
Meanwhile, already at the dawn of the 20th century, A.G.M. Michell observed that an optimal truss does not exist even for the simplest loads. In his celebrated paper \cite{michell1904} he considered the \textit{bridge problem} where the data is the system of three vertical forces in the plane $\R^2$,
\begin{equation}
	\label{eq:bridge_problem}
	F = \frac{e_2}{2}  \delta_{e_1}  + \frac{e_2}{2}  \delta_{-e_1}  - e_2  \delta_{0},
\end{equation}
where $e_1=(1,0)$, $e_2 = (0,1)$. Then, looking for finitely supported solutions of \eqref{eq:optimal_truss}  leads to construction of minimizing sequences $\lambda_h$ with the number of points in $\spt \lambda_h$ going to infinity. When taking the weak-* limit $\ov{\sigma}$ of the sequence $\sigma_{\lambda_h}$  one discovers that it is not representable through \eqref{eq:truss}, see Fig. \ref{fig:Michell}(a).
The measure $\ov\sigma$ is a solution of what today is known as the Michell problem,
\begin{equation}
	\label{eq:Michell}
	\min\left\{ \int \varrho^0(\sigma) \, : \, \sigma \in \Mes(\Rd;\Sdd), \ \ -\dive\, \sigma = F  \right\}.
\end{equation}
Recall that $\varrho^0$ is the Schatten norm: $\varrho^0(S) = \sum_{i=1}^d |\lambda_i(S)|$. In the modern measure-theoretic setting the Michell problem was first formulated in \cite{bouchitte2008}. Therein, it was proved that $\inf \eqref{eq:optimal_truss} = \min \eqref{eq:Michell}$.
Once a compactly supported $F$ satisfies the balance condition, the minimum in the Michell problem is attained. From Fig. \ref{fig:Michell}(a) one can discern that solutions may charge  curved curves (the thick lines in the figure). It rules out representing solutions through \eqref{eq:truss}. To address this, the work \cite{bouchitte2008} put forward another formulation where one seeks a signed measure on the space of regular curves, thus allowing for curved bars. To date, the existence issue remains open.

\begin{figure}[h]
	\centering
	\subfloat[]{\includegraphics*[trim={0cm 0cm -0cm -0cm},clip,width=0.32\textwidth]{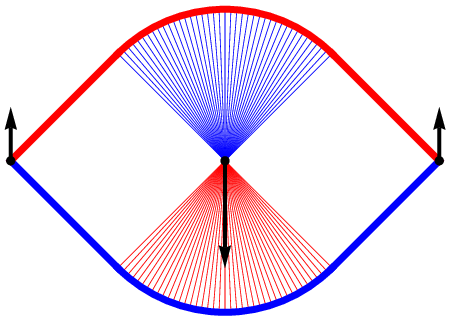}}\hspace{1.5cm}
	\subfloat[]{\includegraphics*[trim={0cm 0cm -0cm -0cm},clip,width=0.32\textwidth]{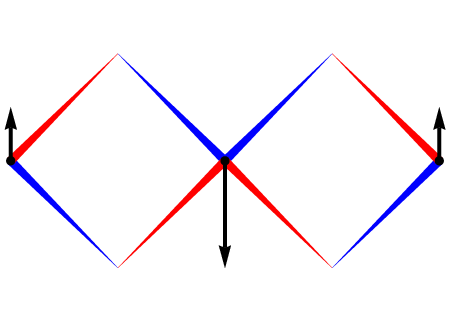}}
	\caption{(a) Michell structure for a finitely supported  system of forces $F$; (b) optimal grillage for a finitely supported torque $f = - \dive\, F$.}
	\label{fig:Michell}       
\end{figure}

\subsection{Optimal grillage via the three-marginal optimal transport}

The other example of a structure that is built from 1D bars is a grillage, and it is a special case of a \textit{plate}. By definition, plates are two-dimensional bodies occupying a horizontal plane $\R^d = \R^2$. In the case of plates, the measure $\sigma \in \Mes(\Rd;\Sdd)$ represents the \textit{bending moment tensor}. The out-of-plane equilibrium of the plates is governed by the equation $\dive^2 \sigma = f$ in $\D'(\Rd)$, where  $f=f_0 -\dive\,F$ is a first-order distribution. The measure $f_0$ models out-of-plane forces. One can think of the positive part $f_{0,+}$ as of the gravity pull, whilst $f_{0,-}$ plays the role of the upward reaction forces. The term $F$ represents torques that act about in-plane axes. The balance condition for the load $f$ reads as in \eqref{genload}.

With the second-order equilibrium equation, the decomposition of the measure $\sigma$ to segments allows for adding affinely varying density. One of the ways of achieving this is through using $\sigma^{x,y,z}$ as the basic measure. It concentrates on the union of segments $[x,z] \cup [z,y]$, see  {Fig. \ref{fig:sigmaxyz}.} Thus, by a grillage we will understand the bending moment tensor of the form,
\begin{equation}
	\label{eq:sigma_pi}
	\sigma_\pi = \iiint \sigma^{x,y,z} \, \pi(dxdydz), \qquad \pi \in \Mes_+\big((\Rd)^3 \big).
\end{equation}
In the case of grillages, $\pi(dxdydz)$ enjoys the interpretation of the \textit{transverse shear force} in the two-bar structure. 
Assuming that the segments $[x,z]$ and $[z,y]$ do not overlap, the energy of this structure is $\int \abs{\sigma^{x,y,z}} = \frac{1}{2}\big(\abs{z-x}^2 + \abs{z-y}^2\big) = c(x,y,z)$. Accordingly, the optimal grillage problem can be formulated  as follows,
\begin{equation*}
	\mathcal{I}_{\mathrm{OG}}(f) := \inf\left\{ \iiint c(x,y,z)\, \pi(dxdydz)\, : \, \pi \in \Mes_+\big((\Rd)^3 \big) , \ \ \dive^2 \sigma_\pi = f  \right\}.
\end{equation*}
Note that $\pi$ is a positive Borel measure that is not necessarily finite. In fact, the condition $\iiint c\,d\pi < \infty$ is sufficient for $\sigma_\pi$ to be a well defined element of the space $\Mes(\Rd;\Sdd)$.

\textit{A priori}, the optimal grillage problem shares the issues of non-compactness that are known for truss optimization. A natural candidate for relaxation is the second-order Beckmann problem \eqref{stress},
\begin{equation*}
	\I(f) = \min
	\biggl\{ \int \varrho^0(\sigma) \ : \ \sigma \in \Mes(\R^d;\Sdd), \ \ \dive^2 \sigma = f  \biggr\}.
\end{equation*} 
The solution is guaranteed to exist provided that $f_0 \in \Mes_2(\Rd)$, and $F \in \Mes_1(\Rd;\Rd)$ satisfy the balance condition \eqref{genload}, see Section \ref{dualclassic}.
Utilizing the subadditivity of the functional $\sigma \mapsto \int \varrho^0(\sigma)$ we can show that $\int \varrho^0(\sigma_\pi) \leq \iiint c \,d\pi$, which  furnishes the inequality,
\begin{equation}
	\label{eq:grillage>Beckmann}
	\mathcal{I}_{\mathrm{OG}}(f) \geq \I(f).
\end{equation}

Historically, the systematic study of optimal grillages was initiated in the engineering paper \cite{rozvany1972a}. Inspired by the theory of Michell structures, the author has tackled the Beckmann problem $\I(f)$ from the outset. However, in the numerous analytical examples worked out in \cite{rozvany1972a} and subsequent works, e.g. \cite{prager1977}, one can discern the grillage-like structure \eqref{eq:sigma_pi} of the optimal solutions $\ov{\sigma}$. Unlike in the Michell problem, the curved bars are not exhibited at optimality.

The next result lays out a foundation for the foregoing observations. Exploiting the novel three-marginal optimal transport formulation developed in this paper, we  show that the optimal grillage problem $\I_{\mathrm{OG}}(f)$ admits a solution when $f$ is a measure, i.e. when the first-order term $-\dive \, F$ is absent. On top of that, we prove that the optimal grillage consists of a finite number of bars once the load $f$ is discrete.

\begin{theorem}
	\label{cor:grillage}
	If the load distribution $f$ is a measure in $\Mes_2(\R^d)$, then the equality $\mathcal{I}_{\mathrm{OG}}(f) = \I(f)$ holds true together with the following statements.
	 	\begin{enumerate}[label={(\roman*)}] 
	 	\item There exists a solution $\ov{\pi}$ of the optimal grillage problem $	\mathcal{I}_{\mathrm{OG}}(f)$, and, for any such solution, $\sigma_{\ov\pi}$ solves the Beckmann problem $\I(f)$. Moreover, $\ov{\pi}$ can be chosen such  that $\ov\pi\big((\Rd)^3 \big) < \infty$, and 
	 	\begin{equation}
	 		\label{eq:inclusion}
	 		\spt \sigma_{\ov{\pi}} \subset  \B(\spt f_+, \spt f_-),
	 	\end{equation}
	 	where $f = f_+ - f_-$ is the Jordan decomposition to the positive and negative part.
	 	\item If, in addition, the measure $f$ is finitely supported, then one can choose a finitely supported solution $\ov\pi$. In particular, 
	 	\begin{equation*}
	 		\sigma_{\ov{\pi}} \ll \Ha^1 \mres G
	 	\end{equation*}
	 	where $G \subset \Rd$ is a graph consisting of at most $2 m n$ segments where $m,n$ is the cardinality of $\spt f_+$, $\spt f_-$, respectively.
	 \end{enumerate}
\end{theorem}
\begin{proof}
It is not restrictive to assume that $f = \nu -\mu$ for the probability distributions $\mu,\nu \in \mathcal{P}_2(\Rd)$ that are centred, $[\mu]=[\nu]=0$. Let $\ov\pi \in  \mathcal{P}\big((\Rd)^3\big)$ be a solution of the problem $\mathcal{J}(\mu,\nu)$, see \eqref{OT3}. By the virtue of Corollary \ref{coro1-intro}, $\sigma_{\ov{\pi}}$ solves the second-order Beckmann problem $\I(f)$. In particular, it satisfies the equation $\dive^2 \sigma_{\ov\pi} = f$, so  $\ov{\pi}$ is a competitor in $\I_{\mathrm{OG}}(f)$. Thanks to assertion (i) of Theorem \ref{thm2-intro} and to the inequality \eqref{eq:grillage>Beckmann}, we obtain
	\begin{equation*}
		\iiint c \,d\ov\pi = \mathcal{J}(\mu,\nu) = \I(f) \leq \I_{\mathrm{OG}}(f) \leq \iiint c \,d\ov\pi,
	\end{equation*}
	which proves optimality of $\ov\pi$ together with the equality $\I(f) = \mathcal{I}_{\mathrm{OG}}(f)$. The finiteness of $\ov{\pi}$ is trivial as it is a probability, while the inclusion \eqref{eq:inclusion} is the final assertion of Corollary \ref{coro1-intro}. This concludes the proof of the part (i).
	
	To prove the statement (ii), we assume that $\mu = \sum_{i=1}^m \mu_i \delta_{x_i}$ and $\nu = \sum_{j=1}^n \nu_j \delta_{y_j}$. Let $\ov{u}$ be any solution of the problem \eqref{secondorder}. Then, by assertion (ii) in Corollary \ref{coro1-intro}, the 3-plan $\ov\pi$ must be of the form,
	\begin{equation*}
		\ov\pi = \sum_{i=1}^m \sum_{j=1}^n \gamma_{ij} \delta_{(x_i,y_j,z_{ij})}, \quad \text{where} \quad z_{ij} = \frac{x_i+y_j}{2} + \frac{\nabla \ov{u}(y_j)-\nabla  \ov{u}(x_i)}{2},
	\end{equation*}
	and thus $\sigma_{\ov{\pi}} =  \sum_{i=1}^m \sum_{j=1}^n \gamma_{ij}\, \sigma^{x_i,y_j,z_{ij}}$. The proof is complete since  $\spt \sigma^{x_i,y_j,z_{ij}} \subset [x_i,z_{ij}] \cup [y_j,z_{ij}]$.
\end{proof}

Optimal grillages have  been  already presented in Example \ref{ex:2vs2}, where $\mu$, $\nu$ were two-point measures. The grillages $\sigma_{\ov{\pi}}$ were showed in Fig. \ref{fig:2vs2}, and they consisted of eight or four bars with affinely varying bending moments. 
Handling more complex data $\mu,\nu$ calls for numerical treatment of the three-marginal optimal transport problem \eqref{OT3}. For a discrete load $\mu,\nu$, it can be rewritten as a finite dimensional second-order conic program. Then, it can be tackled using off-the-shelf convex optimization software. 

\begin{example}[\textit{discrete load}]
	Here we present an optimal grillage found numerically for the discrete load $f=\nu-\mu$ as in Fig. \ref{fig:num1}(a). The measure $\mu$ is uniformly distributed on a grid of $29 \times 29$ points, simulating the gravity pull coming from a square concrete slab. The five equal reaction forces in the columns are encoded by $\nu$. The numerical simulation of an optimal grillage $\sigma_{\ov\pi}$  is showed in Fig. \ref{fig:num1}(c). Meanwhile,
	Fig. \ref{fig:num1}(b) presents the probability $\ov\rho$ solving the problem $\V(\mu,\nu)$.
\end{example}

\begin{figure}[h]
	\centering
	\subfloat[]{\includegraphics*[trim={0cm 0cm -0cm -0cm},clip,width=0.3\textwidth]{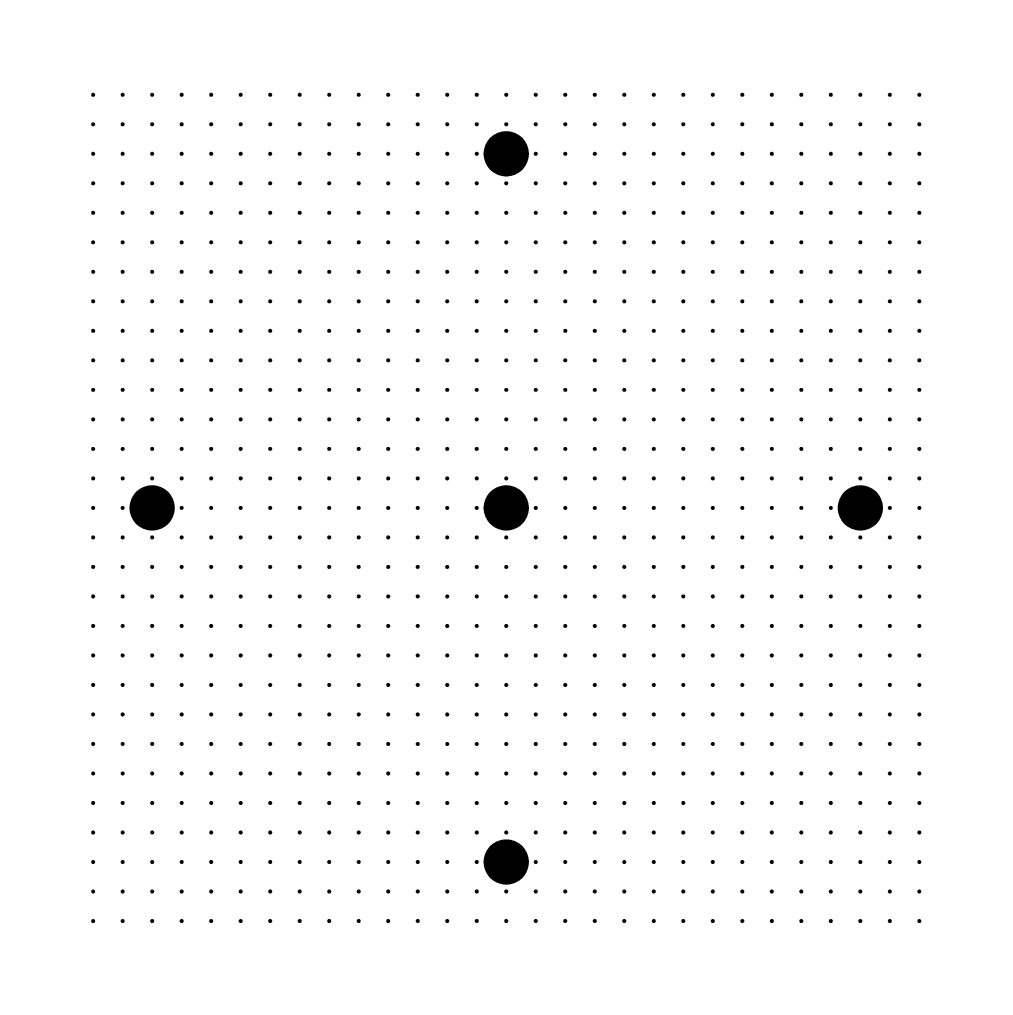}}\hspace{.5cm}
	\subfloat[]{\includegraphics*[trim={0cm 0cm -0cm -0cm},clip,width=0.3\textwidth]{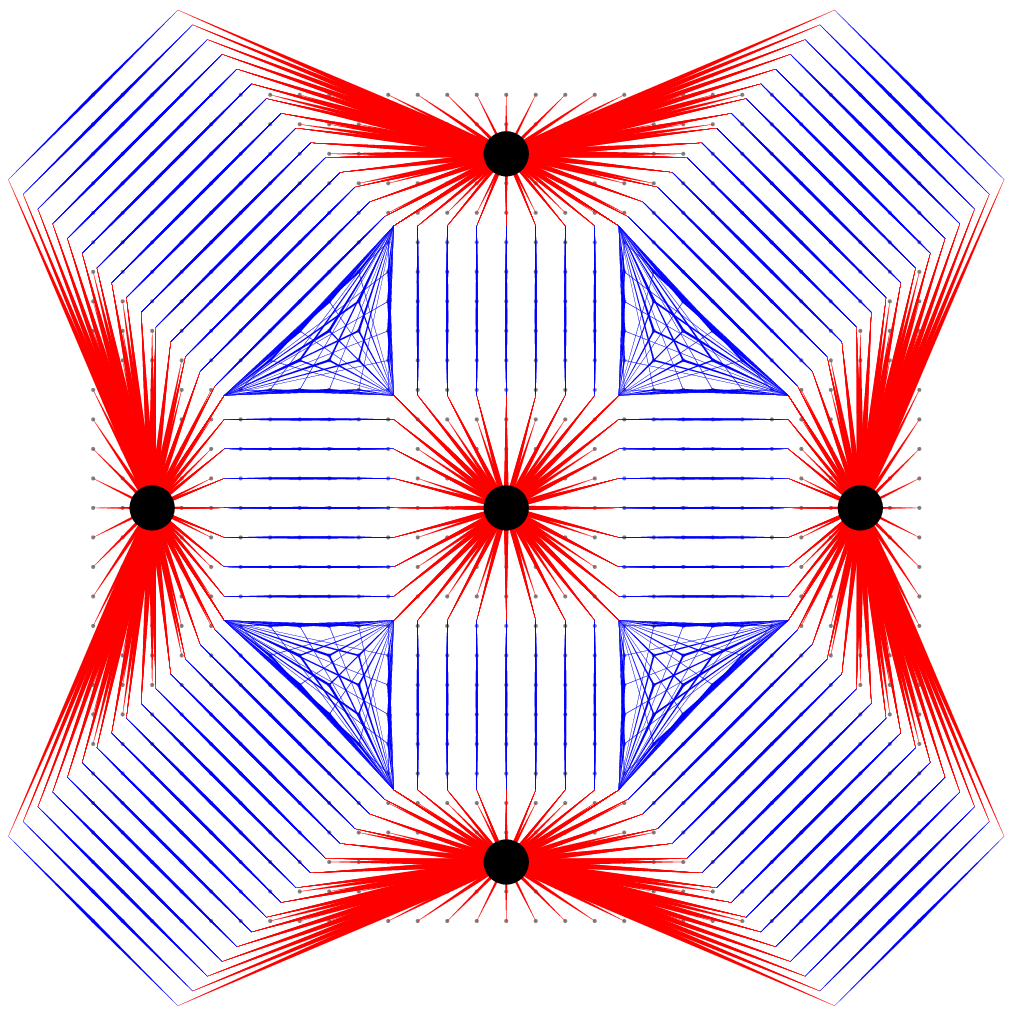}}\hspace{.5cm}
	\subfloat[]{\includegraphics*[trim={0cm 0cm -0cm -0cm},clip,width=0.3\textwidth]{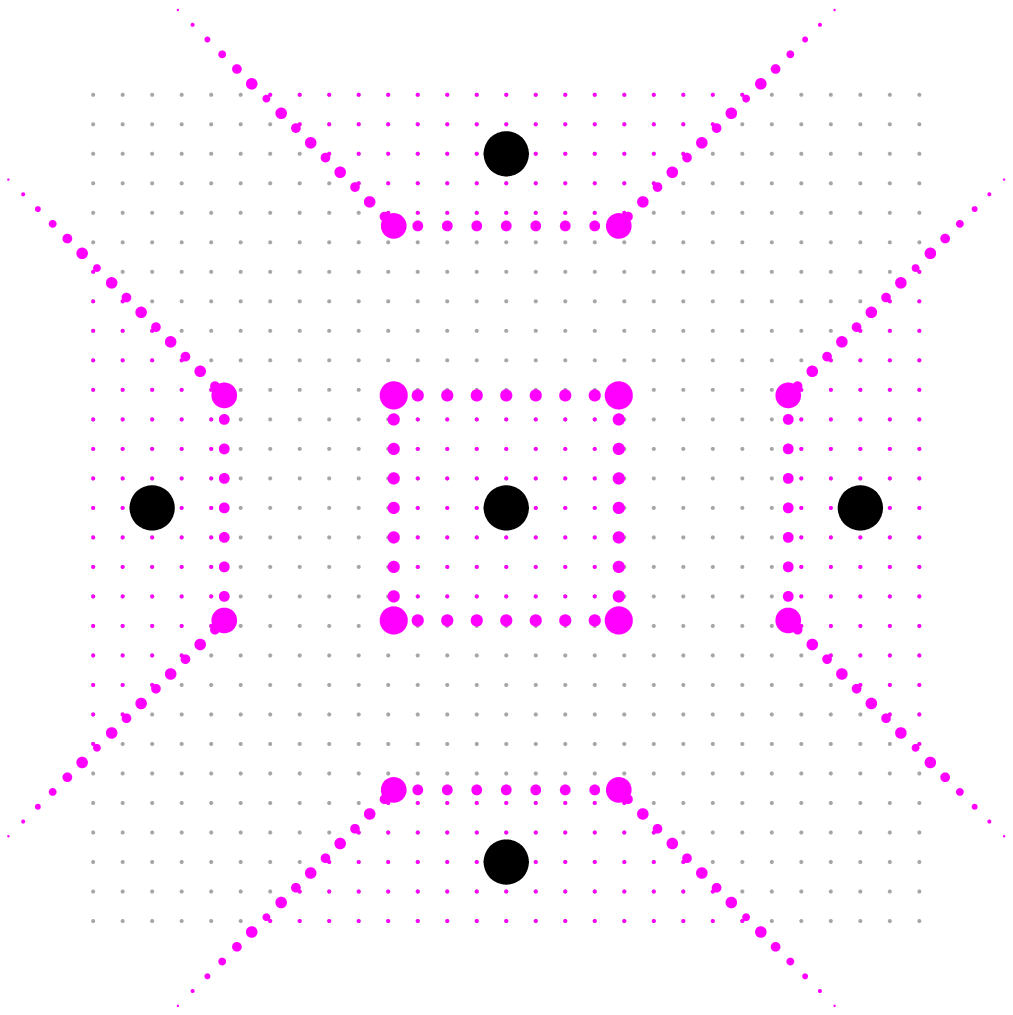}}
	\caption{Numerical solution of the optimal grillage problem: (a) finitely supported data $\mu, \nu$; (b) optimal grillage $\sigma_{\ov\pi}$ where  blue and red indicate, respectively, the positive and the negative part; (c) solution $\ov\rho$ of the optimal dominance problem $\V(\mu,\nu)$.}
	\label{fig:num1}       
\end{figure}

\begin{example}[\textit{continuous load}] \label{ex:continuous}
	In engineering practice, it is typical to assume that the weight of a slab is transferred to the grillage through a finite system of point loads, as demonstrated in the previous example. Nonetheless, it is natural to explore the optimal grillage problem also when the load is continuous: $\mu = \mathcal{L}^2 \mres Q$, where $Q$ is the unit square, see Fig. \ref{fig:num2}(a). Numerically, it comes down to a fine discretization of $\mu$, here by a $113 \times 113$ mesh. Fig. \ref{fig:num2}(c) shows the approximation $\sigma_{\pi_h}$ of an optimal grillage $\sigma_{\ov\pi}$. It is clear that the support of $\sigma_{\ov\pi}$ exceeds the square $Q$, but is contained within the set $ \B(\spt \mu,\spt \nu)$. A prediction of the exact solution $\ov\rho$ for the optimal dominance problem $\V(\mu,\nu)$ is presented in  Fig. \ref{fig:num2}(b).  Based on the numerical simulation the authors expect that in the five quadrilateral regions $\ov{\rho}$ is equal to $\mu$, i.e. to the Lebesgue measure. Partially on their boundaries, there is a part of $\ov\rho$ that is absolutely continuous with respect to $\Ha^1$. Finally, at the vertices, there are concentrations in the form of Dirac delta masses.	
\begin{figure}[h]
	\centering
	\hbox{{\vbox{\offinterlineskip\halign{#\hskip3pt&#\cr
					\centering
					\subfloat[]{\includegraphics*[trim={0cm 0cm -0cm -0cm},height=0.3\textwidth]{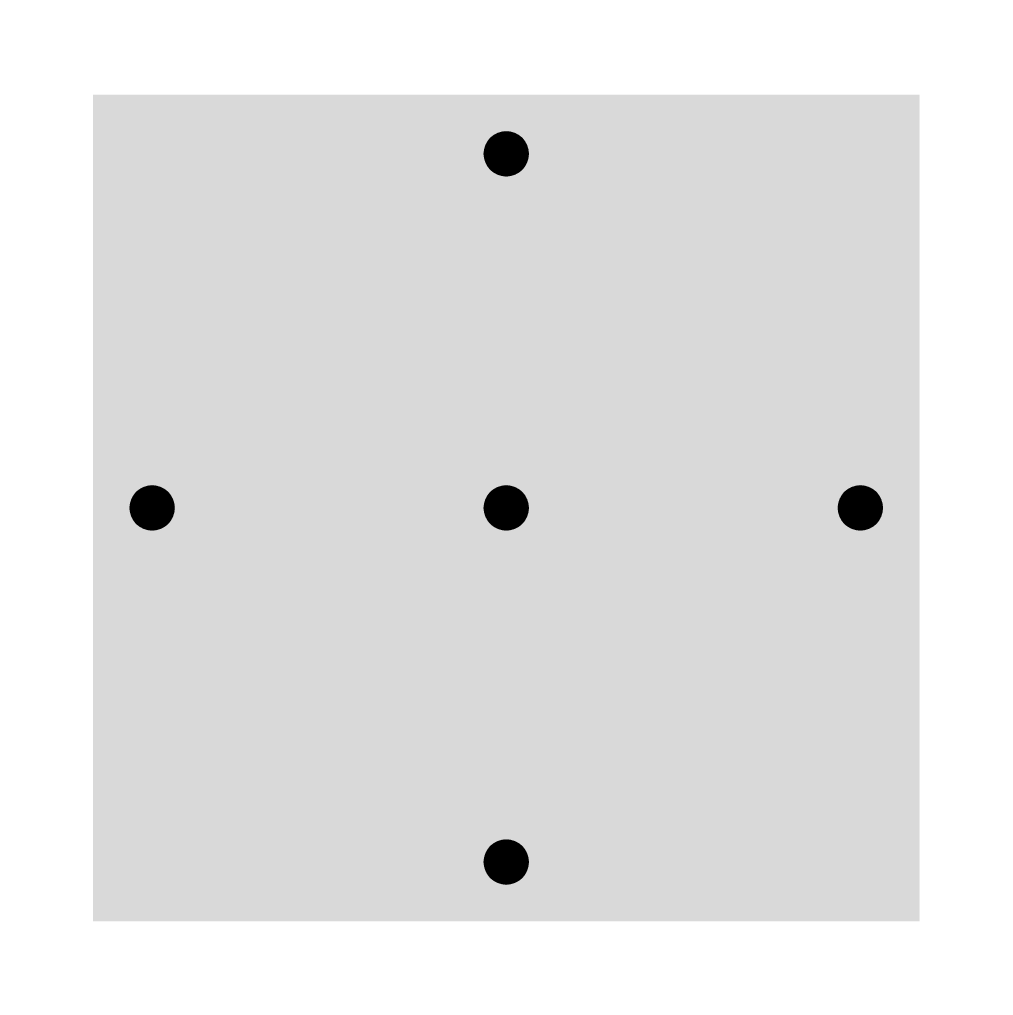}}\cr
					\noalign{\vskip3pt}
					\subfloat[]{\includegraphics*[trim={0cm 0cm -0cm -0cm},width=0.3\textwidth]{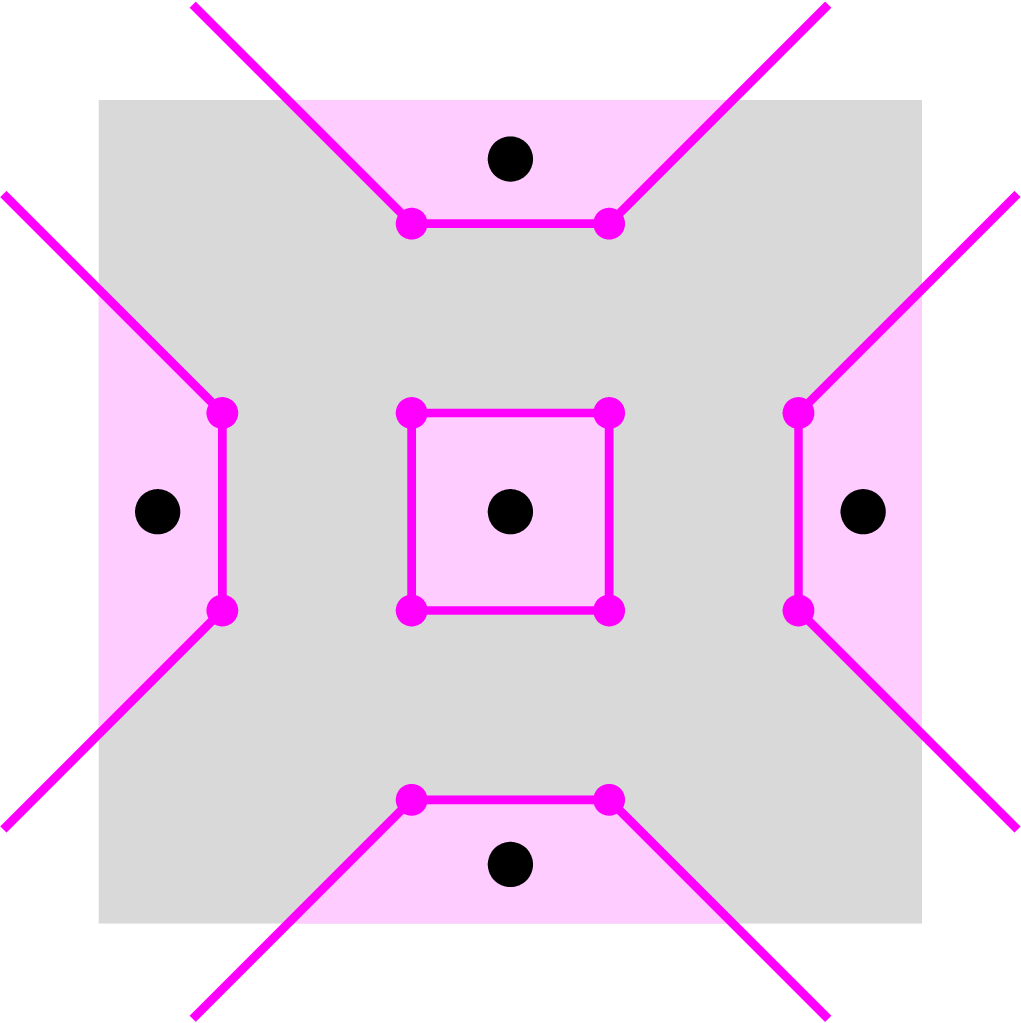}}\cr
		}}}
		\hspace{0.4cm}
		\subfloat[]{\includegraphics*[trim={0cm -0cm -0cm -0cm},clip,width=0.65\textwidth]{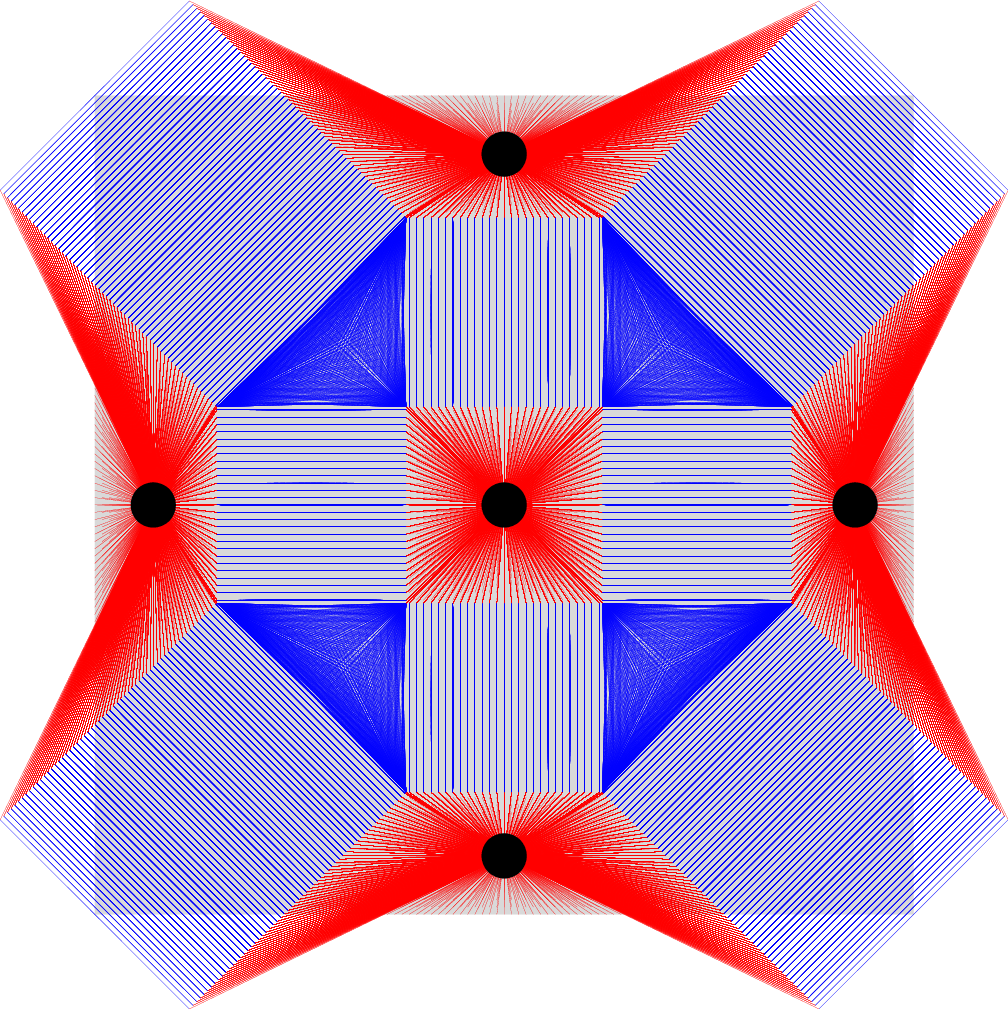}}}
	\caption{(a) Absolutely continuous loading $\mu$ versus discrete reactions $\nu$; (b) prediction of solution $\ov\rho$ of the problem $\V(\mu,\nu)$ consisting of 2D, 1D, and atomic part; (c) numerical approximation of an optimal grillage $\sigma_{\ov{\pi}}$ via a  fine discretization of $\mu$.}
	\label{fig:num2}       
\end{figure}
\end{example}

\subsection{Open problems}

\subsubsection{Loads that are general first-order distributions}

Generalization of Theorem \ref{cor:grillage} towards general first-order distributions $f = f_0 - \dive\, F$ is not straightforward. Unlike $f_0 = \nu -\mu$, the vector measure $F$ does not admit a natural decomposition to a pair of measures. It makes it difficult to propose a generalization of the set $\Sigma(\mu,\nu)$, and thus to find the right optimal transport formulation like \eqref{OT3} whose solution is guaranteed to exist.

The situation improves when the supports of the measures $f_0,F$ are finite. It is then possible to prove that there exists a finitely supported solution of the optimal grillage problem $\I_{\mathrm{OG}}(f)$. The main argument is using \textit{the minimal extensions of jets} put forth in \cite{legruyer2009}. We skip the details here, and, instead, in Fig. \ref{fig:Michell}(b) we show the optimal grillage for $f_0 =0$ and  $F$ as in the bridge problem, see \eqref{eq:bridge_problem}. Note that the mechanical nature of $F$ differs for trusses ($F$ are forces) and grillages ($F$ are torques).

If the measure $F$ is not finitely supported, the issue of existence is more subtle. The authors found examples of $F$ that charge a curved curve for which existence of solutions $\ov{\pi}$ to the optimal grillage problem $\I_{\mathrm{OG}}(-\dive F)$ must imply that $\ov{\pi}\big((\R^2)^3\big) = \infty$. The infinite mass of $\ov{\pi}$ makes it possible to construct  infinite chains of straight bars whose lengths tend to zero, while their thickness is bounded from below by a positive constant. Such chains seem to open the door to forming solutions $\ov{\pi}$ for such data $F$.
Ultimately, the optimal grillage problem for data that are general first-order distributions is not well understood at the moment, and it remains to leave the reader with the following question:

\begin{prob}
	Assume that $f = f_0 - \dive\, F$ where $(f_0,F) \in \Mes_2(\Rd) \times \Mes_1(\Rd;\Rd)$, and that the support of $F$ is infinite. Does  the optimal grillage problem $\I_{\mathrm{OG}}(f)$ admit a solution?
\end{prob}

\subsubsection{Domain confinement}

In practical applications, engineers often work within a prescribed  design domain $\Omega$, a bounded open and connected subset of $\Rd$. For instance, a natural choice for $\Omega$ in Example \eqref{ex:continuous} is the square $Q = \spt \mu$ being the outline of a ceiling. 
The domain confinement can be easily accounted for in the optimal grillage problem $\I_{\mathrm{OG}}(f)$ by adding the constraint $\spt \sigma_\pi \subset \Ob$. Assuming that the load $f$ is a measure, from assertion (i) of Corollary \ref{cor:grillage} we can see that the whole result holds true provided that,
\begin{equation}
	\label{eq:Omega}
	\B(\spt f_+, \spt f_-) \ \subset \ \ov{\Omega}.
\end{equation}
In this case, the constraint $\spt \sigma_\pi \subset \Ob$ is not binding. If the inclusion \eqref{eq:Omega} is not satisfied, then one should work within the framework of the second-order Beckmann problem, whose modification now reads,
\begin{equation*}
	\mathcal{I}(f,\Omega)  := \min \biggl\{ \int \varrho^0(\sigma) \ : \ \sigma \in \Mes(\R^d;\Sdd),  \ \ \spt \sigma \subset \Ob,  \ \ \dive^2 \sigma = f  \biggr\}.
\end{equation*}
Numerical experiments in 2D indicate that, with the condition \eqref{eq:Omega} violated, there might not be solutions of $\I(f,\Omega)$ which take the form $\sigma_\pi$. It appears that optimal $\ov{\sigma}$ may charge subsets of the boundary $\partial \Omega$ with the density being a full-rank matrix. In terms of mechanics, it corresponds to 1D bars (possibly curved) subject not only to bending moments but also to torsion. In the interior $\Omega$, however, the solution seems to decompose to straight bars $\sigma^{x,y,z}$. These observations lead to the following open problem:
\begin{prob}
	Assume a bounded domain $\Omega \subset \Rd$ with Lipschitz regular boundary and a load $f \in  \Mes(\Ob)$. Do there exist  $\sigma_\bO \in \Mes(\R^d;\Sdd)$ concentrated on $\bO$ and $\pi \in \Mes_+(\ov\Omega ^3)$  such that,
	\begin{equation*}
		\ov{\sigma} =  \iiint \sigma^{x,y,z} \, \pi(dxdydz) + \sigma_\bO
	\end{equation*}
	solves the confined  second-order Beckmann  problem $\mathcal{I}(f,\Omega)$?
\end{prob}

\smallskip

\appendix

\section{Convex analysis}\label{app:convex}  \ Let $X$ be a normed space and let $h: X\to \R\cup \{+\infty\}$ be a convex function.
Recall that the Moreau-Fenchel conjugate of $h$ is defined on the dual space $X^*$ by,
$$  h^{*} (x^*)  := \sup_{x\in X}  \big\{ \pairing {x,x^*} - h(x) \big\}\qquad \forall\, x^*\in X^* .$$
Clearly, $h^*$ is convex and lower semi-continuous with respect to the weak-* topology on $X^*$.
Next, we define the biconjugate of $h$ on $X$ by,
$$  h^{**} (x)  :=  \sup_{x^*\in X^*}  \big\{ \pairing {x,x^*} - h^*(x^*) \big\}\qquad \forall \,x\in X .$$
The following classical result (due to J.J.
 Moreau \cite{Moreau1966} in the infinite dimensional case)  is used several times in this paper. 
\begin{proposition}
	\label{duality_classical}
	Assume that there exists $r>0$ such that
	$   \sup \{ h(x) \, :\, \|x\| \le r \} < +\infty .$
	Then:
	\begin{itemize}
\item [(i)]  $h$ is continuous at $0$, while $h^*$ is coercive and attains its minimum on $X^*$;
\item [(ii)]  we have the equalities:  $ h(0) = h^{**}(0) = - \min h^{*}.$
\end{itemize}

\end{proposition}

\section{Integration by parts } 
\label{app:byparts}

We give here the justification of the integration by parts formula on the whole $\R^d$ that was required in Section \ref{Prelim}  (see \eqref{byparts}) and in the proof of Proposition \ref{positive}.

\begin{lemma}\label{lemma:byparts}  Let  $f = f_0 - \dive\, F$, where
 $(f_0, F)$ is any pair in $\Mes_2(\R^d)\times\Mes_1(\R^d;\R^d)$ satisfying \eqref{genload}. Let $\sigma \in \Mes(\R^d;\Sdd)$ 
 satisfy  $\dive^2 \sigma = f$ in  $\D'(\R^d)$.
 Then, for every $u\in C^2(\R^d)$ with $\Lip(\nabla u)<+\infty$, we have,
\begin{equation}\label{Rd-byparts}
 { \pairing {\nabla^2 u, \sigma}}  =   \pairing{u,f}  = \pairing{u,f_0} + \pairing{\nabla u, F}.
\end{equation}
\end{lemma}
\begin{proof} By the orthogonality conditions \eqref{genload}, \eqref{Rd-byparts} is valid for affine functions. Therefore, it is not restrictive to assume that $u$ and $\nabla u$ vanish at $0$; we may also assume that $\Lip(\nabla u)\le 1$.
On the other hand, the  equality $\dive^2\sigma=f$ in the sense of distributions implies that  \eqref{Rd-byparts} holds true if $u\in \D(\R^d)$. By using smooth convolution kernels, this can be extended to $u\in C^2(\R^d)$ that are compactly supported in $\R^d$. In order to remove the latter condition,   we consider a sequence of radial  cut-off functions $\eta_k(x) := \eta(\frac{|x|}{k})$ where,
 $$\eta\in \D(\R;[0,1]),\qquad \eta(t)=1 \quad \text{if\quad$|t|\le k$} ,\qquad \eta(t)=0 \quad \text{if\quad$t\ge 2k$}.$$
Then, we set $u_k := u\, \eta_k$.  Since  $u_k $  satisfies  \eqref{Rd-byparts}, we have only to check that the sequence $(v_k)$, given by $v_k := u- u_k= (1-\eta_k) u$,    satisfies,
\begin{equation}\label{claim-truncature}
\pairing{v_k,f_0}\to 0,\qquad \pairing{\nabla v_k,F}\ \to 0, \qquad \pairing{\nabla^2 v_k,\sigma} \to 0 .
\end{equation}
Since $v_k(x)$ vanishes for $|x|\le k$, while $|v_k(x)| \le |(u(x)| \le \frac1{2} |x|^2$  elsewhere, we infer that 
$(v_k)$ is bounded in $X_2(\R^d)$. Hence, $\pairing{v_k,f_0}\to 0 $ by applying \eqref{domiCV} with $\mu=f_0 \in \Mes_2(\R^d)$.
In the same way, $\nabla v_k$ is supported on the subset $\{k\le |x|\le 2k\}$, where it satisfies the upper bound,
\begin{align*}
	|\nabla v_k(x)| &= \left|\big(1-\eta_k(x)\big) \nabla u(x) -  \frac1{k}  u(x)\,  \eta' \Big(\frac{|x|}{k}\Big) \frac{x}{\abs{x}}\right| \\
	&\le |\nabla u(x)| + \frac{\Lip(\eta)}{k} |u(x)|  \le \big(1 +  \Lip(\eta)\big)\, |x| .
\end{align*}
In the last inequality we used the fact that $|\nabla u(x)| \le |x|$, and  $  |u(x)| \le \frac1{2} |x|^2 \le k |x| $
on $\spt (\nabla v_k)$. 
It follows that  $(\nabla v_k)$ is bounded in $X_1(\R^d;\Rd)$ and, recalling that  $F\in \Mes_1(\R^d;\R^d)$, we may apply 
\eqref{domiCV} to infer that $\pairing{\nabla v_k,F} \to 0$.
Next,  we see that\  $\nabla^2 v_k= (1-\eta_k) \nabla^2 u  - (\nabla \eta_k \otimes \nabla u +\nabla u\otimes \nabla \eta_k
+ u  \nabla^2 \eta_k)$ where:
\begin{align*}
	\nabla  \eta_k (x) &=  \frac1{k} \eta'\Big(\frac{|x|}{k}\Big) \frac{x}{|x|},\\
	 \nabla^2 \eta_k(x) &=\frac1{k^2} \eta''\Big(\frac{|x|}{k}\Big) \frac{x\otimes x}{|x|^2} + \frac1{k|x|} \eta' \Big(\frac{|x|}{k}\Big)\left( \mathrm{Id} - \frac{x\otimes x}{|x|^2}\right).
\end{align*}
 We notice  that:
 \begin{enumerate}
\item[-] for any $x \neq 0$ there hold the bounds $\varrho(\nabla \eta_k \otimes \nabla u +\nabla u\otimes \nabla \eta_k)(x) \le 2|\nabla u(x)|\frac{\Lip(\eta)}{k}$ and $\varrho\big(\nabla^2 \eta_k(x)\big) \leq \frac{\Lip(\eta')}{k^2} + \frac{\Lip(\eta)}{k \abs{x}}$;

\item[-] $\nabla \eta_k$ and $\nabla^2 \eta_k$ are supported on $\{ k\le |x|\le 2k\}$, where it holds that $|u(x)|\le \frac1{2} |x|^2\le 2 k^2$ and $|\nabla u(x)|\le  |x|\le 2 k$.
\end{enumerate}
All in all,  we obtain a uniform upper bound for $\varrho(\nabla^2 v_k)$ whose support is contained  in $\{k\le |x| \le 2k\}$,
\begin{align*} \varrho\big(\nabla^2 v_k(x)\big) &\le  \varrho\big(\nabla^2u(x)\big)  + \varrho\big(\nabla \eta_k \otimes \nabla u +\nabla u\otimes \nabla \eta_k\big)(x)
+ |u(x)|\, \varrho\big(\nabla^2 \eta_k(x)\big) \\
&\le 1 +  2\,|\nabla u(x)|\frac{\Lip(\eta)}{k} + |u(x)| \Big( \frac{\Lip(\eta')}{k^2} + \frac{\Lip(\eta)}{k \abs{x}} \Big)\\
& \le 1 +  2 \,\Lip(\eta) \frac{|x|}{k} + \frac1{2} \Lip(\eta') \frac {|x|^2}{k^2} + \frac1{2} \Lip(\eta) \frac {|x|}{k} \\
&\le C:=1 + 5 \, \Lip(\eta) + 2 \,\Lip(\eta').
\end{align*} 
By virtue of the  inequality  $|\pairing{\nabla^2 v_k,\sigma}|\, \le\,  \varrho(\nabla^2 v_k)\, \varrho^0(\sigma) \le C  \varrho^0(\sigma) $ 
holding in  the sense of measures,
it follows that,
$$  |\pairing{\nabla^2 v_k,\sigma}|  \le  C \int_{\{k\le |x| \le 2k\}}  \varrho^0(\sigma) .
$$
Since $\int_{\R^d} \varrho^0(\sigma)<+\infty$, we conclude that $\pairing{\nabla^2 v_k,\sigma}\to 0$ for $k\to\infty$
as required in \eqref{claim-truncature}.
This ends the proof.
\end{proof}

\section{two-point measures -- Additional proofs}
\label{app:examples}

\begin{proof}[Proof of Lemma \ref{lem:positive_rij}]
	It is not restrictive to assume that $\pairing{x_1,y_1} \geq 0$. It is then easy to check that inequality \eqref{eq:angles2} is met automatically, thus we can focus on \eqref{eq:angles1} only. Next, we can enforce the orientation of the eigenvector $b$ so that:
	\begin{equation}
		\label{eq:ny1}
		\pairing{b,y_1} > 0, \quad \pairing{b,y_2} < 0, \quad \pairing{b,x_1} \geq 0, \quad \pairing{b,x_2} \leq 0.
	\end{equation}
	Accordingly, one can easily check that $\gamma_{ij} >0$ when $i = j$, no matter if inequality \eqref{eq:angles1} holds or not. Therefore, we have to show that \eqref{eq:angles1} is equivalent to the system of two inequalities $\gamma_{12}\geq0$, $\gamma_{21}\geq 0$. What is more, this equivalence is trivial to show when $\pairing{x_1,y_1} = 0$. In the sequel we thus assume that $\pairing{x_1,y_1}>0$.
	
	For $t >0$ we define:
	\begin{equation*}
		\tilde{x}_1(t) = t \,x_1, \quad \tilde{x}_2(t) = \frac{1}{t} \, x_2, \qquad g(t) = t \, \pairing{\tilde{x}_2(t) - y_2, y_1 - \tilde{x}_1(t)}.
	\end{equation*}
	The function $g$ is quadratic on $\R$. Thanks to $\pairing{x_1,y_1}>0$, one can show that $g$ is concave, and it admits two positive roots: $0< t_1 < t_2$. For each $k\in \{1,2\}$ we define two vectors,
	\begin{equation}
		\label{eq:vkwk}
		v_k= \tilde{x}_2(t_k)- y_2, \qquad w_k =  y_1 - \tilde{x}_1(t_k).
	\end{equation}
	We shall show that, for both $k$, $(v_k,w_k)$ are mutually orthogonal eigenvectors of $N-M$ (not necessarily normalized). Orthogonality follows from the fact that $t_k$ are the roots for $g$. The next observation is key,
	\begin{equation*}
		-y_2 \otimes y_1 +\tilde{x}_2(t) \otimes \tilde{x}_1(t) = -y_2 \otimes y_1 + x_2 \otimes x_1 = N-M
	\end{equation*}
	for any $t>0$. We exploit it to obtain,
	\begin{align*}
		(N-M)\, v_{k} &=  -\pairing{y_1,v_{k}} \, y_2  + \pairing{\tilde{x}_1(t_k), v_{k}}\,\tilde{x}_2(t_k) \\
		&= -\pairing{y_1-w_{k},v_{k}} \, y_2  +\pairing{\tilde{x}_1(t_k), v_{k}}\,\tilde{x}_2(t_k)  \\ 
		&= -  \pairing{\tilde{x}_1(t_k),v_{k}} \, y_2 +  \pairing{\tilde{x}_1(t_k), v_{k}}\,\tilde{x}_2(t_k)  \\
		&=  \pairing{\tilde{x}_1(t_k),v_{k}}\, ( \tilde{x}_2(t_k) - y_2) = \pairing{\tilde{x}_1(t_k),v_{k}} \, v_{k}.
	\end{align*}
	Similarly, one shows that $(N-M)\, w_{k} =  \pairing{-y_2,w_{k}} \, w_{k}$. The corresponding eigenvalues are  $\lambda_{v_k} =  \pairing{\tilde{x}_1(t_k),v_{k}}$ and $\lambda_{w_k} =  \pairing{-y_2,w_{k}}$. Next, we assess which of the four vectors \eqref{eq:vkwk}  are parallel to $b$. To that aim we  compare the signs; recall that $\lambda_a<0$, $\lambda_b >0$. We compute the derivative of $g$ at its roots,
	\begin{align*}
		g'(t_k) &= 	\pairing{-y_2, y_1 - t_k x_1} +  \pairing{x_2- t_k y_2,- x_1} \\
		&= \pairing{-y_2, w_k} -  \langle v_k, \tilde{x}_1(t_k) \rangle  = \lambda_{w_k} - \lambda_{v_k}.
	\end{align*}
	Due to the concavity of the quadratic function $g$, it must satisfy $g'(t_1) >0$ and $g'(t_2)<0$. We conclude that  $\lambda_{v_1} < \lambda_{w_1}$ and $\lambda_{v_2} > \lambda_{w_2}$. As a result, $v_2, w_1$ must be the eigenvectors that are parallel to $b$. One can easily check that the three vectors also have the same orientations. To sum up, we have
	\begin{equation}
		\label{eq:ne}
		b= \frac{v_{2}}{\abs{v_{2}}}=\frac{w_{1}}{\abs{w_{1}}}   =  \frac{1}{\abs{v_{2}}} \big( \tilde{x}_2(t_2)- y_2 \big) = \frac{1}{\abs{w_{1}}}   \big(y_1 - \tilde{x}_1(t_1) \big).
	\end{equation}
	
	We are ready to prove our assertion. Since $\langle b, y_{1}-y_2 \rangle >0$ due to \eqref{eq:ny1}, we deduce that
	  $\mathrm{sgn} (\gamma_{12}) = \mathrm{sgn}  (\pairing{b,y_1-x_1})$ and $\mathrm{sgn}  (\gamma_{21}) = \mathrm{sgn}  (\pairing{b,x_2-y_2})$. Defining the two functions:
	\begin{align}
		\label{eq:f1f2}
		&f_1(t):=\abs{v_2} \pairing{b,y_1 - t \,x_1} =  \pairing{\tilde{x}_2(t_2)- y_2,y_1 - \tilde{x}_1(t)},\\
		&f_2(t):=\abs{w_1} \pairing{\tfrac{1}{t}\,x_2 - y_2,b} = \pairing{\tilde{x}_2(t) - y_2,y_1 - \tilde{x}_1(t_1)}
	\end{align}
	we see that $\mathrm{sgn}  (\gamma_{12}) = \mathrm{sgn} (f_1(1))$, and $\mathrm{sgn}  (\gamma_{21}) = \mathrm{sgn} (f_2(1))$.
	Due to \eqref{eq:ny1}, the function $f_1$  is strictly decreasing, and $f_2$ is strictly increasing on $(0,\infty)$. The alternative formulas for $f_1,f_2$ given above follow by \eqref{eq:ne}. They provide the equalities $f_1(t_2) = f_2(t_1) = 0$ since $g(t_1) = g(t_2) =0$.
	
	Let us now assume that the inequality \eqref{eq:angles1} is satisfied or, equivalently,  $g(1) \geq 0$. By the properties of $g$, there holds $t_1 \leq 1 \leq t_2$. Since $f_1$ is decreasing, we have $f_1(1) \geq f_1(t_2) =  0$. Similarly, because $f_2$ is increasing, $f_2(1) \geq f_2(t_1)  =0$. This gives $\gamma_{12} \geq 0$ and $\gamma_{21} \geq 0$.
	
	Contrarily, assume that  \eqref{eq:angles1} does not hold, which gives $g(1) <0$. Then, either $1 <t_1 <t_2$ or $t_1<t_2 < 1$. In the first case, we have $f_2(1)< f_2(t_1)  = 0$, which  yields $\gamma_{21}<0$. In the second case $f_1(1)<f_1(t_2)  = 0$,  and thus $\gamma_{12}<0$ by the same token. The proof is complete.
\end{proof}

\smallskip

\begin{proof}[Proof of Lemma \ref{lem:u}]	
	Let us observe that  $h:\R \to \R$ is $C^1$ and $2\pi$-periodic  on $\R$. We thus immediately infer that $u$ is $C^1$ on $\R^2 \backslash \{z_0\}$. However, thanks to the factor $r^2$ in the definition of $\upsilon$, it can be showed that $\nabla u$ is also continuous at $z_0$ with $\nabla u(z_0) = 0$, that is $u \in C^1(\R^2)$. The
	function $h$ is also piecewise $C^2$. More precisely, $h''$ has discontinuity points $2k \pi$ and $ 2k \pi + \pi/(2\alpha)$ for integer $k$ if and only if  $\alpha \neq \beta$. As a result, $u$ is not of class $C^2$ except for the case $\pairing{x_2-y_2,y_1-x_1}  = 0$, which corresponds to the condition $\alpha = \beta$ exactly. Nonetheless, the piecewise continuity of $h''$ is enough to deduce that,
	\begin{align*}
		u \in C^2(V_1  \cup V_2), \qquad &V_i = \big\{ (\varrho,\vartheta)^{-1}(r,\theta) \, : \, r>0, \  \theta\in A_i  \big\}, \\
		&A_1 = ]0, \pi/(2\alpha)[, \quad \ \ A_2 = ]\pi/(2\alpha),2\pi[.
	\end{align*}
	Moreover, on each open set $V_i$ there holds $\nabla^2 u (x) = (Q(x))^\top H_i\big( \varrho(x), \vartheta(x) \big) \, Q(x)$, where $Q(x)$ is a rotation matrix, and
	\begin{equation*}
		H_i(r,\theta) = \begin{bmatrix}
			\,\frac{\partial^2 \upsilon}{\partial r^2}  &  -\frac{1}{r^2}\frac{\partial \upsilon}{\partial \theta}+ \frac{1}{r}\frac{\partial^2 \upsilon}{\partial r \partial \theta} \,\\
			-\frac{1}{r^2}\frac{\partial \upsilon}{\partial \theta}+ \frac{1}{r}\frac{\partial^2 \upsilon}{\partial r \partial \theta} & \frac{1}{r^2}\frac{\partial^2 \upsilon}{\partial^2 \theta}  + \frac{1}{r}\frac{\partial\upsilon}{\partial r}  \,
		\end{bmatrix} = 
		\begin{bmatrix}
			\, h_i(\theta) &  \frac{1}{2} h_i'(\theta) \,\\
			\,\frac{1}{2} h_i'(\theta) & \frac{1}{2} h''_i(\theta) + h_i(\theta) \,
		\end{bmatrix}.
	\end{equation*}
	Since $\R^2\backslash (V_1 \cup V_2)$ is Lebesgue negligible and $h_i$ are cosine functions, we infer that $u \in W^{2,\infty}_{\mathrm{loc}}(\R^2)$, which establishes the first part of the assertion.
	
	To prove the second part, it is enough that we check that for $i=1,2$ the eigenvalues 
	of $H_i(r,\theta) = H_i(\theta)$ remain in the regime $[-1,1]$ if and only if $\pairing{x_2-y_2,y_1-x_1} \leq 0$. Starting from $i = 1$, we obtain
		\begin{equation*}
		H_1(\theta) = 
		\begin{bmatrix}
			\, \cos(2\alpha \theta) &  \alpha  \sin(2\alpha \theta) \\
			\, \alpha \sin(2\alpha \theta)  & (1- 2\alpha^2) \cos\big(2\alpha \theta\big) \,
		\end{bmatrix},
	\end{equation*}
	and, after using the Pythagorean trigonometric identity, formulas for the eigenvalues  $\lambda_-, \lambda_+$  follow,
	\begin{align*}
		\lambda_{\pm} (\theta) 
		=  (1-\alpha^2) \cos(2\alpha\theta) \pm \alpha  \sqrt{1- (1-\alpha^2)  \cos^2(2\alpha \theta)  }.
	\end{align*}
	Assume first that $\pairing{x_2-y_2,y_1-x_1}  > 0$, which gives $\alpha > 1$. Then, clearly $\lambda_-(0) = 1- 2\alpha^2 <-1$.
	It remains to check the case when $\pairing{x_2-y_2,y_1-x_1}  \leq 0$, for which $\alpha,\beta \leq 1$.  Thanks to elementary computations we get the estimate,
	\begin{align*}
		\Big(\pm1 - &(1-\alpha^2)  \cos(2\alpha \theta) \Big)^2 \\
		=& \, (1-\alpha^2) \big(1 \mp \cos(2\alpha \theta) \big)^2 +\alpha^2 \big(1- (1-\alpha^2)  \cos^2(2\alpha \theta) \big)\\
		\geq &\,  \Big( \alpha  \sqrt{1- (1-\alpha^2)  \cos^2(2\alpha \theta)}\Big)^2,
	\end{align*}
	where we acknowledged that $\alpha \leq 1$. 
	Since the term $(1-\alpha^2)  \cos(2\alpha \theta)$ ranges in $[-1,1]$, from the estimate above we can deduce that indeed $\lambda_{\pm}(\theta)\in [-1,1]$. Handling the matrix $H_2(\theta)$ amounts to replacing $\alpha$ with $\beta$. However, since $\beta\leq 1$ as well, the same reasoning stands. 
\end{proof}

\smallskip

\bibliographystyle{plain}

\begin{thebibliography}{10}
	\bibitem{agueh2011barycenters}
	M. Agueh, G. Carlier:
	\newblock Barycenters in the Wasserstein space.
	\newblock {\em SIAM J. Math. Anal.} 43:904--924, 2011.
	
	\bibitem{ABC}
	J.-J. Alibert, G.~Bouchitt\'e,  T. Champion:
	\newblock A new class of costs for optimal transport planning.
	\newblock {\em Eur. J. Appl. Math.} 30:1229--1263, 2019.
	
	\bibitem{ambrosio2000}
	L. Ambrosio, N. Fusco, D. Pallara: 
	\newblock {\em Functions of bounded variation and free
		discontinuity problems}.
	\newblock Clarendon Press, Oxford, 2000.
	
	\bibitem{azagra2018}
	D. Azagra, E. Le Gruyer, C. Mudarra:
	\newblock Explicit formulas for $C^{1,1}$ and $C^{1,\omega}_{\mathrm{conv}}$ extensions of 1-jets
	in {H}ilbert and superreflexive spaces.
	\newblock {\em J. Funct. Anal.} 274:3003--3032, 2018.
	
	\bibitem{backhoff-veraguas2022}
	J. Backhoff-Veraguas, G. Pammer:
	\newblock Applications of weak transport theory.
	\newblock {\em Bernoulli} 28:370-394, 2022.

	\bibitem{beiglbock2013}
	M. Beiglb{\"o}ck, P. Henry-Labordere, F. Penkner:
	\newblock Model-independent bounds for option prices -- a mass transport
	approach.
	\newblock {\em Finance Stoch.} 17:477--501, 2013.
	
	\bibitem{bolbou2022}
	K. Bo{\l}botowski, G. Bouchitt{\'e}:
	\newblock Optimal design versus maximal Monge--Kantorovich metrics.
	\newblock {\em Arch. Ration. Mech. Anal.}, 243:1449--1524,
	2022.
	
	\bibitem{bolbotowski2018a}
	K. Bo{\l}botowski, L. He, M. Gilbert:
	\newblock Design of optimum grillages using layout optimization.
	\newblock {\em Struct. Multidiscip. Optim.} 58:851--868, 2018.

    \bibitem{bolbotowski2022}
	K. Bo{\l}botowski, T. Lewi{\'n}ski: 
    \newblock Setting the free material design
	problem through the methods of optimal mass distribution. \newblock {\em Calc. Var. Partial. Differ. Equ.} 61: Article No. 76, 2022.
	
	\bibitem{bouchitte2001}
	G. Bouchitt{\'e}, G. Buttazzo:
	\newblock Characterization of optimal shapes and masses through
	{M}onge-{K}antorovich equation.
	\newblock {\em J. Eur. Math. Soc.} 3:139--168, 2001.
	
	\bibitem{BBS1997}
	G. Bouchitt\'e, G. Buttazzo, P. Seppecher:
	\newblock Shape optimization solutions via {M}onge-{K}antorovich equation.
	\newblock {\em C. R. Acad. Sci. Paris S\'er. I Math.} 324:1185--1191,
	1997.
	
	\bibitem{bouchitte2007}
	G. Bouchitt{\'e}, I. Fragal{\`a}:
	\newblock Optimality conditions for mass design problems and applications to
	thin plates.
	\newblock {\em Arch. Ration. Mech. Anal.} 184:257--284,
	2007.
	
	\bibitem{bouchitte2008}
	G. Bouchitt{\'e}, W. Gangbo, P. Seppecher:
	\newblock Michell trusses and lines of principal action.
	\newblock {\em Math. Models Methods Appl. Sci.}
	18:1571--1603, 2008.

        \bibitem{bouchitte1988}
	G. Bouchitt{\'e}, M. Valadier:
	\newblock Integral representation of convex functionals on a space of measures.
	\newblock {\em J. Funct. Anal.}
	80:398--420, 1988.

        
        \bibitem{cependello1998}
	M. Cepedello-Boiso:
	\newblock On regularization in superreflexive Banach spaces by infimal convolution formulas.
	\newblock {\em Stud. Math.} 129:265--284, 1998.
	
	
	 \bibitem{ciosmak2021}
	K.J. Ciosmak:
	\newblock Optimal transport of vector measures. \newblock {\em Calc. Var. Partial. Differ. Equ.} 60: Article No. 230, 2021.
	
	\bibitem{conti2006}
	S. Conti, F. Maggi, S. M{\"u}ller:
	\newblock Rigorous derivation of {F}{\"o}ppl’s theory for clamped elastic
	membranes leads to relaxation.
	\newblock {\em SIAM J. Math. Anal.} 38:657--680, 2006.


        \bibitem{dephilippis2015}
	G. De Philippis, A. Figalli: 
	\newblock Optimal regularity of the convex envelope.
	\newblock {\em Trans. Am. Math. Soc.} 367:4407--4422, 2015.
    
	\bibitem{Dentcheva2003}
	D. Dentcheva, A. Ruszczy\'{n}ski:
	\newblock Optimization with stochastic dominance constraints.
	\newblock {\em SIAM J. Optim.} 14:548--566, 2003.
	
	\bibitem{Dweik}
	S. Dweik, F. Santambrogio:
	\newblock {$L^p$} bounds for boundary-to-boundary transport densities, and
	{$W^{1,p}$} bounds for the {BV} least gradient problem in 2{D}.
	\newblock {\em Calc. Var. Partial Differ. Equ.} 58: Article. No. 31,
	2019.

        \bibitem{evans1999}
	L.C. Evans, W. Gangbo:
	\newblock {\em Differential equations methods for the Monge-Kantorovich mass transfer problem}.
	\newblock American Mathematical Society, 1999.
	
	\bibitem{gangbo1998optimal}
	W. Gangbo, A. Święch:
	\newblock Optimal maps for the multidimensional Monge-Kantorovich problem.
	\newblock {\em Comm. Pure Appl. Math.}  51:23--45, 1998.
    
	\bibitem{Ghoussoub}
	N. Ghoussoub, Y.-H. Kim, T. Lim:
	\newblock Structure of optimal martingale transport plans in general
	dimensions.
	\newblock {\em Ann. Probab.} 47:109--164, 2019.
	
	\bibitem{Goffman}
	C. Goffman, J. Serrin:
	\newblock Sublinear functions of measures and variational integrals.
	\newblock {\em Duke Math. J.} 31:159--178, 1964.
	
	\bibitem{gozlan2020}
	N. Gozlan, N. Juillet:
	\newblock On a mixture of Brenier and Strassen theorems.
	\newblock {\em Proc. Lond. Math. Soc.},
	120:434--463, 2020.
	
	\bibitem{gozlan}
	N. Gozlan, C. Roberto, P.-M. Samson, P. Tetali:
	\newblock Kantorovich duality for general transport costs and applications.
	\newblock {\em J. Funct. Anal.} 273:3327--3405, 2017.
	

	
	\bibitem{hanin1997}
	L.G. Hanin:
	\newblock  Duality for general Lipschitz classes and applications.
	\newblock
	{\em Proc. London Math. Soc.} 75:134--146, 1997.
	
	\bibitem{hanin1994}
	L.G. Hanin, S.T. Rachev:
	\newblock Mass-transshipment problems and ideal metrics.
	\newblock
	{\em J. Comput. Appl. Math.} 56:183--196, 1994.
	
	\bibitem{henry-labordere2017}
	P. Henry-Labordere:
	\newblock {\em Model-free hedging: A martingale optimal transport viewpoint}.
	\newblock  Chapman and Hall/CRC, 2017.
	
	
	\bibitem{kolesnikov2022}
	A.V. Kolesnikov, F. Sandomirskiy, A. Tsyvinski,  A.P. Zimin:
	\newblock Beckmann's approach to multi-item multi-bidder auctions.
	\newblock arXiv preprint arXiv:2203.06837, 2022.

	
	\bibitem{legruyer2009}
	E. Le Gruyer:
	\newblock Minimal Lipschitz extensions to differentiable functions defined on a
	Hilbert space.
	\newblock {\em Geom. Funct. Anal.} 19:1101--1118, 2009.
	
	\bibitem{lewinski2019}
	T. Lewi{\'{n}}ski, T. Sok{\'o}{\l}, C. Graczykowski:
	\newblock {\em Michell {S}tructures}.
	\newblock Springer International Publishing, Cham, 2019.
	
	\bibitem{michell1904}
	A.G.M. Michell:
	\newblock The limits of economy of material in frame-structures.
	\newblock {\em Phil. Mag. S.} 8:589--597, 1904.
	
	\bibitem{Moreau1966}
	J.-J. Moreau:
	\newblock Fonctionnelles convexes.
	\newblock {\em S{\'e}minaire Jean Leray} 2:1--108, 1966.
	
	\bibitem{muller2001stochastic}
	A. M\"{u}ller:
	\newblock Stochastic ordering of multivariate normal distributions.
	\newblock {\em Ann. Inst. Statist. Math.} 53:567--575, 2001.
	
	\bibitem{MulSca2006}
	A. M\"{u}ller, M. Scarsini:
	\newblock Stochastic order relations and lattices of probability measures.
	\newblock {\em SIAM J. Optim.} 16:1024--1043, 2006.
	

	
	\bibitem{olubummo2004}
	Y. Olubummo:
	\newblock On duality for a generalized Monge–Kantorovich problem.
	\newblock {\em J. Funct. Anal.} 207: 253--263, 2004.
	

	
	\bibitem{prager1977}
	W.~Prager, G.I.N. Rozvany:
	\newblock Optimal layout of grillages.
	\newblock {\em J. Struct. Mech.} 5:1--18, 1977.
	
	\bibitem{rozvany1972a}
	G.I.N. Rozvany:
	\newblock Grillages of maximum strength and maximum stiffness.
	\newblock {\em Int. J. Mech. Sci.} 14:651--666, 1972.
	
	\bibitem{santambrogio2015}
	F. Santambrogio:
	\newblock {\em Optimal transport for applied mathematicians}.
	\newblock 	Birk{\"a}user, New York, 2015.
	
	\bibitem{Smirnov}
	S.K. Smirnov:
	\newblock Decomposition of solenoidal vector charges into elementary solenoids,
	and the structure of normal one-dimensional flows.
	\newblock {\em Algebra i Analiz} 5:206--238, 1993.
	
	\bibitem{strassen1965}
	V. Strassen:
	\newblock The existence of probability measures with given marginals.
	\newblock {\em Ann. Math. Stat.} 36:423--439, 1965.
	
	
	\bibitem{villani2003}
	C. Villani:
	\newblock {\em Topics in optimal transportation}.
	\newblock American Mathematical Society, 2003.
	
	\bibitem{Wiesel2023}
	J. Wiesel, E. Zhang:
	\newblock An optimal transport-based characterization of convex order.
	\newblock {\em Depend. Model.}, 11: Article No. 20230102, 2023.
	
	\bibitem{zolotarev1977}
	V.M. Zolotarev:
	\newblock Ideal metrics in the problem of approximating the distributions of
	sums of independent random variables.
	\newblock {\em Theory Probab. Appl.}, 22:433--449, 1978.
	
\end{thebibliography}

\smallskip

{\footnotesize
\noindent
Karol Bo{\l}botowski:\\
Lagrange Mathematics and Computing Research Center\\
103 rue de Grenelle, Paris 75007 - FRANCE\\
and\\
Faculty of Mathematics, Informatics and Mechanics, University of Warsaw\\
2 Banacha Street, 02-097 Warsaw - POLAND\\
{\tt k.bolbotowski@mimuw.edu.pl}

\smallskip 
\noindent
	Guy Bouchitt\'e:\\
	Laboratoire IMATH, Universit\'e de Toulon\\
	BP 20132, 83957 La Garde Cedex - FRANCE\\
		{\tt bouchitte@univ-tln.fr}
}

\end{document}